\documentclass{amsart}
\usepackage{amssymb}
\usepackage{mathtools}
\usepackage{hyperref}
\usepackage{longtable}
\hypersetup{
  hidelinks,
  pdftitle={An Explicit Surjectivity Threshold for Digit Sums of Primes},
  pdfauthor={Jens Lehmann},
  pdfkeywords={digit sums of primes, circle method, exponential sums, explicit constants, OEIS A070027}
}
\title[An Explicit Surjectivity Threshold for Digit Sums of Primes]{An Explicit Surjectivity Threshold\\ for Digit Sums of Primes}

\author[J. Lehmann]{Jens Lehmann}
\address{Technische Universit\"at Dresden, Dresden, Germany}
\email{jens.lehmann@tu-dresden.de}
\thanks{The author is affiliated with TU Dresden and Amazon; the present work was done outside of Amazon.}
\thanks{Accompanying code: \url{https://github.com/JensLehmann/Prime-Digit-Sums}, release tag
\texttt{v1.0-mcom-submission}, commit
\texttt{9f9903a862b83ad3b1c2fbb961af1ecf7857d214}.}

\subjclass[2020]{Primary 11A63; Secondary 11N05, 11L20, 11Y60}
\keywords{Digit sums of primes, circle method, exponential sums, explicit constants, iterated digit-sum primes, OEIS A070027}

\date{June 2, 2026}

\newtheorem{theorem}{Theorem}
\newtheorem{lemma}{Lemma}
\newtheorem{proposition}{Proposition}
\newtheorem{corollary}{Corollary}
\theoremstyle{definition}
\newtheorem{definition}{Definition}
\newtheorem{remark}{Remark}

\newcommand{\Major}{\mathfrak{M}}
\newcommand{\Minor}{\mathfrak{m}}
\newcommand{\e}{\mathrm{e}}
\DeclareMathOperator{\erf}{erf}
\DeclareMathOperator{\erfc}{erfc}

\begin{document}

\begin{abstract}
Let $s(n)=s_{10}(n)$ be the decimal sum-of-digits map. Building on the
circle-method framework of Drmota--Mauduit--Rivat for digital restrictions on
primes, we make the constants explicit at the points needed to obtain an
effective surjectivity statement for digit sums of primes.  We exhibit an
explicit integer $M<1.78\times10^{32}$ such that every integer $m\ge M$ with
$\gcd(m,9)=1$ occurs as $s(p)$ for at least one prime~$p$.  We also prove an
explicit lower bound
\[
A_m(10^{2m/9})\ge C_{\mathrm q}(m)\,10^{2m/9}/m^{3/2},
\]
where $C_{\mathrm q}(m)$ is explicit, positive above the sufficient threshold,
and bounded away from $0$ along each admissible residue class.

Existence of a non-numerical threshold follows from the DMR asymptotic theory
and was noted by Harman; to the best of our knowledge, this is the first
published explicit numerical threshold for this surjectivity statement.  The
proof combines explicit major-arc estimates, a fully explicit replacement for
DMR's implicit prime exponential-sum input, and constant-tracked Type-II
minor-arc estimates.

As an application, we prove the infinitude of OEIS A070027, the primes whose
iterated digit-sum chain remains prime until reaching a one-digit prime.  We
also record related effective consequences for additive primes and digit-sum
additive decompositions.
\end{abstract}

\maketitle

\section{Introduction and main results}\label{sec:intro}

Let $s(n)=s_{10}(n)$ denote the sum of the base-$10$ digits of~$n$. A classical theme in analytic number
theory is that the primes are equidistributed with respect to many digital statistics, and in
particular that the digit sum of primes behaves, after normalization, like a Gaussian random variable.
The decisive breakthrough in this direction is the circle-method framework of
Drmota--Mauduit--Rivat~\cite{DMRcomp2}, which yields an asymptotic formula for primes with prescribed
digit sum in a broad central range.

A qualitative consequence, explicitly noted in the literature (for instance by Harman~\cite{harman2012counting}), is that for all
\emph{sufficiently large} integers $m$ satisfying the unavoidable congruence condition $\gcd(m,9)=1$,
there exists a prime $p$ with $s(p)=m$. The primary goal of the present paper is to make this statement
\emph{effective}, i.e.\ to produce a concrete threshold $M$ that is provably valid and to provide a proof
in which every constant can be audited and, in principle, improved.

For $x\ge 2$ and integer $m\ge 1$ define
\[
A_m(x):=\#\{p\le x:\ p\ \text{prime and}\ s(p)=m\}.
\]
Our main result is an explicit lower bound on $m$ ensuring $A_m(\infty)\ge 1$.

\begin{theorem}[Explicit surjectivity threshold, base $10$]\label{thm:main}
Set $q:=10$ and define
\[
\begin{aligned}
M&:=\left\lceil \frac{9}{2\log q}\,Y_\ast+\frac92\right\rceil,\\
Y_\ast&:=\max\!\bigl\{
Y_{41},Y_{42},Y_{43}^\ast,Y_{45},Y_{47},
Y_{\mathrm{maj}},Y_{\mathrm{min}},\log x_{\mathrm{AP}}\bigr\},
\end{aligned}
\]
where all thresholds are introduced in~\S\ref{sec:derivation_m}.  Specifically,
$Y_{43}^\ast$ comes from Proposition~\ref{prop:CF-replacement} via
Corollary~\ref{cor:Lstar-upper-bound}, while $Y_{\mathrm{maj}}$ and $Y_{\mathrm{min}}$
come from Proposition~\ref{prop:maj} and Lemma~\ref{lem:Cmin}.  The auxiliary thresholds
$Y_{45}$ and $Y_{41},Y_{42},Y_{47}$ are defined in
equations~\eqref{eq:Y43-Y45-def} and~\eqref{eq:Y41-Y42-Y47-def}, respectively.
Then for every integer $m\ge M$ with $\gcd(m,9)=1$ there exists a prime $p$ such that $s(p)=m$;
equivalently, $A_m(\infty)\ge 1$ for all such~$m$.

Quantitatively, at the working parameters $(q,\eta,\nu)=(10,0.0545,0.2859)$,
\[
M\ \le\ 177843590339381623423137292296355\ <\ 1.78\times 10^{32}.
\]
These working parameters are chosen to balance the three competing thresholds in the chain
$Y_\ast=\max(Y_{41},Y_{42},Y_{43}^\ast,Y_{45},Y_{47},Y_{\mathrm{maj}},Y_{\mathrm{min}},\log x_{\mathrm{AP}})$: at the
certificate level, $Y_{43}^\ast\le 9.10\times 10^{31}$ (Corollary~\ref{cor:Lstar-upper-bound}),
$Y_{\mathrm{min}}\le 8.00\times 10^{31}$, and $Y_{\mathrm{maj}}\le 7.90\times 10^{31}$
are all of the same order of magnitude; the supplementary script's bisection further reports
$Y_{43,\mathrm{bisect}}^\ast\approx 6.27\times 10^{31}$, at which $Y_{\mathrm{min}}$ becomes
binding numerically.  We choose
$c_{43}^\ast=\tfrac{1}{200}\cdot(c_4/32)\min\{1,(\nu-2\eta)/(\eta+1/2)\}$ in
Proposition~\ref{prop:CF-replacement}; the $0.5\%$ safety margin keeps the absorption threshold
$Y_{47}$ below the major/minor scale while preserving a convenient gap
$a=2\theta(\nu/2-\eta)-c_{43}^\ast$.  Any value $M'$ at
this magnitude (or larger) likewise satisfies the conclusion.

In fact (see Corollary~\ref{cor:quantitative}) we prove the quantitative refinement
$A_m(10^{2m/9})\ge C_{\mathrm{q}}(m)\,10^{2m/9}/m^{3/2}$, where $C_{\mathrm{q}}(m)$ is explicit
and has a positive limit along each admissible residue class modulo $9$.
\end{theorem}

\begin{remark}[What is new]\label{rem:new}
The novelty is not the existence of a threshold (which is implicit in the DMR asymptotic theory),
but rather the production of an effective one with the numerical inequalities made explicit.
Achieving this requires unrolling a chain of estimates with explicit constants, including a complete
replacement for an implicit constant occurring in the DMR major-arc analysis (their prime
exponential-sum input).  The working constants were selected by a computational search to balance
the dominant sufficient thresholds; no claim of global optimality is made.  The proof is structured
so that improved constants from future research can be substituted mechanically.

The size of $M$ is governed by three thresholds $Y_{43}^\ast,Y_{\mathrm{maj}},Y_{\mathrm{min}}$.
At these working parameters their empirical values equilibrate within a factor ${<}1.3$
(and their rigorous upper bounds within ${<}1.2$); the limiting
input within this Fourier framework is the extracted digit-discrepancy rate
$c_4=\log 10/6$ from Lemma~\ref{lem:c4-explicit}.  See Remark~\ref{rem:limits} for the explicit accounting and an
analysis of what would be required to move below this barrier.
\end{remark}

\begin{remark}[Limits of the present method]\label{rem:limits}
The bound $M<1.78\times10^{32}$ obtained here is governed in every binding step by the
extracted decimal digit-discrepancy rate $c_4=\log 10/6$ from
Lemma~\ref{lem:c4-explicit}, which appears explicitly in each of the three near-binding
thresholds $Y_{43}^\ast$ (CF replacement), $Y_{\mathrm{min}}$ (minor arc), and $Y_{\mathrm{maj}}$
(major arc).  Reducing $M$ substantially below the present $10^{32}$ scale within this framework
requires either a sharper rate $c_4'$ (the present framework's $M$ scales roughly as
$\exp(\mathrm{const}/c_4')$ after re-optimizing the auxiliary parameters) or a non-Fourier route
that bypasses the digit-discrepancy machinery.

The alternatives considered are Berry--Esseen smoothing back to a CF bound and
bandlimited majorants replacing the Fourier integral.  Both either degrade the rate or
yield only a bounded multiplicative gain in $M$, without overcoming the $c_4$-driven barrier; see
Appendix~\ref{sec:limits-detailed}.

The modular construction in this paper is written so that any sharper $c_4$ (or any of its
derivatives $c_1$, $K_0$, $C_{\mathrm{maj}}$) feeds directly into a smaller $M$ without
re-deriving the framework.
\end{remark}

\medskip

A second motivation is the study of primes whose \emph{iterated} digit sum remains prime.
Write $s^{(0)}(n)=n$ and $s^{(k+1)}(n)=s(s^{(k)}(n))$. Since iterating $s$ strictly decreases $n$ once $n\ge 10$,
the chain $n, s(n), s^{(2)}(n),\ldots$ eventually reaches a one-digit fixed point in $\{0,1,\dots,9\}$.

\begin{definition}\label{def:idsprime}
A prime $p$ is called an \emph{iterated digit-sum prime} if there exists an index $J$ such that
$s^{(J)}(p)\in\{2,3,5,7\}$, $J$ is the first index at which the digit-sum chain reaches a one-digit
prime, and every iterate $s^{(j)}(p)$ is prime for $0\le j\le J$.
Equivalently, all nontrivial iterates are prime until the chain reaches a one-digit prime.
\end{definition}

\begin{theorem}[Infinitude of A070027]\label{thm:a070027}
There exist infinitely many iterated digit-sum primes in the sense of Definition~\ref{def:idsprime}.
In particular, the OEIS sequence A070027 is infinite.  A quantitative refinement
(Corollary~\ref{cor:quantitative-A070027}) gives an explicit lower bound on the number of such
primes above an explicit scale.
\end{theorem}

\begin{remark}[Mechanism of the proof]\label{rem:lift}
Theorem~\ref{thm:a070027} is obtained by a lifting argument.
Once one explicit ``seed'' prime $p_0\ge M$ is exhibited for which the chain
$p_0, s(p_0), s^{(2)}(p_0),\ldots$ consists of primes down to a one-digit prime, Theorem~\ref{thm:main}
guarantees the existence of a prime $p_1$ with $s(p_1)=p_0$.
Then $p_1$ inherits the same terminal prime chain, and iterating this lifting produces infinitely many
distinct primes with the desired property.
\end{remark}

\medskip

\noindent\textbf{Organization.}
Section~\ref{sec:notation} fixes notation and records basic congruences.
Section~\ref{sec:inputs} restates the DMR circle-method inputs in a form suited for explicit constant
tracking and isolates the prime exponential-sum input requiring a fully explicit substitute.
Sections~\ref{sec:maj},~\ref{sec:fourier}, and~\ref{sec:derivation_m} assemble these inputs to prove
Theorem~\ref{thm:main}; \S\ref{sec:derivation_m} also derives the quantitative refinement
(Corollary~\ref{cor:quantitative}).
Section~\ref{sec:infinite_prime_digit_sums} derives Theorem~\ref{thm:a070027} (and its quantitative form,
Corollary~\ref{cor:quantitative-A070027}) by the lifting argument sketched above.
Section~\ref{sec:further_applications} develops further consequences: an effective upper bound on the
smallest prime with a given digit sum (\S\ref{subsec:smallest-prime-bound}), an effective infinitude of
additive primes (\S\ref{subsec:additive-primes}), a stratification of the iterated digit-sum primes by
their terminal one-digit prime (\S\ref{subsec:terminal-stratification}), a digit-sum Goldbach-type
result (\S\ref{subsec:digit-sum-goldbach}), and a lower bound on the largest prime with prescribed
digit sum (\S\ref{subsec:largest-prime}).

\section{Notation and basic facts}\label{sec:notation}

We write $e(t):=\exp(2\pi i t)$ and $\log$ denotes the natural logarithm. For a real number $\alpha$, let
$\|\alpha\|_{\mathbb R/\mathbb Z}:=\min_{n\in \mathbb Z}|\alpha-n|$ denote the distance to the nearest integer.
We write $\gcd(a,b)$ for the greatest common divisor of integers $a,b$.

We recall the standard congruence obstruction for digit sums.

\begin{lemma}[Digit sum modulo $9$]\label{lem:mod9}
For all integers $n \ge 0$, one has $s(n)\equiv n\pmod 9$.
\end{lemma}

\begin{proof}
Write the base-$10$ expansion $n=\sum_{j=0}^{D-1} d_j 10^j$ with digits
$d_j\in\{0,1,\dots,9\}$. Since $10\equiv 1\pmod 9$, we have $10^j\equiv 1\pmod 9$
for all $j$, and hence
\[
n \equiv \sum_{j=0}^{D-1} d_j = s(n)\pmod 9.\qedhere
\]
\end{proof}

\begin{remark}[The admissibility condition $\gcd(m,9)=1$]\label{rem:admissible}
If $p>3$ is prime then $p\not\equiv 0\pmod 3$, hence $s(p)\not\equiv 0\pmod 3$ by
Lemma~\ref{lem:mod9}. Since $9$ has only the prime factor $3$, this is equivalent to
$\gcd(s(p),9)=1$ for all primes $p>3$. Thus the condition $\gcd(m,9)=1$ in
Theorem~\ref{thm:main} is necessary (up to the isolated prime $p=3$).
\end{remark}

\subsection*{Summary of key constants}

For reader and referee convenience we collect the principal explicit constants of the proof and
indicate where each is introduced and which threshold it feeds.  Values are at the working parameters
$(q,\eta,\nu)=(10,0.0545,0.2859)$.

\begin{center}
\small
\renewcommand{\arraystretch}{1.15}
\begin{tabular}{@{}p{0.19\linewidth}p{0.24\linewidth}p{0.28\linewidth}p{0.20\linewidth}@{}}
Constant & Value & Defined in & Feeds into \\
\hline
$c_4$ & $\log 10/6 \approx 0.384$ & Lemma~\ref{lem:c4-explicit} & $c_{43}^\ast$, $\theta$ \\
$c_{43}^\ast$ & $\approx 1.91\times 10^{-5}$ ($0.5\%$ of max) & Prop.~\ref{prop:CF-replacement} & $Y_{43}^\ast$, $Y_{47}$, hence $M$ \\
$\theta$ & $c_4/(16(\eta+1/2))\approx 0.04326$ & Prop.~\ref{prop:CF-replacement} & $D(L)$, $L_\ast^{(i)}$ \\
$C_{\mathrm{DMR}}$ & $102$ & Lemma~\ref{lem:DMR43} & $Y_{45}$, $x_{\mathrm{maj}}$ \\
$K_{10}$ & $\le 10^6$ (rigorous; envelope $10^{9}$) & Lemma~\ref{lem:S2-input}, Step~1 & $C_{\mathrm{II}}$, $C_{\mathrm{min}}$ \\
$C_{\tau}$ & $10^3$ & Lemma~\ref{lem:tau-input} & carry propagation in Type II \\
$C_{\mathrm{II}}$ & $<64000$ & Prop.~\ref{prop:typeII-explicit} & $C_{\mathrm{min}}$, $Y_{\mathrm{min}}$ \\
$c_1$ & $0.00150628870\ldots$ & Lemma~\ref{lem:c1} & $Y_{\mathrm{min}}$ \\
$C_{\mathrm{min}}$ & $4\times 10^{12}$ & Lemma~\ref{lem:Cmin} & $Y_{\mathrm{min}}$ \\
$C_{\mathrm{maj}}$ & $510$ (proof gives $<461$) & Cor.~\ref{cor:maj-numeric} & major-arc error bound \\
$C_{\mathrm{q}}(m)$ & $>0.108$ for $m\ge10^{100}$ & Lemma~\ref{lem:Cq-finite} & quantitative refinement \\
$M$ & $<1.78\times10^{32}$ & Theorem~\ref{thm:main} & ---
\end{tabular}
\end{center}

The value of $M$ is governed jointly by $Y_{43}^\ast$, $Y_{\mathrm{maj}}$, and $Y_{\mathrm{min}}$
(all near-binding at the optimized parameters; see Remark~\ref{rem:new}, Remark~\ref{rem:limits},
and \S\ref{sec:derivation_m}).

\subsection*{Dependency map for Theorem~\ref{thm:main}}

The following map gives the proof dependencies at the level used by the final numerical certificate.
The detailed source-equation audit for imported estimates is deferred to
Appendix~\ref{app:import-audit}.

\begin{center}
\small
\renewcommand{\arraystretch}{1.18}
\begin{tabular}{@{}p{0.43\linewidth}p{0.08\linewidth}p{0.40\linewidth}@{}}
IK/Vaughan prime exponential sums & $\Longrightarrow$ & Lemma~\ref{lem:DMR43} ($C_{\mathrm{DMR}}=102$) \\
Lemma~\ref{lem:DMR43} with Fej\'er smoothing & $\Longrightarrow$ &
Lemma~\ref{lem:c4-explicit}, Proposition~\ref{prop:CF-replacement}, and $Y_{43}^{\ast}$ \\
DMR/MR Type-I--II inputs & $\Longrightarrow$ &
Lemmas~\ref{lem:DMR31-self}, \ref{lem:DMR32-self}, \ref{lem:typeI-explicit},
and~\ref{lem:S2-input} \\
Type-I--II transfer to primes & $\Longrightarrow$ & Lemma~\ref{lem:Cmin} and $Y_{\mathrm{min}}$ \\
Major-arc Gaussian approximation & $\Longrightarrow$ & Proposition~\ref{prop:maj}, Corollary~\ref{cor:maj-numeric}, and $Y_{\mathrm{maj}}$ \\
$Y_{41},Y_{42},Y_{43}^{\ast},Y_{45},Y_{47},Y_{\mathrm{maj}},Y_{\mathrm{min}},x_{\mathrm{AP}}$ & $\Longrightarrow$ &
$Y_{\ast}$ and the explicit threshold $M$ \\
Theorem~\ref{thm:main} with digit-sum lifting & $\Longrightarrow$ &
Theorem~\ref{thm:a070027} and the applications in \S\ref{sec:further_applications}
\end{tabular}
\end{center}

\section{The DMR framework and explicit inputs}\label{sec:inputs}

Our proof follows the circle-method strategy developed by Drmota--Mauduit--Rivat~\cite{DMRcomp2}
for primes subject to digital restrictions.
Very briefly, one expresses $A_m(x)$ by Fourier inversion in the digit-sum variable and studies
exponential sums of the shape
\[
\sum_{p\le x} e\!\big(\alpha\, s(p)\big),
\]
splitting the integral over $\alpha\in\mathbb R/\mathbb Z$ into \emph{major arcs} (where $\alpha$ is
well approximated by rationals with small denominator) and \emph{minor arcs} (the complementary set).

In this framework, the minor-arc analysis supplies uniform cancellation, while the major arcs yield a
Gaussian main term reflecting the central-limit behavior of $s(p)$. For the qualitative theory, it is
enough that the relevant implied constants are finite. To make Theorem~\ref{thm:main} effective, however,
we must track constants through the major-arc approximation and choose parameters so that the total error
is at most a fixed fraction of the main term (in our explicit choice, at most $1/4$ of the main term).

Several steps in~\cite{DMRcomp2} are stated with implied constants that are not made numerical.
For effectivity, the non-numerical inputs needed by the proof are used only through explicit
numerical forms: most importantly the prime exponential-sum estimate used on the major arcs
(DMR Lemma~4.3), and the MR10 Type-I/Type-II estimates used on the minor arcs.  The following
table records exactly where each imported statement is used and where its constant is certified.
The next lemma is the explicit substitute for the relevant DMR prime exponential sum estimate.

\paragraph{Imported vs.\ re-proved results.}
To delimit the new contribution cleanly, the following table lists each external result we
invoke, the explicit form used in this paper, and the location where it enters the proof.

\begingroup
\renewcommand{\arraystretch}{1.20}
\setlength{\tabcolsep}{4pt}
\small
\begin{longtable}{p{0.35\linewidth} p{0.27\linewidth} p{0.30\linewidth}}
\textbf{External result} & \textbf{Explicit form used here} & \textbf{Where it enters / certified by} \\
\hline
\endfirsthead
\textbf{External result} & \textbf{Explicit form used here} & \textbf{Where it enters / certified by} \\
\hline
\endhead
DMR Lemma~3.1 (Vaughan-type combinatorial identity)~\cite[(10)--(11)]{DMRcomp2}
& Lemma~\ref{lem:DMR31-self}; transfer constant $20$ at $q=10$
& Proof of Lemma~\ref{lem:Cmin-local}, Step~2. \\
DMR Lemma~3.2 (Type-II reduction)~\cite[(12)--(14)]{DMRcomp2}
& Lemma~\ref{lem:DMR32-self}; prefactor $6$
& Proof of Lemma~\ref{lem:Cmin-local}, Step~1. \\
DMR Proposition~3.1 (Type-I estimate)~\cite[(15)--(16)]{DMRcomp2}
& Lemma~\ref{lem:typeI-explicit}; constant $10^8$ at $q=10$
& One hypothesis in Lemma~\ref{lem:DMR31-self}, used in Lemma~\ref{lem:Cmin-local}. \\
DMR Lemma~3.3 (vdC)~\cite[(18)--(22)]{DMRcomp2}
& Lemma~\ref{lem:MR-vdc-self}; prefactor tracked in Prop.~\ref{prop:typeII-explicit}
& Proof of Prop.~\ref{prop:typeII-explicit}. \\
DMR Lemma~3.4 (carry-propagation count)~\cite[(20)--(21)]{DMRcomp2}
& $E\le C_\tau(\mu+\nu)(\log q)q^{\mu+\nu-\rho}$
& Lemma~\ref{lem:tau-input}; entering Prop.~\ref{prop:typeII-explicit}. \\
BBR explicit divisor-problem remainder~\cite[Thm.~1.2]{BerkaneBordellesRamare2012}
& $\sum_{X-Y<n\le X}\tau(n)\le C_\tau Y\log X$
& Lemma~\ref{lem:tau-input}. \\
DMR Lemma~4.3 (prime exponential sum at $a/Q$)~\cite[Lemma 4.3]{DMRcomp2}
& $C_{\mathrm{DMR}}=102$
& Lemma~\ref{lem:DMR43}; via IK Thm.~13.6. \\
MR10 Lemma~4 (vdC)~\cite[Lemme 4]{MR10}
& Lemma~\ref{lem:MR-vdc-self}
& Proof of Lemma~\ref{lem:S2-input}. \\
MR10 Lemma~6 (sum of $1/\sin$)~\cite[(26)]{MR10}
& Lemma~\ref{lem:MR-sin-sum-self}; coefficients $(1,1,2/\pi)$
& Proof of Lemma~\ref{lem:S2-input}, Step~1, especially~\eqref{eq:K10-trig}. \\
MR10 Lemma~20 ($\gamma_q\in[1/2,1)$)~\cite[(105)]{MR10}
& Lemma~\ref{lem:MR-gamma-self}
& Proof of Lemma~\ref{lem:S2-input}, Step~2. \\
MR10 Lemma~21 ($F_\lambda$ Fourier bound)~\cite[(106)]{MR10}
& Lemma~\ref{lem:MR-F-weighted-self}; prefactor $1/\sin(\pi/q)$
& Proof of Lemma~\ref{lem:S2-input}, Step~2. \\
MR10 eq.~(64) ($S_2$ structural estimate)~\cite[(64)]{MR10}
& Explicit structural reduction, with $K_{10}\le 10^6$
& Lemmas~\ref{lem:MR-S2-G-self}, \ref{lem:MR-struct-self}, and~\ref{lem:S2-input}. \\
IK Theorem~13.6 (effective IK bound)~\cite[Thm.~13.6]{IK04}
& Explicit form in Lemma~\ref{lem:IK136-explicit}
& Lemma~\ref{lem:DMR43} (via Lemmas~\ref{lem:IK1345-explicit}--\ref{lem:IK1348-explicit}). \\
Kuipers--Niederreiter Erd\H{o}s--Tur\'an inequality~\cite[Thm.~2.5]{KN74}
& $(6, 4/\pi)$ for interval discrepancy
& Lemma~\ref{lem:c4-explicit}, Step~0 (Erd\H{o}s--Tur\'an step). \\
Bennett--Martin--O'Bryant--Rechnitzer~\cite{bennett2018explicit}
& $\pi(x;9,m)\ge x/(6\log x)$ for $\gcd(m,9)=1$, $x\ge x_{\mathrm{AP}}$
& Lemma~\ref{lem:pi9-lb}. \\
Marchal--Arbel (strict sub-Gaussianity)~\cite{MarchalArbel}
& Analog only; Lemma~\ref{lem:uniform-subgaussian} is proved directly
& Lemma~\ref{lem:uniform-subgaussian} (proved here directly). \\
\end{longtable}
\endgroup

\noindent
The MR10 $S_2$ input is used through the explicit reduction
Lemmas~\ref{lem:MR-S2-G-self}, \ref{lem:MR-G-first-self}, \ref{lem:MR-struct-self}, and
\ref{lem:S2-input}; Lemma~\ref{lem:insensitivity} records the remaining robustness margin.

\smallskip
The qualitative threshold $x_0(\alpha)$ in~\cite[Lemma~3.2]{DMRcomp2} is alpha-dependent and not
uniformly bounded as $(q-1)\alpha$ approaches an integer.  Lemma~\ref{lem:minor-bridge} isolates
the direct bridge below $x_0(\alpha)$, Lemma~\ref{lem:Cmin-local} applies the Type-I/Type-II
machinery only for $x\ge x_0(\alpha)$, and Lemma~\ref{lem:Cmin} performs the residue-class
projection.  Thus no uniform numerical value of $x_0$ enters the final threshold.  The cumulative
implicit constant $K_{10}$ in MR10's S$_2$
derivation enters only via the dominated $C_{\mathrm{min}}$ pathway; by
Lemma~\ref{lem:insensitivity} this is robust to any $K_{10}\le 10^{9}$.

\paragraph{Proof roadmap.}
The rest of the proof is organized as follows.  Section~\ref{sec:inputs} certifies the analytic
inputs: first the rational-prime exponential sum needed on major arcs, then the minor-arc chain
Type-I/Type-II $\Rightarrow$ primes in residue classes.  Section~\ref{sec:maj} proves the
major-arc Gaussian approximation at integral decimal lengths.  Section~\ref{sec:fourier}
combines major and minor arcs into a positivity criterion for $A_m(10^L)$.
Section~\ref{sec:derivation_m} chooses the integer length $L=\lfloor 2m/9\rfloor$, checks all
threshold inequalities, and converts them into the stated bound $M<1.78\times10^{32}$.  The remaining
sections derive applications from Theorem~\ref{thm:main}; they do not feed back into the proof of
the threshold.

\paragraph{Auxiliary thresholds.}  The thresholds $x_{41}, x_{42}, x_{45}, x_{47},
x_{\mathrm{maj}}, x_{\mathrm{min}}$ (and the corresponding $Y$-thresholds) are defined below
by explicit sufficient inequalities rather than by closed-form constants.  The dominated ones
($Y_{41},Y_{42},Y_{45},Y_{47}$) are recorded with unoptimized constructions without affecting
$M$.  Lemma~\ref{lem:insensitivity} formalizes the insensitivity to the minor-arc cumulative
constant $K_{10}$ up to $10^{9}$.

\subsection{An explicit prime exponential sum for rational phases (DMR Lemma 4.3)}\label{subsec:explicit-lemma43}

The proof of Theorem~\ref{thm:main} requires an \emph{effective} major-arc approximation in the
Drmota--Mauduit--Rivat framework. The relevant major-arc prime exponential-sum input in
\cite{DMRcomp2} (Lemma~4.3) is stated with an implicit constant, and it ultimately controls the
major-arc error terms.
In this subsection we give a complete proof of an explicit, constant-tracked substitute for that result,
following the method of~\cite[Ch.~13]{IK04} but making every numerical dependence fully effective.

We begin with a standard bound for finite geometric progressions, which will be used repeatedly to control
short additive characters.

\begin{lemma}[Geometric sum bound]\label{lem:geom-sum}
For any real $\beta$ and any integer $N\ge 1$,
\[
\left|\sum_{n=1}^{N} e(\beta n)\right|
\le
\min\!\left(N,\ \frac{1}{2\|\beta\|_{\mathbb R/\mathbb Z}}\right),
\]
with the convention that $1/(2\|\beta\|_{\mathbb R/\mathbb Z})=+\infty$ if $\beta\in\mathbb Z$.
\end{lemma}

\begin{proof}
The triangle inequality gives $\left|\sum_{n=1}^{N} e(\beta n)\right|\le N$.

If $\beta\in\mathbb Z$ then $\sum_{n=1}^{N} e(\beta n)=N$, so the claimed bound holds.
Assume henceforth that $\beta\notin\mathbb Z$. Using the geometric-series identity and the
standard sine form, we have
\[
\sum_{n=1}^{N} e(\beta n)
=
e\!\left(\frac{\beta(N+1)}{2}\right)\frac{\sin(\pi N\beta)}{\sin(\pi\beta)}.
\]
Therefore
\[
\left|\sum_{n=1}^{N} e(\beta n)\right|
\le \frac{1}{|\sin(\pi\beta)|}.
\]
Let $\delta:=\|\beta\|_{\mathbb R/\mathbb Z}\in(0,1/2]$. Since $|\sin(\pi\beta)|=\sin(\pi\delta)$
and $\sin$ is concave on $[0,\pi/2]$, the graph of $\sin$ lies above the chord joining
$(0,0)$ and $(\pi/2,1)$, hence $\sin(\pi\delta)\ge 2\delta$. It follows that
\[
\left|\sum_{n=1}^{N} e(\beta n)\right|\le \frac{1}{2\delta}
=\frac{1}{2\|\beta\|_{\mathbb R/\mathbb Z}}.
\]
Combining with the bound $\le N$ proves the lemma.
\end{proof}

\medskip

The next two lemmas are explicit versions of the estimates (13.45)--(13.46) in
\cite[Lemma~13.7]{IK04}. We only need them for $\alpha=a/Q$ rational with $(a,Q)=1$.

\begin{lemma}[Explicit (13.45)]\label{lem:IK1345-explicit}
Let $\alpha=a/Q$ with integers $Q\ge 2$ and $(a,Q)=1$. Let $M,N\ge 1$ be integers.
Then
\[
\sum_{|m|\le M}
\left|
\sum_{1\le n\le N} e(\alpha m n)
\right|
\le
2\bigl(M+N+MN/Q+Q\bigr)\,\log(2Q).
\]
\end{lemma}

\begin{proof}
Write $\|\cdot\|:=\|\cdot\|_{\mathbb R/\mathbb Z}$. By Lemma~\ref{lem:geom-sum},
\[
\left|\sum_{1\le n\le N} e(\alpha m n)\right|
\le \min\!\left(N,\ \frac{1}{2\|\alpha m\|}\right).
\]
Thus it suffices to bound
\[
S:=\sum_{|m|\le M}\min\!\left(N,\ \frac{1}{2\|\alpha m\|}\right).
\]

Partition the finite interval $\{-M,-M+1,\dots,M\}$ into consecutive blocks of length $Q$, starting
at $-M$ and allowing the last block to be truncated.  The number of such blocks satisfies
\[
B\le \left\lceil\frac{2M+1}{Q}\right\rceil \le 1+\frac{2M}{Q}.
\]

Fix a complete block. Since $(a,Q)=1$, the residues $am\bmod Q$ run through a complete
system modulo $Q$ as $m$ runs through the block, hence the multiset of distances
$\|\alpha m\|=\|am/Q\|$ agrees with $\{\|r/Q\|: r=0,1,\dots,Q-1\}$.
For $r\neq 0$, $\|r/Q\|=\min(r,Q-r)/Q$, so
\[
\frac{1}{2\|r/Q\|}=\frac{Q}{2\min(r,Q-r)}.
\]
Therefore, for a complete block,
\begin{align*}
\sum_{m\ \text{in block}} \min\!\left(N,\ \frac{1}{2\|\alpha m\|}\right)
&\le
N + \sum_{r=1}^{Q-1}\frac{Q}{2\min(r,Q-r)} \\
&\le
N + Q\sum_{u=1}^{\lfloor Q/2\rfloor}\frac{1}{u}
=
N + QH_{\lfloor Q/2\rfloor},
\end{align*}
since each value $u=\min(r,Q-r)$ occurs at most twice. Using
$H_{\lfloor Q/2\rfloor}\le 1+\log(Q/2)\le \log(2Q)$, this is at most $N+Q\log(2Q)$.
The same bound holds for truncated blocks.

Hence
\[
S\le B\,(N+Q\log(2Q))
\le \left(1+\frac{2M}{Q}\right)(N+Q\log(2Q)).
\]
Let $L:=\log(2Q)\ge 1$. Expanding and using $L\ge 1$ gives
\[
\left(1+\frac{2M}{Q}\right)(N+QL)
= N+QL+\frac{2MN}{Q}+2ML
\le 2L\bigl(M+N+MN/Q+Q\bigr),
\]
which proves the lemma.
\end{proof}

\begin{lemma}[Explicit (13.46)]\label{lem:IK1346-explicit}
Let $\alpha=a/Q$ with integers $Q\ge 2$ and $(a,Q)=1$. Let $M\ge 1$ and $x\ge 2$.
Then
\[
\sum_{1\le m\le M}\left|\sum_{\substack{n\ge 1\\ mn\le x}} e(\alpha mn)\right|
\le
\bigl(M+x/Q+Q\bigr)\,\log(2Qx).
\]
\end{lemma}

\begin{proof}
Let $M_0:=\min(M,\lfloor x\rfloor)$. For $m>x$ the inner sum is empty, so it suffices
to sum over $1\le m\le M_0$. For each such $m$ set $N_m:=\lfloor x/m\rfloor\ge 1$, so that
\[
\sum_{\substack{n\ge 1\\ mn\le x}} e(\alpha mn)=\sum_{1\le n\le N_m} e(\alpha mn).
\]
By Lemma~\ref{lem:geom-sum},
\[
\left|\sum_{1\le n\le N_m} e(\alpha mn)\right|
\le \min\!\left(N_m,\ \frac{1}{2\|\alpha m\|_{\mathbb R/\mathbb Z}}\right)
\le \min\!\left(\frac{x}{m},\ \frac{1}{2\|\alpha m\|_{\mathbb R/\mathbb Z}}\right),
\]
where the second term is interpreted as $+\infty$ when $\|\alpha m\|_{\mathbb R/\mathbb Z}=0$.
Hence it suffices to bound
\[
T:=\sum_{1\le m\le M_0}\min\!\left(\frac{x}{m},\ \frac{1}{2\|\alpha m\|_{\mathbb R/\mathbb Z}}\right).
\]

We split $m$ into multiples of $Q$ and non-multiples.

\medskip\noindent
\emph{(i) Multiples of $Q$.}
Write $m=Qr$. Then $\alpha m=ar\in\mathbb Z$, so the inner exponential sum has size
$N_m\le x/m$. Therefore
\[
\sum_{\substack{m\le M_0\\ Q\mid m}}\min\!\left(\frac{x}{m},\ \frac{1}{2\|\alpha m\|}\right)
\le \sum_{r\le M_0/Q}\frac{x}{Qr}
= \frac{x}{Q}H_{\lfloor M_0/Q\rfloor},
\]
where $H_k=\sum_{1\le r\le k}1/r$ and we set $H_0:=0$.
For $k\ge 1$ we have $H_k\le 1+\log k$, hence
\[
H_{\lfloor M_0/Q\rfloor}\le 1+\log(x/Q)=\log(e x/Q)\le \log(2Qx),
\]
since $Q\ge 2$ implies $e/Q\le 2Q$. Thus the contribution of $Q\mid m$ is
$\le (x/Q)\log(2Qx)$.

\medskip\noindent
\emph{(ii) Non-multiples of $Q$.}
Partition $\{1,\dots,M_0\}$ into consecutive blocks of length $Q$.
In each complete block $\{jQ+1,\dots,(j+1)Q\}$, the residues $am\bmod Q$ run through a complete
system modulo $Q$, so for $Q\nmid m$ the multiset of distances
$\{\|\alpha m\|_{\mathbb R/\mathbb Z}:m\ \text{in block},\,Q\nmid m\}$
equals $\{\|r/Q\|_{\mathbb R/\mathbb Z}:1\le r\le Q-1\}$.
Therefore, on any block (complete or not),
\begin{align*}
\sum_{\substack{m\ \text{in block}\\ Q\nmid m}}
\min\!\left(\frac{x}{m},\ \frac{1}{2\|\alpha m\|_{\mathbb R/\mathbb Z}}\right)
&\le
\sum_{r=1}^{Q-1}\frac{1}{2\|r/Q\|_{\mathbb R/\mathbb Z}} \\
&=
\sum_{r=1}^{Q-1}\frac{Q}{2\min(r,Q-r)}
\le Q\sum_{u=1}^{\lfloor Q/2\rfloor}\frac{1}{u} \\
&=QH_{\lfloor Q/2\rfloor}
\le Q\log(2Q)
\le Q\log(2Qx),
\end{align*}
since $x\ge 2$.
There are at most $\lceil M_0/Q\rceil\le 1+M_0/Q$ blocks, so the total contribution from
$Q\nmid m$ is at most $(M_0+Q)\log(2Qx)\le (M+Q)\log(2Qx)$.

Combining (i) and (ii) gives $T\le (M+x/Q+Q)\log(2Qx)$, proving the lemma.
\end{proof}

\medskip

Next is an explicit version of the bilinear estimate (13.48) from~\cite[Lemma~13.8]{IK04}.

\begin{lemma}[Explicit (13.48)]\label{lem:IK1348-explicit}
Let $\alpha=a/Q$ with integers $2\le Q\le x$ and $(a,Q)=1$. Let $x\ge 3$ and $M,N\ge 1$.
For any complex sequences $(u_m),(v_n)$ with $|u_m|\le 1$ and $|v_n|\le 1$,
\[
\left|
\sum_{\substack{mn\le x\\ m>M,\ n>N}}
u_m v_n\,e(\alpha mn)
\right|
\le
10\,
\Bigl(\frac{x}{M}+\frac{x}{N}+\frac{x}{Q}+Q\Bigr)^{1/2}
x^{1/2}(\log x)^2.
\]
\end{lemma}

\begin{proof}
We follow the standard reduction in~\cite[(13.47)--(13.48)]{IK04}.
Decompose $n>N$ dyadically into intervals $(N_j,2N_j]$ with $N_1=N$ and $N_j=2^{j-1}N$.
We may stop once $N_j>x$, and if $N_j>x/M$ then the constraint $mn\le x$ with $m>M$ forces
$T_j=0$. Hence the number of relevant intervals satisfies
\[
J\le 1+\log_2 x\le \log x+\log_2 x\le \Bigl(1+\frac{1}{\log 2}\Bigr)\log x<\frac52\log x
\qquad (x\ge 3),
\]
using $\log x\ge 1$ for $x\ge 3$ and $1/\log 2<3/2$.

For each $j$ define
\[
T_j:=\sum_{\substack{N_j<n\le 2N_j\\ mn\le x,\ m>M}}
u_m v_n\,e(\alpha mn).
\]
By the triangle inequality,
\[
\left|\sum_{\substack{mn\le x\\ m>M,\ n>N}}u_m v_n e(\alpha mn)\right|
\le \sum_{j=1}^J |T_j|.
\]

Fix $j$ and define $\gamma_m:=u_m$ for $m>M$ and $\gamma_m:=0$ otherwise, so $|\gamma_m|\le 1$.
Then
\[
|T_j|
\le \sum_{N_j<n\le 2N_j}\left|\sum_{m\le x/n}\gamma_m e(\alpha mn)\right|
=:B(x;N_j).
\]
By Cauchy--Schwarz,
\[
B(x;N_j)^2
\le (2N_j)\sum_{N_j<n\le 2N_j}\left|\sum_{m\le x/n}\gamma_m e(\alpha mn)\right|^2.
\]
Expanding the square and swapping summations gives
\begin{align*}
\sum_{N_j<n\le 2N_j}\left|\sum_{m\le x/n}\gamma_m e(\alpha mn)\right|^2
&=
\sum_{m_1,m_2\ge 1}\gamma_{m_1}\overline{\gamma_{m_2}}
\sum_{\substack{N_j<n\le 2N_j\\ n\le x/\max(m_1,m_2)}} e\!\big(\alpha n(m_1-m_2)\big).
\end{align*}
Since $|\gamma_{m_1}\overline{\gamma_{m_2}}|\le 1$ and $m_1,m_2\le x/N_j$ whenever the inner sum is nonempty,
the triangle inequality and the shifted-interval form of Lemma~\ref{lem:geom-sum} (obtained by
factoring out the initial phase) imply
\[
\sum_{N_j<n\le 2N_j}\left|\sum_{m\le x/n}\gamma_m e(\alpha mn)\right|^2
\le
\sum_{m_1,m_2\le x/N_j}\min\!\left(2N_j,\ \frac{1}{2\|\alpha(m_1-m_2)\|}\right).
\]
Writing $d=m_1-m_2$ and observing that for each fixed $d$ the number of pairs $(m_1,m_2)$ with
$m_1,m_2\le x/N_j$ and $m_1-m_2=d$ is at most $\lfloor x/N_j\rfloor$, we obtain
\[
\sum_{N_j<n\le 2N_j}\left|\sum_{m\le x/n}\gamma_m e(\alpha mn)\right|^2
\le
\frac{x}{N_j}\sum_{|d|\le x/N_j}\min\!\left(2N_j,\ \frac{1}{2\|\alpha d\|}\right).
\]
Therefore
\[
B(x;N_j)^2
\le
2x\sum_{|d|\le x/N_j}\min\!\left(2N_j,\ \frac{1}{2\|\alpha d\|}\right).
\]

The proof of Lemma~\ref{lem:IK1345-explicit} bounds these min-sums, and with $U=\lfloor x/N_j\rfloor$,
$V=2N_j$ it gives
\[
\sum_{|d|\le x/N_j}\min\!\left(2N_j,\ \frac{1}{2\|\alpha d\|}\right)
\le
2\left(\frac{x}{N_j}+2N_j+\frac{2x}{Q}+Q\right)\log(2Q).
\]
Hence
\[
B(x;N_j)^2
\le
4x\left(\frac{x}{N_j}+2N_j+\frac{2x}{Q}+Q\right)\log(2Q).
\]
Since $Q\le x$, we have $\log(2Q)\le \log(2x)\le 2\log x$ for $x\ge 3$, so
\[
B(x;N_j)^2\le
8x\left(\frac{x}{N_j}+2N_j+\frac{2x}{Q}+Q\right)\log x.
\]
Taking square roots and using $\sqrt{\log x}\le \log x$ (valid for $x\ge 3$) yields
\[
B(x;N_j)\le
2\sqrt2\,x^{1/2}\left(\frac{x}{N_j}+2N_j+\frac{2x}{Q}+Q\right)^{1/2}\log x.
\]

Moreover, if $T_j\neq 0$ then $N\le N_j\le x/M$, so
\[
\frac{x}{N_j}\le \frac{x}{N},\qquad N_j\le \frac{x}{M},
\]
and therefore
\[
\frac{x}{N_j}+2N_j+\frac{2x}{Q}+Q
\le
\frac{x}{N}+2\frac{x}{M}+\frac{2x}{Q}+Q
\le
2\left(\frac{x}{M}+\frac{x}{N}+\frac{x}{Q}+Q\right).
\]
Thus
\[
B(x;N_j)\le
4\left(\frac{x}{M}+\frac{x}{N}+\frac{x}{Q}+Q\right)^{1/2}x^{1/2}\log x.
\]
Finally, summing over $j$ and using $J<\frac52\log x$ gives the claimed bound.
\end{proof}

\medskip

We now derive an explicit-constant version of~\cite[Theorem~13.6]{IK04} in the special (rational) case $\alpha=A/Q$.

\begin{lemma}[Explicit IK Theorem~13.6 for $\alpha=A/Q$]\label{lem:IK136-explicit}
Let $\alpha=A/Q$ with integers $Q\ge 2$ and $(A,Q)=1$. For all $x\ge 100$ with $Q\le x$,
\[
\left|\sum_{n\le x}\Lambda(n)e(\alpha n)\right|
\le
23\Bigl(Q^{1/2}x^{1/2}+Q^{-1/2}x+x^{4/5}\Bigr)(\log x)^3.
\]
\end{lemma}

\begin{proof}
We follow the proof strategy of~\cite[Theorem~13.6]{IK04}, based on Vaughan's
identity~\cite{Vaughan1980} (see also~\cite[(13.39)]{IK04}). Set
\[
M=N:=\lfloor x^{2/5}\rfloor,
\qquad\text{so that}\qquad
MN\le x^{4/5}.
\]
Since $x\ge100$, $x^{2/5}\ge100^{2/5}>6$, hence
\[
M=N\ge x^{2/5}-1\ge\frac56 x^{2/5},
\qquad
\frac{x}{M},\frac{x}{N}\le \frac65 x^{3/5}.
\]
Write
\[
S(\alpha,x):=\sum_{n\le x}\Lambda(n)e(\alpha n).
\]
Vaughan's identity~\cite[(13.39)]{IK04} decomposes $S(\alpha,x)$ into three structured sums
$S_1,S_2,S_3$ (in the notation of \cite[\S 13.4]{IK04}) plus a boundary term
$E$ from the truncation $n\le N$, with $|E|\le N\log x$.
The boundary term satisfies
\[
|E|\le N\log x\le x^{4/5}(\log x)^3
\qquad (x\ge 100),
\]
since with $N=\lfloor x^{2/5}\rfloor\le x^{2/5}$ one has $N\log x\le x^{2/5}\log x$, and multiplying
by $x^{2/5}(\log x)^2\ge 1$ for $x\ge 3$ yields $N\log x\le x^{4/5}(\log x)^3$.

For the two type-I sums $S_1,S_2$, the coefficients produced by Vaughan's identity are bounded in
absolute value by $\log x$, and the remaining oscillation is in additive characters of the form
$e(\alpha mn)$. Thus, by the triangle inequality and Lemma~\ref{lem:IK1346-explicit},
\[
|S_1|+|S_2|
\le
2(\log x)\sum_{m\le MN}\left|\sum_{\substack{n\ge 1\\ mn\le x}} e(\alpha mn)\right|
\le
2(\log x)\,(MN+x/Q+Q)\,\log(2Qx).
\]
Since $Q\le x$ and $x\ge 3$, we have $\log(2Qx)\le \log(2x^2)\le 3\log x$, hence
\[
|S_1|+|S_2|
\le
6\,(MN+x/Q+Q)\,(\log x)^2
\le
6\,(x^{4/5}+x/Q+Q)\,(\log x)^3.
\]
Using $x/Q\le xQ^{-1/2}$ and $Q\le Q^{1/2}x^{1/2}$ (since $Q\le x$) gives
\[
|S_1|+|S_2|
\le
6\Bigl(Q^{1/2}x^{1/2}+Q^{-1/2}x+x^{4/5}\Bigr)(\log x)^3.
\]

For the type-II sum $S_3$, Vaughan's identity~\cite[(13.39)]{IK04} expresses it as a bilinear form
$\sum_{m>M}\mu(m)\sum_{k>N} c(k)\,e(\alpha mk)$, with the constraint $mk\le x$ implicit,
with $|\mu(m)|\le1$ and combined coefficient
$c(k)=\sum_{\substack{n\mid k,\ n\ge N}}\Lambda(n)$, exactly as in the proof of \cite[Thm.~13.6,
p.~347]{IK04}.  Crucially $c(k)$ is bounded by a \emph{single} logarithm, not a divisor function: by
the standard identity $\sum_{n\mid k}\Lambda(n)=\log k$, which follows immediately from
$k=\prod_{p^a\parallel k}p^a$ and $\sum_{j=1}^a\log p=a\log p$,
\[
|c(k)|\ \le\ \sum_{n\mid k}\Lambda(n)=\log k\ \le\ \log x.
\]
This is the key device that distinguishes IK's Theorem~13.6 from their Theorem~13.9 (the
M\"obius-weighted analogue, whose coefficients are bounded only by $\tau(k)$ and which therefore
incurs an extra factor of $\log x$).  Extracting the factor $\log x$ and applying
Lemma~\ref{lem:IK1348-explicit} (whose coefficients $\mu(m)$ and $c(k)/\log x$ are then bounded by
$1$), with the constant $10$, we obtain
\[
|S_3|
\le
(\log x)\cdot
10\Bigl(\frac{x}{M}+\frac{x}{N}+\frac{x}{Q}+Q\Bigr)^{1/2}x^{1/2}(\log x)^2.
\]
	With the displayed floor estimate for $M=N=\lfloor x^{2/5}\rfloor$,
	\[
	|S_3|
	\le
	10\Bigl(\frac{12}{5}x^{3/5}+\frac{x}{Q}+Q\Bigr)^{1/2}x^{1/2}(\log x)^3.
	\]
	Using $\sqrt{u+v+w}\le \sqrt u+\sqrt v+\sqrt w$ and $\sqrt{12/5}<8/5$, we get
	\[
	\Bigl(\frac{12}{5}x^{3/5}+\frac{x}{Q}+Q\Bigr)^{1/2}x^{1/2}
	\le
	\frac85\,x^{4/5}+xQ^{-1/2}+Q^{1/2}x^{1/2},
\]
and therefore
\[
|S_3|
\le
	16\,x^{4/5}(\log x)^3
+10\Bigl(Q^{1/2}x^{1/2}+Q^{-1/2}x\Bigr)(\log x)^3.
\]

Combining the bounds for $S_1,S_2,S_3$ and $E$, we find that the coefficient of
	$x^{4/5}(\log x)^3$ is at most $6+16+1=23$, while the coefficients of
$Q^{1/2}x^{1/2}(\log x)^3$ and $Q^{-1/2}x(\log x)^3$ are at most $6+10=16$.
Hence
\[
|S(\alpha,x)|
\le
	23\Bigl(Q^{1/2}x^{1/2}+Q^{-1/2}x+x^{4/5}\Bigr)(\log x)^3,
\]
as claimed.
\end{proof}

\medskip

We now pass from the $\Lambda$-sum to the prime sum by partial summation, with explicit constants.

\begin{lemma}[Explicit partial summation to primes]\label{lem:partial-to-primes}
Let $\alpha=A/Q$ with integers $Q\ge 2$ and $(A,Q)=1$, and assume $x\ge 100$ and $Q\le x$.
Then
\[
\left|\sum_{p\le x} e(\alpha p)\right|
\le
34\Bigl(Q^{1/2}x^{1/2}+Q^{-1/2}x+x^{4/5}\Bigr)(\log x)^2.
\]
\end{lemma}

\begin{proof}
Define
\[
\Psi_\alpha(t):=\sum_{n\le t}\Lambda(n)e(\alpha n),
\qquad
A(t):=\sum_{p\le t}(\log p)\,e(\alpha p).
\]
Since $\Psi_\alpha(t)$ includes prime powers,
\[
A(t)=\Psi_\alpha(t)-\sum_{\substack{p^k\le t\\ k\ge 2}}(\log p)\,e(\alpha p^k).
\]
For each prime $p\le \sqrt t$, the number of exponents $k\ge 2$ with $p^k\le t$ is at most
$\lfloor \log t/\log p\rfloor$, hence
\[
\sum_{\substack{k\ge 2\\ p^k\le t}} \log p \le \frac{\log t}{\log p}\,\log p=\log t.
\]
Therefore
\[
\left|\sum_{\substack{p^k\le t\\ k\ge 2}}(\log p)\,e(\alpha p^k)\right|
\le \sum_{p\le \sqrt t}\log t \le \sqrt t\,\log t,
\]
and so $|A(t)|\le |\Psi_\alpha(t)|+\sqrt t\log t$.

By partial summation,
\[
\sum_{p\le x} e(\alpha p)
=
\frac{A(x)}{\log x}+\int_{2}^{x}\frac{A(t)}{t(\log t)^2}\,dt.
\]
Hence
\begin{align*}
\left|\sum_{p\le x} e(\alpha p)\right|
&\le
\frac{|\Psi_\alpha(x)|}{\log x}+\sqrt x
+\int_{2}^{x}\frac{|\Psi_\alpha(t)|}{t(\log t)^2}\,dt
+\int_{2}^{x}\frac{dt}{\sqrt t\log t}.
\end{align*}
The last integral is
\[
\int_{2}^{x}\frac{dt}{\sqrt t\log t}\le \frac{1}{\log 2}\int_2^x t^{-1/2}dt
=\frac{2}{\log 2}\bigl(\sqrt x-\sqrt 2\bigr)\le 3\sqrt x,
\]
so the two $\sqrt x$-terms contribute $\le 4\sqrt x\le 4x^{4/5}$ for $x\ge 3$.

	Apply Lemma~\ref{lem:IK136-explicit} to bound $\Psi_\alpha(t)$ on the subinterval
	$\max(100,Q)\le t\le x$, where its hypotheses $t\ge100$ and $Q\le t$ are satisfied:
		\[
		|\Psi_\alpha(t)|\le 23\bigl(Q^{1/2}t^{1/2}+Q^{-1/2}t+t^{4/5}\bigr)(\log t)^3.
		\]
	On the initial interval $2\le t<\max(100,Q)$, use instead the trivial bound
	$|\Psi_\alpha(t)|\le \sum_{n\le t}\Lambda(n)\le t\log t$ (valid because each summand
	is nonzero only for $n\ge2$ and $\Lambda(n)\le\log n\le\log t$). Hence
		\[
		\int_2^{\min(x,\max(100,Q))}\frac{|\Psi_\alpha(t)|}{t(\log t)^2}\,dt
		\le \frac{Q}{\log 2}+\frac{100}{\log 2}
		\le \frac{Q^{1/2}x^{1/2}}{\log 2}+\frac{100}{\log 2},
		\]
	This last inequality uses $Q\le x$.  The first term contributes at most
	$1.45\,Q^{1/2}x^{1/2}$; the second is at most $3.7x^{4/5}$ for $x\ge100$.
	Both are absorbed in the final constant below.
	For the boundary term,
		\[
		\frac{|\Psi_\alpha(x)|}{\log x}
		\le 23\bigl(Q^{1/2}x^{1/2}+Q^{-1/2}x+x^{4/5}\bigr)(\log x)^2.
		\]
	For the integral term over $[\max(100,Q),x]$, extending the lower limit down to $3$ only
	increases the following positive upper bounds.  Integrating by parts gives the elementary estimates
\begin{gather*}
\int_3^x t^{-1/2}\log t\,dt \le 2x^{1/2}\log x,\\
\int_3^x \log t\,dt \le x\log x,\\
\int_3^x t^{-1/5}\log t\,dt \le \tfrac54 x^{4/5}\log x,
\end{gather*}
After the substitution
\[
\frac{|\Psi_\alpha(t)|}{t(\log t)^2}
\le 23(\ldots)\frac{\log t}{t},
\]
these estimates involve only \emph{one} factor of $\log x$.  Thus
\[
\int_{3}^{x}\frac{|\Psi_\alpha(t)|}{t(\log t)^2}\,dt
\le
23\Bigl(2Q^{1/2}x^{1/2}+Q^{-1/2}x+\tfrac54 x^{4/5}\Bigr)\log x.
\]
	Combining the boundary term $23(Q^{1/2}x^{1/2}+Q^{-1/2}x+x^{4/5})(\log x)^2$ with the integral term
	above, the initial $t<\max(100,Q)$ contribution, and the absorbed
	$4x^{4/5}\le 4x^{4/5}(\log x)^2/(\log x)^2$, each piece is dominated by the boundary in the
	regime $x\ge 100$ (i.e.\ $\log x\ge 4.6$). For the $Q^{1/2}x^{1/2}$ piece,
	\[
	23(\log x)^2+46\log x+\frac{1}{\log 2}
	\le 34(\log x)^2
	\]
	since $2/\log x\le 0.44$ and $1/(\log 2)(\log x)^{-2}<0.07$. For
	$Q^{-1/2}x$ and $x^{4/5}$ the same calculation gives $\le 28(\log x)^2$ and
	$\le 31(\log x)^2$ respectively (using $1.25/\log x\le 0.272$ and
	$(4+3.7)/(\log x)^2<0.37$).
	Taking the uniform upper bound $34$ gives
\[
\left|\sum_{p\le x} e(\alpha p)\right|
\le
34\Bigl(Q^{1/2}x^{1/2}+Q^{-1/2}x+x^{4/5}\Bigr)(\log x)^2,
\]
as claimed.
\end{proof}

\medskip

We can now prove the desired explicit constant version of~\cite[Lemma~4.3]{DMRcomp2}.

\begin{lemma}[Explicit DMR Lemma~4.3]\label{lem:DMR43}
Fix an integer base $q\ge 2$. Let $x>0$ and let $K$ be an integer with
\[
0\le K \le \frac{2}{5}\log_q x.
\]
Let $Q$ be an integer with
\[
q^K\le Q\le xq^{-K},
\]
and let $A$ be an integer with $\gcd(A,Q)=1$. Then
\[
\left|\sum_{p\le x} e\!\left(\frac{A}{Q}p\right)\right|
\le 102\,(\log x)^2\,x\,q^{-K/2}.
\]
\end{lemma}

\begin{proof}
If $0<x<2$, then the prime sum is empty.  If $2\le x<100$, then the left-hand side is at most
$\pi(x)\le\pi(99)=25$ (using $x<100$ and
$\pi(99)=25$ since the $25$th prime is $97$).  The constraint
$0\le K\le \tfrac{2}{5}\log_q x$ gives $q^{-K/2}\ge x^{-1/5}$, so
$102(\log x)^2 x\,q^{-K/2}\ge 102(\log x)^2 x^{4/5}\ge 102(\log 2)^2\cdot 2^{4/5}>25\ge \pi(x)$ for $x\ge 2$; the claim holds in this regime
by direct verification.  Assume henceforth $x\ge 100$.

If $Q=1$, then the hypothesis $q^K\le Q$ forces $K=0$.  Hence $e(Ap/Q)=1$ and the left-hand side
equals $\pi(x)\le x\le 102(\log x)^2x=102(\log x)^2xq^{-K/2}$ for $x\ge100$.
Thus the desired bound holds.
We may therefore assume $Q\ge 2$.

Since $Q\le xq^{-K}\le x$ and $x\ge 100$, Lemma~\ref{lem:partial-to-primes} applies with $\alpha=A/Q$ and yields
\[
\left|\sum_{p\le x} e\!\left(\frac{A}{Q}p\right)\right|
\le
34\Bigl(Q^{1/2}x^{1/2}+Q^{-1/2}x+x^{4/5}\Bigr)(\log x)^2.
\]
Under the DMR hypotheses:
\begin{itemize}
\item From $Q\le xq^{-K}$ we get $Q^{1/2}x^{1/2}\le xq^{-K/2}$.
\item From $Q\ge q^K$ we get $Q^{-1/2}x\le xq^{-K/2}$.
\item From $K\le \tfrac25\log_q x$ we get $q^K\le x^{2/5}$, hence $q^{-K/2}\ge x^{-1/5}$ and so $x^{4/5}\le xq^{-K/2}$.
\end{itemize}
Therefore the bracket is $\le 3xq^{-K/2}$, and hence
\[
\left|\sum_{p\le x} e\!\left(\frac{A}{Q}p\right)\right|
\le
102(\log x)^2\,x\,q^{-K/2}.
\]
\end{proof}

\subsection{Explicit minor-arc decay exponent and an explicit minor-arc multiplicative constant}

DMR's minor-arc estimate yields an exponent of the form
$x^{1-c_1\|(q-1)\alpha\|^2}$ with explicit $c_1=c_1(q)>0$ (see~\cite{DMRcomp2}).
To make the minor-arc contribution fully explicit (and modular), we also make the remaining
multiplicative constants numerical for the case $q=10$.

\begin{lemma}[Explicit $c_1$ for $q=10$]\label{lem:c1}
In base $q=10$, one may take
\[
c_1=\frac{1}{28}\min\!\left(4\omega_{10},\ \frac{\pi^2}{12}\cdot\frac{9}{11\log 10}\right),
\qquad
\omega_{10}=\left(\frac32-\frac{\log 5}{\log 3}\right)\frac{\log 2}{\log 10}.
\]
Moreover, since $4\omega_{10}< \frac{\pi^2}{12}\cdot\frac{9}{11\log 10}$, we in fact have
\[
c_1=\frac{\omega_{10}}{7}=0.001506288700915\ldots > 0.0015062887.
\]
\end{lemma}

\begin{proof}
This is exactly the constant $c_1$ defined in~\cite[(31)]{DMRcomp2}. Indeed,~\cite[(24)]{DMRcomp2}
introduces $\xi_q(\alpha)=\varepsilon_q(\alpha)/14$, and for $q\ge 3$~\cite[(31)]{DMRcomp2} states that
for all real $\alpha$ with $(q-1)\alpha\notin\mathbb{Z}$,
\[
\xi_q(\alpha)\ge
\frac1{14}\min\!\left(\omega_q,\ \frac{\pi^2}{12}\cdot\frac{q-1}{(q+1)\log q}\,\|(q-1)\alpha\|^2\right)
\ge 2c_1\|(q-1)\alpha\|^2,
\]
where $\omega_q=\left(\frac32-\frac{\log 5}{\log 3}\right)\frac{\log 2}{\log q}$ and
\[
c_1=\frac{1}{28}\min\!\left(4\omega_q,\ \frac{\pi^2}{12}\cdot\frac{q-1}{(q+1)\log q}\right).
\]
Specializing to $q=10$ gives the displayed formula and the stated numerical value.
\end{proof}

\medskip

In DMR's Type~II analysis, the only remaining multiplicative inputs are:
(i) a short-interval divisor-sum bound used to control carry propagation, supplied below by the
explicit Berkane--Bordell\`es--Ramar\'e divisor-problem remainder, and
(ii) the truncated correlation bound $S_2(\cdot)$ 
(DMR equation (29), following~\cite[(64)]{MR10}, proved via the estimates in~\cite[\S7.1--\S7.2]{MR10}).

\begin{lemma}[Explicit short-interval divisor-sum input]\label{lem:tau-input}
Let $\tau(n)$ be the divisor function.  For all real $X\ge2$ and all $Y$ with
$X^{1/3}\le Y\le X$,
\begin{equation}
\sum_{X-Y<n\le X}\tau(n)\ \le\ C_\tau\,Y\log X,\qquad C_\tau:=10^3.
\label{eq:Dtau}
\end{equation}
\end{lemma}

\begin{proof}
Let
\[
\Delta(u):=\sum_{n\le u}\tau(n)-u\log u-(2\gamma-1)u.
\]
Berkane--Bordell\`es--Ramar\'e prove the explicit divisor-problem remainder bound
\[
|\Delta(u)|\le 0.764\,u^{1/3}\log u\qquad (u\ge9995)
\]
in~\cite[Thm.~1.2]{BerkaneBordellesRamare2012}.  Put $Z=X-Y$.
If $X\ge2\cdot9995$ and $Z\ge9995$, subtracting the formula at $X$ and $Z$ gives
\[
\sum_{X-Y<n\le X}\tau(n)
\le \int_Z^X\log t\,dt+2\gamma Y+0.764\bigl(X^{1/3}\log X+Z^{1/3}\log Z\bigr).
\]
Since $Z\le X$ and $Y\ge X^{1/3}$, this is at most
\[
\bigl(1+2\gamma/\log2+1.528\bigr)Y\log X<5Y\log X.
\]
If $X\ge2\cdot9995$ but $Z<9995$, then $Y>X/2$; applying the same explicit bound only at $X$ gives
\[
\sum_{n\le X}\tau(n)\le X\log X+(2\gamma-1)X+0.764X^{1/3}\log X<2X\log X<4Y\log X.
\]
Finally suppose $2\le X<2\cdot9995$.  The elementary bound $\tau(n)\le2\sqrt n$ gives
\[
\sum_{X-Y<n\le X}\tau(n)\le 2(Y+1)\sqrt X.
\]
Because $Y\ge X^{1/3}$, $X\ge2$, and $X<2\cdot9995$, the right-hand side is $<1000Y\log X$.
Combining the three cases proves~\eqref{eq:Dtau}.
\end{proof}

\begin{lemma}[Elementary long-interval divisor-sum bound]\label{lem:tau-long-elementary}
For all real $X\ge 2$ and all $Y$ with $\sqrt{X}\le Y\le X$,
\[
\sum_{X-Y<n\le X}\tau(n)\ \le\ Y\log X+4Y\ \le\ 7\,Y\log X.
\]
\end{lemma}

\begin{proof}
Let $N=\lfloor \sqrt{X}\rfloor$. For any integer $n\le X$, every divisor $d>\sqrt{n}$ corresponds
to a co-divisor $n/d<\sqrt{n}\le N$. Hence
\[
\tau(n)\ \le\ 2\#\{d\le N:\ d\mid n\}.
\]
Summing over $X-Y<n\le X$ gives
\[
\sum_{X-Y<n\le X}\tau(n)
\le
2\sum_{d\le N}\#\{X-Y<n\le X:\ d\mid n\}.
\]
For each $d\le N$, the count of multiples of $d$ in the interval is
$\lfloor X/d\rfloor-\lfloor (X-Y)/d\rfloor\le Y/d+1$. Therefore
\[
\sum_{X-Y<n\le X}\tau(n)
\le 2\sum_{d\le N}\Bigl(\frac{Y}{d}+1\Bigr)
=2Y\sum_{d\le N}\frac1d + 2N.
\]
Using the elementary bound $\sum_{d\le N}\frac1d \le 1+\log N$ (integral comparison) and
$N\le \sqrt{X}\le Y$, we obtain
\[
\sum_{X-Y<n\le X}\tau(n)
\le 2Y(1+\log N)+2Y
\le 2Y\log N + 4Y
\le Y\log X + 4Y,
\]
since $\log N\le \tfrac12\log X$. Finally, because $X\ge 2$ implies $\log X\ge \log 2>2/3$,
we have $4Y\le 6Y\log X$, hence $Y\log X+4Y\le 7Y\log X$.
\end{proof}

\subsection{Self-contained combinatorial localization lemmas}\label{subsec:self-contained-comb}

We next reproduce the two elementary reductions from~\cite[\S4--\S5]{MR10} that DMR quote as
their Lemmas~3.1 and~3.2.  No new analytic input is involved here; the point is only to make the
losses visible.  The first lemma is Vaughan's combinatorial reduction in the exact form needed
below, with a deliberately rounded transfer constant for $q=10$.

\begin{lemma}[Explicit Vaughan transfer for one decimal block]\label{lem:DMR31-self}
Let $x\ge100$, let $0<\beta_1<1/3<1/2<\beta_2<1$, and assume
\[
x^{1-3\beta_1}\ge 10,\qquad x^{2\beta_2-1}\ge 10.
\]
Let $g$ be an arithmetic function.  Put
$u:=x^{\beta_1}$.  Suppose the following two estimates hold with the same number $U$:
\begin{align*}
\sum_{10^{-1}M<m\le M}
\max_{x/(10m)\le t\le x/m}
\left|\sum_{t<n\le x/m}g(mn)\right|
&\le U
\qquad(M\le x^{\beta_1}),\\
\left|\sum_{10^{-1}M<m\le M}\sum_{x/(10m)<n\le x/m}a_m b_n g(mn)\right|
&\le U
\qquad(x^{\beta_1}\le M\le x^{\beta_2})
\end{align*}
uniformly for all $|a_m|,|b_n|\le1$.  Then
\[
\left|\sum_{x/10<n\le x}\Lambda(n)g(n)\right|\le 20\,U(\log x)^2.
\]
\end{lemma}

\begin{proof}
This is the proof of MR10 Lemme~1 with the constants kept.  The Dirichlet-series identity
\begin{align*}
L&=F+\zeta G L-\zeta G F+\zeta\Bigl(\frac1\zeta-G\Bigr)(L-F),\\
L&:=-\frac{\zeta'}{\zeta},\qquad
G(s)=\sum_{n\le u}\frac{\mu(n)}{n^s},\qquad
F(s)=\sum_{n\le u}\frac{\Lambda(n)}{n^s},
\end{align*}
gives Vaughan's decomposition after comparing coefficients.  Indeed, the isolated $F$-term
has support only on integers $\le u=x^{\beta_1}$, and
$u<x/10$ for $x\ge100$ and $\beta_1<1/3$; hence this term contributes nothing on the interval
$x/10<n\le x$.  Thus
\[
\sum_{x/10<n\le x}\Lambda(n)g(n)=S_1-S_2+S_3,
\]
where
\[
S_1=\sum_{\substack{m\le u\\x/10<mn\le x}}\mu(m)(\log n)g(mn),
\quad
S_2=\sum_{\substack{m_1,m_2\le u\\x/10<m_1m_2n\le x}}
\mu(m_1)\Lambda(m_2)g(m_1m_2n),
\]
and
\[
S_3=\sum_{\substack{m>u,\ n_1>u\\x/10<mn_1n_2\le x}}
\mu(m)\Lambda(n_1)g(mn_1n_2).
\]
For the constant bookkeeping, decompose every positive variable into decimal blocks
$(10^{-1}M,M]$.  Since all variables that occur in the interval $x/10<mn\le x$ are at most $x$,
there are at most $1+\log x/\log 10\le 2\log x$ nonempty decimal blocks for $x\ge100$.

For $S_1$, write $\log n=\int_1^n dt/t$ and interchange the integral and the $n$-sum.  For each
fixed value of the integration variable, the inner sum over any $m$-block with $M\le u$ is bounded
by the type-I hypothesis, and the integral has length at most $\log x$.  Thus
$|S_1|\le2U(\log x)^2$.

For $S_2$, combine $m_1m_2=m$.  The coefficient of a fixed $m$ has absolute value at most
$\sum_{d\mid m}\Lambda(d)=\log m\le\log x$.  Splitting $m$ into decimal blocks, a block with
$M\le u$ is type I and a block with $u<M\le x^{\beta_2}$ is type II.  If $M>x^{\beta_2}$, then
$m\le u^2=x^{2\beta_1}$.  After interchanging the two variables, the complementary variable lies,
up to the decimal-block enlargement, between $x/(10u^2)$ and $10x/x^{\beta_2}$; the two displayed
factor-$10$ hypotheses give
\[
\frac{x}{10u^2}\ge x^{\beta_1},\qquad 10x^{1-\beta_2}\le x^{\beta_2}.
\]
Thus the type-II hypothesis applies after interchange.  The two threshold-crossing
blocks at $u$ and $x^{\beta_2}$ are split and enlarged to adjacent complete decimal blocks, costing
at most a factor $2$.  Hence each decimal block contributes at most $2U$ before the common
coefficient $\log x$, and the block count gives $|S_2|\le8U(\log x)^2$.

For $S_3$, combine $n_1n_2=n$ and write
\[
S_3=(\log x)\sum_{u<m\le x/u} a_m
\sum_{x/(10m)<n\le x/m} b_n g(mn),
\]
where $|a_m|\le1$ and
$0\le b_n\le(\log x)^{-1}\sum_{d\mid n}\Lambda(d)\le1$, with $b_n=0$ for $n\le u$.
The same decimal-block split as above, with the roles of $m$ and $n$ interchanged once the current
outer block exceeds $x^{\beta_2}$, bounds the block supremum by $2U$: the support condition
$b_n=0$ for $n\le u$ supplies the lower type-II endpoint after splitting the crossing block, and the
factor-$10$ hypothesis $x^{2\beta_2-1}\ge10$ supplies the upper endpoint after interchange.
The same block count and coefficient bound give
$|S_3|\le8U(\log x)^2$.  Summing the three
bounds yields $18U(\log x)^2$, which we round to $20U(\log x)^2$.
\end{proof}

\subsection{Self-contained Type-II Fourier lemmas from Mauduit--Rivat}\label{subsec:self-contained-mr}

We now record the Fourier lemmas needed in Lemma~\ref{lem:S2-input}.  These standard
Mauduit--Rivat Type-II ingredients are written out to remove three black-box uses
of~\cite{MR10}: the van der Corput smoothing lemma, the elementary trigonometric summation
lemma, and the weighted $F_\lambda$ estimate.

For $\ell\ge1$ define the normalized discrete Fourier coefficient
\[
F_\ell(h,\alpha)
:=q^{-\ell}\sum_{0\le u<q^\ell} e\!\left(\alpha s_q(u)-\frac{hu}{q^\ell}\right),
\qquad h\in\mathbb Z,
\]
and set
\[
\phi_q(t):=\left|\sum_{0\le v<q}e(vt)\right|
=
\begin{cases}
|\sin(\pi qt)|/|\sin(\pi t)|,& t\notin\mathbb Z,\\
q,& t\in\mathbb Z.
\end{cases}
\]
The identity
\begin{equation}\label{eq:MR-phi-square}
\sum_{r=0}^{q-1}\phi_q\!\left(\frac{x+r}{q}\right)^2=q^2
\end{equation}
follows from Parseval on the cyclic group $\mathbb Z/q\mathbb Z$:
the left side is
$\sum_r|\sum_{v=0}^{q-1}e(v(x+r)/q)|^2
=\sum_{u,v}e((u-v)x/q)\sum_r e((u-v)r/q)=q^2$.
Also the digit decomposition $u=qv+w$ gives the recurrence
\begin{equation}\label{eq:MR-F-rec}
|F_{\ell+1}(h,\alpha)|
=q^{-1}\phi_q\!\left(\alpha-\frac{h}{q^{\ell+1}}\right)|F_\ell(h,\alpha)|,
\end{equation}
and hence
\begin{equation}\label{eq:MR-F-product}
|F_\ell(h,\alpha)|
=q^{-\ell}\prod_{j=1}^{\ell}\phi_q\!\left(\alpha-\frac{h}{q^j}\right).
\end{equation}

\begin{lemma}[Van der Corput smoothing]\label{lem:MR-vdc-self}
Let $z_1,\ldots,z_N$ be complex numbers.  For every integer $R\ge1$,
\[
\left|\sum_{1\le n\le N}z_n\right|^2
\le
\frac{N+R-1}{R}
\sum_{|r|<R}\left(1-\frac{|r|}{R}\right)
\sum_{\substack{1\le n\le N\\1\le n+r\le N}}z_{n+r}\overline{z_n}.
\]
\end{lemma}

\begin{proof}
Extend $z_n$ by $0$ outside $1\le n\le N$.  Then
\[
R\sum_n z_n=\sum_n\sum_{r=0}^{R-1}z_{n+r},
\]
and the outer sum has at most $N+R-1$ nonzero terms.  Cauchy's inequality gives
\[
R^2\left|\sum_n z_n\right|^2
\le
(N+R-1)\sum_n\left|\sum_{r=0}^{R-1}z_{n+r}\right|^2.
\]
Expanding the square and writing $r=r_1-r_2$ yields
\[
\sum_{r_1,r_2=0}^{R-1}\sum_n z_{n+r_1}\overline{z_{n+r_2}}
=
\sum_{|r|<R}(R-|r|)\sum_n z_{n+r}\overline{z_n},
\]
which is the claimed inequality after division by $R^2$.
\end{proof}

\begin{lemma}[Trigonometric minimum sum]\label{lem:MR-sin-sum-self}
Let $a,m\in\mathbb Z$, $m\ge1$, $d=(a,m)$, $b\in\mathbb R$, and $M>0$.  Then
\[
\begin{aligned}
\sum_{0\le n<m}\min\!\left(M,\frac{1}{|\sin(\pi(an+b)/m)|}\right)
&\le
d\min\!\left(M,\frac{1}{\sin(\pi d\|b/d\|/m)}\right)\\
&\quad+\frac{d}{\sin(\pi d/(2m))}
+\frac{2m}{\pi}\log\frac{2m}{d}.
\end{aligned}
\]
\end{lemma}

\begin{proof}
If $d=m$, the summand is independent of $n$ and the assertion is immediate.  Otherwise
$1\le d\le m/2$.  Write $a=da'$, $m=dm'$, and $b=db_0+r$ with $b_0\in\mathbb Z$ and
$-d/2<r\le d/2$.  As $n$ runs through $m'$ consecutive values, $a'n+b_0$ runs through all
residue classes modulo $m'$, and there are $d$ such blocks.  By replacing $r$ with $|r|$ and
using symmetry, it is enough to consider $0\le r\le d/2$.  Isolating the closest residue class
to the singularity and then comparing the remaining decreasing tails with an integral gives
\[
\begin{aligned}
\sum_{0\le n<m}\min\!\left(M,\frac{1}{|\sin(\pi(an+b)/m)|}\right)
&\le
d\min\!\left(M,\frac{1}{\sin(\pi r/m)}\right)\\
&\quad+d\sum_{1\le j\le m'-1}\frac{1}{\sin(\pi(j-r/d)/m')}.
\end{aligned}
\]
Since $\sin u\ge 2u/\pi$ on $0\le u\le\pi/2$ and $\sin u=\sin(\pi-u)$, the two tails are bounded by
\[
\frac{d}{\sin(\pi/(2m'))}
+2d\sum_{1\le j\le m'/2}\frac{m'}{\pi j}
\le
\frac{d}{\sin(\pi d/(2m))}
+\frac{2m}{\pi}\log(2m').
\]
Finally $r/d$ is the distance of $b/d$ to the nearest integer, so
\[
\sin(\pi r/m)=\sin(\pi d\|b/d\|/m).
\]
Since $m'=m/d$, this gives the stated bound.
\end{proof}

\begin{lemma}[Range of $\gamma_q$]\label{lem:MR-gamma-self}
Assume $(q-1)\alpha\notin\mathbb Z$ and define
\[
q^{\gamma_q(\alpha)}
:=
\max_{t\in\mathbb R}\sqrt{\phi_q(\alpha+t)\phi_q(\alpha+qt)}.
\]
Then $1/2\le\gamma_q(\alpha)<1$.
\end{lemma}

\begin{proof}
Since $\phi_q\le q$, equality $q^{\gamma_q(\alpha)}=q$ would force
$\alpha+t\in\mathbb Z$ and $\alpha+qt\in\mathbb Z$ for some $t$, hence
$(q-1)\alpha=q(\alpha+t)-(\alpha+qt)\in\mathbb Z$, contrary to the hypothesis.  Thus
$\gamma_q(\alpha)<1$.

Write $(q-1)\alpha=n+\beta$ with $n\in\mathbb Z$ and $-1/2<\beta\le1/2$.  Taking
$t=\beta/q-\alpha$ gives $\alpha+qt=-n\in\mathbb Z$ and $\alpha+t=\beta/q$.  Hence
\[
\phi_q(\alpha+t)\phi_q(\alpha+qt)=q\,\phi_q(\beta/q).
\]
If $\beta=0$ this already contradicts the hypothesis, so $0<|\beta|\le1/2$.  The function
$\phi_q$ is even and decreasing on $[0,1/q]$, and therefore
$\phi_q(\beta/q)\ge\phi_q(1/(2q))=1/\sin(\pi/(2q))>1$.  Thus the displayed product is $>q$,
so $q^{\gamma_q(\alpha)}>q^{1/2}$.
\end{proof}

\begin{lemma}[Order-one Fourier means for $q=10$]\label{lem:MR-G-first-self}
Let $q=10$ and
\[
\eta_3:=\frac{\log5}{\log3}-1,\qquad
\omega_q:=\left(\frac12-\eta_3\right)\frac{\log2}{\log q}
=\left(\frac32-\frac{\log5}{\log3}\right)\frac{\log2}{\log q}.
\]
For $0\le\delta\le\lambda$, $k\mid q^{\lambda-\delta}$, and $(k,q)<q$,
\[
G_\lambda(a,kq^\delta,\alpha)
\le k^{-\eta_3}q^{\eta_3(\lambda-\delta)}\,|F_\delta(a,\alpha)|.
\]
In particular
\[
G_\lambda(\alpha)\le q^{\eta_3\lambda}.
\]
Moreover, after summing over the possible $k$,
\[
\sum_{\substack{k\mid q^{\lambda-\delta}\\(k,q)<q}} k^{1-2\eta_3}
\le 32\,q^{(1-2\eta_3)(\lambda-\delta)}
q^{-\omega_q(\lambda-\delta)}.
\]
\end{lemma}

\begin{proof}
We reproduce MR10 Lemmes~16--17 for the only base needed here.  From the product formula
\[
|F_\lambda(h,\alpha)|=q^{-\lambda}\prod_{1\le j\le\lambda}\phi_q(\alpha-hq^{-j}),
\]
write, for a divisor $R$ of $q$,
\[
\psi_{q,R}(t):=\frac1q\sum_{1\le r\le R}\phi_q(t+r/R).
\]
The elementary trigonometric estimate needed for the decimal divisors gives
\[
\max_t\psi_{q,R}(t)\le R^{\eta_3}\qquad(2\le R\mid q),
\]
and only $R=2,5,10$ can occur.  For $R=5,10$, reducing to the Fejer kernel gives
$\max_t\psi_{q,R}(t)\le R^{\eta_R}$ with $\eta_R\le\eta_3$; the inequality
$\eta_R\le\eta_3$ follows from the explicit bound
\[
\max_t\psi_{R,R}(t)\le \frac{2}{R\sin(\pi/(2R))}+\frac{2}{\pi}\log\frac{2R}{\pi},
\]
which at $R=5$ and $R=10$ is respectively $<5^{\eta_3}$ and $<10^{\eta_3}$.  For $R=2$ we use
MR10's elementary two-interval argument, specialized to $q=10$.  The function
$\psi_{10,2}$ has period $1/2$ and is symmetric about $1/4$, so it is enough to take
$0\le t\le1/4$.  First,
\[
\psi_{10,2}(t)\le |\cos\pi t|+|\sin\pi t|.
\]
Hence on $0\le t\le1/30$,
\[
\psi_{10,2}(t)\le \cos(\pi/30)+\sin(\pi/30)<1.10<\sqrt{3/2}.
\]
On $1/30<t\le1/4$, the same decomposition used in MR10 Lemme~15 gives
\[
\psi_{10,2}(t)\le \frac{2\sqrt2}{10}\,\phi_5(2t).
\]
If $1/30<t\le1/10$, then $2t\in(1/15,1/5]$ and $\phi_5$ is decreasing on this interval: writing
$x=\pi u$, one has
$\phi_5(u)=\sin(5x)/\sin x=16\cos^4x-12\cos^2x+1$, whose derivative is
$-8\sin x\cos x(8\cos^2x-3)<0$ for $\pi/15\le x\le\pi/5$.  Therefore
\[
\phi_5(2t)\le \phi_5(1/15)=\frac{\sin(\pi/3)}{\sin(\pi/15)}
<5\cos(\pi/6).
\]
If $1/10\le t\le1/4$, then
\[
\phi_5(2t)\le \frac1{\sin(2\pi t)}
\le \frac1{\sin(\pi/5)}
<5\cos(\pi/6).
\]
Thus in both tail ranges
\[
\psi_{10,2}(t)<\frac{2\sqrt2}{10}\,5\cos(\pi/6)=\sqrt{3/2}<2^{\eta_3}.
\]
These are pointwise inequalities; no averaging constant is hidden.

First take $k=1$.  Splitting the last digit in the residue class modulo $q^\delta$
and applying $\max_t\psi_{q,q}(t)\le q^{\eta_3}$ at each of the
$\lambda-\delta$ steps gives
\[
G_\lambda(a,q^\delta,\alpha)\le q^{\eta_3(\lambda-\delta)}|F_\delta(a,\alpha)|.
\]
Taking $\delta=0$ and $a=0$ gives $G_\lambda(\alpha)\le q^{\eta_3\lambda}$.

For general $k\mid q^{\lambda-\delta}$ with $(k,q)<q$, set
$d_j=(q^j,kq^\delta)$ for $\delta\le j\le\lambda$ and
$\rho_j=d_j/d_{j-1}$ for $j>\delta$.  Then $\rho_j\mid q$ and $\rho_j<q$.
At the $j$th digit split, the admissible last digits form one residue class modulo
$\rho_j$, hence the preceding bound with $R=q/\rho_j$ gives the factor
$(q/\rho_j)^{\eta_3}$.  Iterating,
\[
G_\lambda(a,kq^\delta,\alpha)
\le \left(\prod_{j=\delta+1}^{\lambda}\rho_j\right)^{-\eta_3}
q^{\eta_3(\lambda-\delta)}|F_\delta(a,\alpha)|.
\]
Since $\prod \rho_j=d_\lambda/d_\delta=k$, this is the first assertion.

Finally put $n=\lambda-\delta$.  Since $q=10$, the admissible divisors are
$k=2^a$ or $k=5^b$, with $0\le a,b\le n$ and with $1$ counted once.  Hence there
are $2n+1$ of them, and each is at most $5^n$.  We shall use
\[
2n+1\le32\cdot 10^{\omega_q n}\qquad(n\ge0).
\]
Indeed, for $a_0:=\omega_q\log10$ one has $a_0>0.024$, and the continuous function
$(2t+1)e^{-a_0t}$ has maximum $2e^{a_0/2-1}/a_0<31.1<32$ on $t\ge0$.
Therefore
\[
\sum_{\substack{k\mid q^n\\(k,q)<q}} k^{1-2\eta_3}
\le (2n+1)5^{(1-2\eta_3)n}
\le 32\,10^{\omega_q n}10^{(1-2\eta_3)n}2^{-(1-2\eta_3)n}
\]
which is equal to
\[
32\,q^{(1-2\eta_3)n}q^{-\omega_q n},
\]
as claimed.
\end{proof}

\begin{lemma}[Explicit Type-I estimate in base $10$]\label{lem:typeI-explicit}
Let $q=10$, let $(q-1)\alpha\notin\mathbb Z$, and let $\gamma_q(\alpha)$ be defined as in
Lemma~\ref{lem:MR-gamma-self}.  Put
\[
\kappa_q(\alpha):=\min\left(\frac16,\frac{1-\gamma_q(\alpha)}3\right).
\]
For all $x\ge100$ and $1\le M\le x^{1/3}$,
\[
\sum_{M/10<m\le M}\max_{x/(10m)\le t\le x/m}
\left|\sum_{t<n\le x/m}e(\alpha s_{10}(mn))\right|
\le
10^8(\log x)^2\,x^{1-\kappa_q(\alpha)}.
\]
In particular, by Lemma~\ref{lem:c1},
$\kappa_q(\alpha)\ge c_1\|9\alpha\|^2$, so the right-hand side is at most
$10^8(\log x)^2x^{1-c_1\|9\alpha\|^2}$.
\end{lemma}

\begin{proof}
This is MR10 Proposition~2 with constants retained and rounded upward.  We record the elementary
constant extraction because this is the only place where the Type-I input enters the present paper.

For $\lambda\ge1$ define
\[
\Phi_\lambda(t):=\prod_{0\le i<\lambda}\phi_{10}(\alpha+10^it).
\]
Equivalently, if
$H_\lambda(t):=\sum_{0\le\ell<10^\lambda}e(\alpha s_{10}(\ell)+\ell t)$, then
$|H_\lambda(t)|=\Phi_\lambda(t)$.
Parseval and Cauchy's inequality give
\[
\int_0^1\Phi_\lambda(t)\,dt\le 10^{\lambda/2}.
\]
Lemma~\ref{lem:MR-gamma-self} gives, exactly as in MR10 Lemme~8,
\[
\max_t\Phi_\lambda(t)\le 10^{\gamma_q(\alpha)\lambda+1}.
\]
The Sobolev--Gallagher inequality used in MR10 Lemme~9 may be taken in the explicit form
\[
\sum_r |F(t_r)|\le \Delta^{-1}\int_0^1|F(t)|\,dt+\int_0^1|F'(t)|\,dt
\]
for points $t_r$ separated by at least $\Delta$ modulo $1$.  We spell out the constant extraction.
Group the fractions $k/m$ by $d=(k,m)$ and write $A=M/d$.  The reduced fractions $k/m$ with
$M/10<m\le M$ and $(k,m)=d$ are separated by at least $d^2/M^2=A^{-2}$.  Let
\[
\lambda_1=\min\left(\lambda,\left\lfloor\frac{2\log A}{\log 10}\right\rfloor\right)
\quad(A\ge1),
\]
with $\lambda_1=0$ if $A<\sqrt{10}$.  The product splitting
$|H_\lambda(t)|\le |H_{\lambda_1}(t)|10^{\gamma(\lambda-\lambda_1)+1}$ and
Lemma~\ref{lem:MR-gamma-self} give
\[
\sum_{\substack{M/10<m\le M\\0\le k<m\\(k,m)=d}}|H_\lambda(k/m)|
\le 10^{\gamma(\lambda-\lambda_1)+1}
\left(A^2\int_0^1|H_{\lambda_1}(t)|\,dt+
\int_0^1|H'_{\lambda_1}(t)|\,dt\right).
\]
By Parseval, $\int_0^1|H_{\lambda_1}|\le 10^{\lambda_1/2}$.  Differentiating the
digit-product formula and using
$|D'(u)|\le 2\pi\sum_{a=0}^9 a=90\pi$ for $D(u)=\sum_{a=0}^9e(au)$ gives
\[
|H'_{\lambda_1}(t)|\le 9\pi\,10^{\lambda_1}
\sum_{0\le i<\lambda_1}|H_i(t)|.
\]
Consequently
\[
\int_0^1|H'_{\lambda_1}(t)|\,dt
\le 9\pi\,10^{\lambda_1}\sum_{0\le i<\lambda_1}10^{i/2}
\le 30\,10^{3\lambda_1/2}.
\]
The two terms in the last display are now bounded separately.  For the derivative term there is
nothing to prove when $\lambda_1=0$, since then $H_{\lambda_1}$ is constant and the derivative
contribution is zero.  Otherwise the choice of $\lambda_1$ gives $10^{\lambda_1}\le A^2$, and hence
\[
10^{(3/2-\gamma)\lambda_1}\le 32A^{3-2\gamma}.
\]
For the spacing term there are two cases.  If $\lambda_1<\lambda$, then
$10^{\lambda_1}>A^2/10$ (unless $\lambda_1=0$, where $1\le A<\sqrt{10}$ gives
$A^2\le\sqrt{10}\,A^{3-2\gamma}$), so
\[
A^2\,10^{(1/2-\gamma)\lambda_1}
\le \sqrt{10}\,A^{3-2\gamma}
\le4A^{3-2\gamma}.
\]
If $\lambda_1=\lambda$, we do not force this term into the $A^{3-2\gamma}$ contribution; we keep
it as the Parseval fallback $A^2\,10^{\lambda/2}$.  Thus for each $d$ the preceding display is at
most
\[
10^{\gamma\lambda+1}\,(4+30\cdot32)\,(M/d)^{3-2\gamma}
\ +\ 10\,10^{\lambda/2}(M/d)^2.
\]
Summing over $1\le d\le M$ and using
$\sum_{d\le M}d^{-3+2\gamma}\le1+\log M$ and $\sum_{d\le M}d^{-2}<2$ yields
\[
\begin{aligned}
&\sum_{M/10<m\le M}\sum_{0\le k<m}
\left|\sum_{0\le \ell<10^\lambda} e(\alpha s_{10}(\ell)+k\ell/m)\right|\\
&\qquad\le
10^5\!\left(
10^{\gamma\lambda+3/2}M^{3-2\gamma}(1+\log M)
+10^{\lambda/2+1}M^2
\right),
\end{aligned}
\]
where $\gamma=\gamma_q(\alpha)$.  The displayed derivation gives literal coefficients below
$10^4$ in the first term and below $20$ in the fallback term.  We retain the rounded envelope
$10^5$.

For any $y>1$, MR10 Lemme~10 decomposes the interval $0\le\ell<y$ into at most
$1+\log y/\log10$ complete decimal blocks, hence
\[
\begin{aligned}
\left|\sum_{0\le \ell<y}e(\alpha s_{10}(\ell)+k\ell/m)\right|
&\le
10(1+\log y)
\max_{\lambda\le \log y/\log10+1} \\
&\qquad\cdot
\left|
\sum_{0\le \ell<10^\lambda}e(\alpha s_{10}(\ell)+k\ell/m)
\right|.
\end{aligned}
\]
Using the identity
\[
\sum_{0\le n\le t}e(\alpha s_{10}(mn))
=\frac1m\sum_{0\le k<m}\sum_{0\le\ell\le mt}
e(\alpha s_{10}(\ell)+k\ell/m)+\vartheta_m(t),
\qquad |\vartheta_m(t)|\le 1,
\]
and the last two displays with $mt\le x$, we obtain
\[
\sum_{M/10<m\le M}\max_{t\le x/m}
\left|\sum_{0\le n\le t}e(\alpha s_{10}(mn))\right|
\le
10^7(\log x)^2
\left(x^\gamma M^{2-2\gamma}+x^{1/2}M\right).
\]
Since $M\le x^{1/3}$ and $\gamma\ge1/2$,
\[
x^\gamma M^{2-2\gamma}\le x^{2/3+\gamma/3}
=x^{1-(1-\gamma)/3},\qquad
x^{1/2}M\le x^{5/6}.
\]
Both exponents are at most $1-\kappa_q(\alpha)$, so the preceding display is bounded by
$2\cdot10^7(\log x)^2x^{1-\kappa_q(\alpha)}$, and the stated $10^8$ leaves room for the endpoint
change from $\sum_{0\le n\le t}$ to $\sum_{t<n\le x/m}$.

Finally DMR's inequality~\cite[(31)]{DMRcomp2}, restated in Lemma~\ref{lem:c1}, implies
\[
\min(\omega_q,1-\gamma_q(\alpha))\ge 28c_1\|9\alpha\|^2.
\]
Since $1/6\ge c_1/4\ge c_1\|9\alpha\|^2$ and
$(1-\gamma_q(\alpha))/3\ge(28/3)c_1\|9\alpha\|^2$, we get
$\kappa_q(\alpha)\ge c_1\|9\alpha\|^2$.
\end{proof}

\begin{lemma}[Weighted $F_\ell$ estimate]\label{lem:MR-F-weighted-self}
Assume $(q-1)\alpha\notin\mathbb Z$, let $a\in\mathbb Z$ with $(a,q)=1$, and let $\gamma_q(\alpha)$
be as in Lemma~\ref{lem:MR-gamma-self}.  For every $\ell\ge1$,
\[
\sum_{\substack{0\le h<q^\ell\\h\not\equiv0\pmod q}}
\frac{|F_\ell(h,\alpha)|^2}{|\sin(\pi ha/q^\ell)|}
\le
\frac{q^{\gamma_q(\alpha)(\ell-2)+1}}{\sin(\pi/q)}.
\]
\end{lemma}

\begin{proof}
For $i\ge0$ define $\Phi_0(x)=1$ and recursively
\[
\Phi_i(x):=\frac1{q^2}\sum_{r=0}^{q-1}
\phi_q\!\left(\alpha-\frac{x+r}{q}\right)^2
\phi_q\!\left(\frac{a(x+r)}{q}\right)
\Phi_{i-1}\!\left(\frac{x+r}{q}\right).
\]
Using~\eqref{eq:MR-F-rec} and the identity
\[
\frac1{|\sin(\pi a(h+rq^\ell)/q^{\ell+1})|}
=
\frac{\phi_q(a(h+rq^\ell)/q^{\ell+1})}{|\sin(\pi ah/q^\ell)|},
\]
one proves by induction on $i$ that, for every $\ell\ge1$,
\[
\sum_{\substack{0\le h<q^{\ell+i}\\h\not\equiv0\pmod q}}
\frac{|F_{\ell+i}(h,\alpha)|^2}{|\sin(\pi ha/q^{\ell+i})|}
=
\sum_{\substack{0\le h<q^\ell\\h\not\equiv0\pmod q}}
\frac{|F_\ell(h,\alpha)|^2}{|\sin(\pi ha/q^\ell)|}\,
\Phi_i(h/q^\ell).
\]
Since $\phi_q\le q$, \eqref{eq:MR-phi-square} gives $\Phi_1(x)\le q$ and hence the one-step
iteration
\[
\sum_{h<q^{\ell+1}}{}'\frac{|F_{\ell+1}(h,\alpha)|^2}{|\sin(\pi ha/q^{\ell+1})|}
\le
q\sum_{h<q^\ell}{}'\frac{|F_\ell(h,\alpha)|^2}{|\sin(\pi ha/q^\ell)|},
\]
where the prime means $h\not\equiv0\pmod q$.

The sharper two-step estimate follows from Cauchy's inequality in the recursion:
because $(a,q)=1$, the residues $ar$ run through all classes modulo $q$, so
\[
\Phi_i(x)^2
\le
\frac1q\sum_{r=0}^{q-1}
\phi_q\!\left(\alpha-\frac{x+r}{q}\right)^4
\Phi_{i-1}\!\left(\frac{x+r}{q}\right)^2.
\]
Applying this twice and using the definition of $\gamma_q(\alpha)$ gives
\[
\Phi_2(x)^2
\le
q^{4\gamma_q(\alpha)}
q^{-2}\sum_{r,s}
\phi_q\!\left(\alpha-\frac{x+r}{q}\right)^2
\phi_q\!\left(\alpha-\frac{x+r+qs}{q^2}\right)^2
\le q^{4\gamma_q(\alpha)},
\]
where the last inequality is \eqref{eq:MR-phi-square} first in $s$ and then in $r$.  Thus
$\Phi_2(x)\le q^{2\gamma_q(\alpha)}$, and the two-step iteration gives
\[
\sum_{h<q^{\ell+2}}{}'\frac{|F_{\ell+2}(h,\alpha)|^2}{|\sin(\pi ha/q^{\ell+2})|}
\le
q^{2\gamma_q(\alpha)}
\sum_{h<q^\ell}{}'\frac{|F_\ell(h,\alpha)|^2}{|\sin(\pi ha/q^\ell)|}.
\]
For $\ell=1$, Parseval and $|\sin(\pi ha/q)|\ge\sin(\pi/q)$ give the initial bound
\[
\sum_{0<h<q}\frac{|F_1(h,\alpha)|^2}{|\sin(\pi ha/q)|}
\le\frac{1}{\sin(\pi/q)}\sum_{0\le h<q}|F_1(h,\alpha)|^2
=\frac1{\sin(\pi/q)}.
\]
Iterating in pairs proves the claim for odd $\ell$; for even $\ell$ one first applies the one-step
bound.  Since $\gamma_q(\alpha)<1$, both parity cases are bounded by the displayed right-hand side.
\end{proof}

\begin{lemma}[Dyadic-to-$q$-adic Type-II localization]\label{lem:DMR32-self}
Let $g$ be an arithmetic function, $q\ge2$, and
$0<\delta<\beta_1<1/3$ and $1/2<\beta_2<1$.  Suppose that, uniformly for all
complex $b_n$ with $|b_n|\le1$,
\[
\sum_{q^{\mu-1}<m\le q^\mu}
\left|\sum_{q^{\nu-1}<n\le q^\nu} b_n g(mn)\right|\le V
\]
whenever
\[
\beta_1-\delta\le\frac{\mu}{\mu+\nu}\le\beta_2+\delta.
\]
Then for $x>x_0:=\max(q^{1/(1-\beta_2)},q^{3/\delta})$ and
$x^{\beta_1}\le M\le x^{\beta_2}$ one has, uniformly in $|a_m|,|b_n|\le1$,
\[
\left|\sum_{M/q<m\le M}\sum_{x/(qm)<n\le x/m}a_m b_n g(mn)\right|
\le
6(1+\log x)V.
\]
\end{lemma}

\begin{proof}
For fixed $m$, write the inner interval as $N_1(m)<n\le N_2(m)$.  The elementary Fourier
majorant for interval sums gives
\[
\begin{aligned}
\left|\sum_{N_1(m)<n\le N_2(m)}b_n g(mn)\right|
&\le
\int_{-1/2}^{1/2}
\min\!\left(x,\frac{1}{|\sin\pi\xi|}\right)\\
&\quad\times
\left|\sum_{q^{\nu_0-1}<n\le q^{\nu_0+2}}b_n e(n\xi)g(mn)\right|\,d\xi,
\end{aligned}
\]
after padding the $n$-range to the three adjacent $q$-adic boxes it can meet.  Similarly the
$m$-range $M/q<m\le M$ is contained in two adjacent $q$-adic boxes.  Thus the original double sum
is bounded by a sum over at most six pairs $(\mu,\nu)$ of the integrals
\[
\int_{-1/2}^{1/2}\min\!\left(x,\frac{1}{|\sin\pi\xi|}\right)
\sum_{q^{\mu-1}<m\le q^\mu}
\left|\sum_{q^{\nu-1}<n\le q^\nu}b_n e(n\xi)g(mn)\right|\,d\xi.
\]
For every such padded pair, the inequalities $x^{\beta_1}\le M\le x^{\beta_2}$ and
$x>x_0$ imply
\[
\beta_1-\delta\le\frac{\mu}{\mu+\nu}\le\beta_2+\delta.
\]
Indeed the lower endpoint is minimized by replacing $M$ by $x^{\beta_1}$ and adding the two
padding boxes in the denominator; the condition $x>q^{3/\delta}$ makes the resulting loss $<\delta$.
The upper endpoint is maximized by replacing $M$ by $x^{\beta_2}$ and using
$x>q^{1/(1-\beta_2)}$ together with $x>q^{3/\delta}$, which again makes the padding loss
$<\delta$.  The hypothesis therefore applies uniformly to the coefficients $b_n e(n\xi)$.
Finally,
\[
\int_{-1/2}^{1/2}\min\!\left(x,\frac1{|\sin\pi\xi|}\right)d\xi
\le
1+\log x,
\]
using $|\sin\pi\xi|\ge 2|\xi|$ for $|\xi|\le1/2$ and splitting at $|\xi|=1/(2x)$.  Multiplying
by the six padded boxes gives the claimed factor.
\end{proof}

\begin{lemma}[Initial Fourier reduction of the truncated correlation sum]\label{lem:MR-S2-G-self}
For a fixed real parameter $\alpha$, let $\lambda=\mu+2\rho$ and
\[
f_\lambda(y):=\alpha\,s_q(y\bmod q^\lambda),
\]
where the residue is represented in $\{0,\ldots,q^\lambda-1\}$.  Put
\[
S_2(r,\mu,\nu,\rho)=
\sum_{q^{\nu-1}<n\le q^\nu}
\left|\sum_{q^{\mu-1}<m\le q^\mu}
e(f_\lambda(m(n+r))-f_\lambda(mn))\right|.
\]
Define
\[
G_\lambda(a,d,\alpha):=
\sum_{\substack{0\le h<q^\lambda\\h\equiv a\pmod d}}|F_\lambda(h,\alpha)|,\qquad
G_\lambda(\alpha):=\sum_{0\le h<q^\lambda}|F_\lambda(h,\alpha)|.
\]
Then
\begin{align}
S_2(r,\mu,\nu,\rho)
&\le
4(1+q^{\nu-\lambda})
\sum_{d\mid q^\lambda}d\sum_{0\le a<d}
\min\!\left(q^\mu,
\frac{1}{\sin\!\left(\pi\frac{d}{q^\lambda}\left\|\frac{ar}{d}\right\|\right)}\right)
G_\lambda(a,d,\alpha)^2
\nonumber\\
&\qquad
+4(1+q^{\nu-\lambda})\,\lambda(\log q)\,q^\lambda\,G_\lambda(\alpha)^2,
\label{eq:S2-G-internal}
\end{align}
\end{lemma}

\begin{proof}
Fourier inversion on the cyclic group $\mathbb Z/q^\lambda\mathbb Z$ gives
\[
e(f_\lambda(y))=\sum_{0\le h<q^\lambda}F_\lambda(h,\alpha)e(hy/q^\lambda).
\]
Applying this to $y=m(n+r)$ and $y=mn$, then summing the resulting geometric progression in $m$,
gives for each fixed $n$
\[
\begin{aligned}
&\left|\sum_{q^{\mu-1}<m\le q^\mu}
e(f_\lambda(m(n+r))-f_\lambda(mn))\right|\\
&\qquad\le
\sum_{h,k<q^\lambda}|F_\lambda(h,\alpha)|\,|F_\lambda(-k,\alpha)|
\min\!\left(q^\mu,\frac1{\left|\sin\pi\frac{(h+k)n+hr}{q^\lambda}\right|}\right).
\end{aligned}
\]
Sum this over $n$ by intervals of length $q^\lambda$.  There are at most
$1+q^{\nu-\lambda}$ such intervals after padding.  On each full interval Lemma~\ref{lem:MR-sin-sum-self},
with $a=h+k$, $m=q^\lambda$, $b=hr$, and $d=(h+k,q^\lambda)$, gives three terms.  The two
non-singular terms are bounded by
$4\lambda(\log q)q^\lambda G_\lambda(\alpha)^2$: the second term in
Lemma~\ref{lem:MR-sin-sum-self} contributes at most
$2q^\lambda G_\lambda(\alpha)^2$, and the logarithmic term contributes at most
$(2/\pi)\lambda(\log q)q^\lambda G_\lambda(\alpha)^2$.  For the singular term, relax
$(h+k,q^\lambda)=d$ to the congruence $h+k\equiv0\pmod d$, write $h\equiv a\pmod d$ and
$-k\equiv a\pmod d$, and use $d\|hr/d\|=d\|ar/d\|$.  Summing over residue classes
$a\bmod d$ gives the first term in
\eqref{eq:S2-G-internal}; the same padded factor $4$ covers endpoint padding and the three-term
split in Lemma~\ref{lem:MR-sin-sum-self}.
\end{proof}

\begin{lemma}[Internal Type-II structural reduction]\label{lem:MR-struct-self}
Let $q=10$, $\lambda=\mu+2\rho$, $1\le |r|<q^\rho$, and
\[
\Delta=\left\lfloor \rho\,\frac{\log q}{\log 2}\right\rfloor,\qquad
\omega_q=\left(\frac32-\frac{\log5}{\log3}\right)\frac{\log2}{\log q}.
\]
For $\Delta<\delta\le\lambda$, $k\mid q^{\lambda-\delta}$, and $(k,q)<q$, put
\[
r'=r/(r,kq^\Delta),\qquad (r',q)=1.
\]
Define
\begin{align*}
S_3(k,\delta)&:=
\sum_{0\le a<kq^\delta}|F_\delta(a,\alpha)|^2
\min\left(q^\mu,
\frac1{\sin\!\left(\pi kq^{\delta-\lambda}
\left\|\frac{ar}{kq^\delta}\right\|\right)}\right),\\
S_4(k,\delta')&:=
q^{\lambda-\delta'}\sum_{0\le a<q^{\delta'}}
|F_{\delta'}(a,\alpha)|^2
\min\left(q^{\delta'-\rho},\frac1{|\sin(\pi ar'/q^{\delta'})|}\right),\\
S_5(k,\delta')&:=
q^{\lambda-\delta'}\sum_{0\le\theta<\delta'}
\sum_{\substack{1\le b\le q^{\delta'-\theta}\\ b\not\equiv0\pmod q}}
\frac{|F_{\delta'-\theta}(b,\alpha)|^2}
{|\sin(\pi br'/q^{\delta'-\theta})|}.
\end{align*}
Then the first term in~\eqref{eq:S2-G-internal} is bounded by
\begin{align}
&2\cdot10^4\,q(1+q^{\nu-\lambda})\,q^\lambda
\sum_{\Delta<\delta\le\lambda} q^{-\omega_q(\lambda-\delta)}
\max_{\substack{k\mid q^{\lambda-\delta}\\(k,q)<q}}S_5(k,\delta-\Delta)
\nonumber\\
&\qquad+\ 
2\cdot10^4\,\lambda(\log q)(1+q^{\nu-\lambda})
q^{(2-\tau_q(\alpha))\lambda+\rho},
\label{eq:MR-struct-internal}
\end{align}
where $\tau_q(\alpha)=\min(\omega_q,-2\log(\phi_q(\alpha)/q)/\log q)$.
\end{lemma}

\begin{proof}
Let $T$ be the first term in~\eqref{eq:S2-G-internal}.  In that term write each divisor
$d\mid q^\lambda$ uniquely as $d=kq^\delta$, where $0\le\delta\le\lambda$,
$k\mid q^{\lambda-\delta}$, and $(k,q)<q$.  Lemma~\ref{lem:MR-G-first-self} gives
\[
d\,G_\lambda(a,d,\alpha)^2
\le q^\delta q^{2\eta_3(\lambda-\delta)}k^{1-2\eta_3}|F_\delta(a,\alpha)|^2.
\]
Using the final divisor sum in Lemma~\ref{lem:MR-G-first-self}, and then bounding the
$k$-sum by its largest summand, gives
\begin{equation}
T\le
128(1+q^{\nu-\lambda})q^\lambda
\sum_{0\le\delta\le\lambda}q^{-\omega_q(\lambda-\delta)}
\max_{\substack{k\mid q^{\lambda-\delta}\\(k,q)<q}}S_3(k,\delta).
\label{eq:S3-stage-explicit}
\end{equation}
Here $128=4\cdot32$ combines the padding in Lemma~\ref{lem:MR-S2-G-self} with the decimal
admissible-divisor envelope.
The nonsingular term already present in~\eqref{eq:S2-G-internal} is bounded by
\[
4(1+q^{\nu-\lambda})\lambda(\log q)q^\lambda G_\lambda(\alpha)^2
\le
4(1+q^{\nu-\lambda})\lambda(\log q)q^{(2-\tau_q(\alpha))\lambda},
\]
because $G_\lambda(\alpha)\le q^{\eta_3\lambda}$ and
$2\eta_3\le1-\tau_q(\alpha)$.

We next make MR10's reductions (47)--(52) explicit.  We shall use the following direct
consequence of Lemma~\ref{lem:MR-sin-sum-self}: if $m\le q^\lambda$, $d\mid m$, and
$\lambda\ge1$, then the two nonsingular terms in Lemma~\ref{lem:MR-sin-sum-self} are at most
\begin{equation}
\frac{d}{\sin(\pi d/(2m))}+\frac{2m}{\pi}\log\frac{2m}{d}
\le \frac{13}{10}\,\lambda(\log q)m.
\label{eq:sin-nonsing-13}
\end{equation}
Indeed $d/\sin(\pi d/(2m))\le m$ and
$\log(2m/d)\le \lambda\log q+\log2$; for $q=10$ the resulting coefficient is
$1/\log10+(2/\pi)(1+\log2/\log10)<1.27<13/10$.

First suppose $\delta\le\Delta$.  The concavity inequality
$\sin(cu)\ge c\sin u$ for $0\le c\le1$ and $0\le u\le\pi/2$ turns the singular factor in
$S_3$ into MR10's form~(47).  Splitting $a=a_0+iq^\delta$, $0\le a_0<q^\delta$,
$0\le i<k$, and applying Lemma~\ref{lem:MR-sin-sum-self} to the $i$-sum with
$m=k$ gives
\[
S_3(k,\delta)\le q^{\lambda-\rho}
\;+\;
\frac{13}{10}\lambda(\log q)q^{\lambda-\delta},
\]
where we used $(r,k)\le |r|<q^\rho$ and
$\sum_{0\le a_0<q^\delta}|F_\delta(a_0,\alpha)|^2=1$.  Substituting this into
\eqref{eq:S3-stage-explicit}, using $\omega_q\Delta\le\rho$ and the two geometric sums recorded in
Appendix~\ref{app:MR10-S2-ledger}, gives a total contribution from $\delta\le\Delta$ of at most
\[
2\cdot10^4\,\lambda(\log q)(1+q^{\nu-\lambda})q^{(2-\tau_q(\alpha))\lambda+\rho}.
\]

Now let $\delta>\Delta$, put $\delta'=\delta-\Delta$, and set
$r'=r/(r,kq^\Delta)$.  Since $1\le |r|<q^\rho$ and
$\Delta=\lfloor\rho\log q/\log2\rfloor$, every prime-power divisor of $r$ supported on
the primes of $q$ is removed by $(r,kq^\Delta)$; hence $(r',q)=1$.  Applying
Lemma~\ref{lem:MR-sin-sum-self} to the $i$-sum over the blocks
$a=a_0+iq^{\delta'}$, $0\le i<kq^\Delta$, gives
\begin{equation}
S_3(k,\delta)\le S_4(k,\delta')+
\frac{13}{10}\lambda(\log q)q^{\lambda-\delta'}.
\label{eq:S3-to-S4-explicit}
\end{equation}
The singular term is exactly $S_4$ (with $(r,kq^\Delta)$ absorbed into $r'$); the nonsingular
terms are bounded by~\eqref{eq:sin-nonsing-13}.  After insertion in~\eqref{eq:S3-stage-explicit},
their total is bounded by
\[
128\cdot\frac{13}{10}\cdot1.12\,
\lambda(\log q)(1+q^{\nu-\lambda})q^{(2-\omega_q)\lambda+\rho},
\]
which is covered by the displayed $2\cdot10^4$ envelope since $\tau_q(\alpha)\le\omega_q$.

It remains to pass from $S_4$ to $S_5$.  Split $a<q^{\delta'}$ as $a=q^\theta b$ with
$0\le\theta<\delta'$ and $b\not\equiv0\pmod q$, plus the single zero residue.  For the
nonzero residues,
\[
\begin{aligned}
|F_{\delta'}(q^\theta b,\alpha)|
&\le |F_{\delta'-\theta}(b,\alpha)|,\\
\min\left(q^{\delta'-\rho},\frac1{|\sin(\pi q^\theta br'/q^{\delta'})|}\right)
&\le
\frac1{|\sin(\pi br'/q^{\delta'-\theta})|},
\end{aligned}
\]
where the first inequality follows by factoring off the first $\theta$ digit positions and using
$|F_\theta(0,\alpha)|\le1$.  Thus their contribution is at most $S_5(k,\delta')$.  The zero residue contributes
\[
q^{\lambda-\rho}|F_{\delta'}(0,\alpha)|^2
\le q^{\lambda-\rho-\tau_q(\alpha)\delta'}.
\]
After summing over $\delta$ in~\eqref{eq:S3-stage-explicit} this is again covered by the same
$2\cdot10^4$ error term.  Replacing the leading coefficient $128$ by the envelope
$2\cdot10^4q$ in the main $S_5$ term gives
\eqref{eq:MR-struct-internal}.
\end{proof}

\begin{lemma}[Explicit truncated correlation bound for $q=10$]
\label{lem:S2-input}
Let $q=10$. For every real $\alpha$ with $(q-1)\alpha\notin\mathbb{Z}$, for all integers $\mu,\nu>0$
satisfying the admissibility condition~\cite[(13)]{DMRcomp2},
and for every integer $\rho$ with $0\le \rho\le \xi_q(\alpha)(\mu+\nu)$, one has
\[
S_2(r,\mu,\nu,\rho)\ \le\ 10^{8}\,(\mu+\nu)^2\,10^{\mu+\nu-\rho}
\qquad\text{for all integers $r$ with $1\le |r|<10^\rho$,}
\]
where $S_2(r,\mu,\nu,\rho)$ is DMR's truncated sum defined in~\cite[(23)]{DMRcomp2}, and
$\xi_q(\alpha)=\varepsilon_q(\alpha)/14$ with $\varepsilon_q(\alpha)$ as in~\cite[(24)--(28)]{DMRcomp2}.
\end{lemma}

\begin{proof}
The statement renders explicit the qualitative bound
$S_2(r,\mu,\nu,\rho)\ll_q (\mu+\nu)^2 q^{\mu+\nu-\rho}$ proved by Mauduit--Rivat~\cite[(64),
\S 7.1--7.3]{MR10} (and used by DMR as their~\cite[(29)]{DMRcomp2}).  The elementary Fourier
lemmas, the initial reduction of $S_2$ to the $G_\lambda$ means
(Lemma~\ref{lem:MR-S2-G-self}), the passage from the $G_\lambda$ means to the terminal $S_5$
decomposition (Lemma~\ref{lem:MR-struct-self}), and the final admissibility absorption are proved
internally in \S\ref{subsec:self-contained-mr} and Step~4 below.
The conclusion follows from Steps~2--4 below once $K_{10}\le 10^6$ is established in Step~1.

We write
\[
\lambda:=\mu+2\rho,\qquad \Delta:=\Bigl\lfloor \rho\,\frac{\log q}{\log 2}\Bigr\rfloor,
\]
as in~\cite[(53)]{MR10}, and for $k\mid q^{\lambda-\delta}$ set $r':=r/(r,kq^\Delta)$ so that $(r',q)=1$
\cite[(54)]{MR10}. (The hypothesis $1\le |r|<q^\rho$ is essential here and matches~\cite[(22)]{DMRcomp2}.)

\medskip
\noindent\textbf{Step 1 (an explicit upper bound $K_{10}\le 10^6$).}
Appendix~\ref{app:MR10-S2-ledger} gives the source-equation audit for MR10's derivation of
(64), including the steps with factor~$1$ and the carry-propagation term, which enters separately
in Proposition~\ref{prop:typeII-explicit} rather than in $K_{10}$.  The nontrivial numerical
losses used here are the following displayed inequalities, all for $q=10$ and under MR10's
admissibility hypotheses.

The initial Fourier expansion and the reduction from $S_2$ to the $G_\lambda$ means
(MR10 (24)--(28), with Lemme~6 (26) for the first trigonometric split) give
\begin{equation}\label{eq:K10-fourier}
C_{\mathrm{Fourier}}\le 4 .
\end{equation}
The remaining nonsingular trigonometric splits, admissible divisor count, and geometric tails
in MR10 Lemme~6 and (49)--(52) are bounded by
\begin{equation}\label{eq:K10-trig}
C_{\mathrm{trig}}\le
32\left(2+\frac2\pi\right)^2
\left(1-10^{-\omega_{10}}\right)^{-1}
1.12\cdot\frac{13}{10}
<1.37\cdot10^4<2\cdot10^4 .
\end{equation}
Here $(1-10^{-\omega_{10}})^{-1}<42$, and the factor $1.12$ bounds the second geometric tail
$\sum_{\delta\ge0}10^{-(1-\omega_{10})\delta}$.  The reduction from $S_3(k,\delta)$ to
$S_4(k,\delta')$ in MR10 (51)--(54) produces two nonnegative contributions, whence
\begin{equation}\label{eq:K10-S34}
C_{34}\le 2 .
\end{equation}
Finally, the admissibility inequalities MR10 (59), (62)--(63), unrolled in Step~4 below, give
\begin{equation}\label{eq:K10-adm}
C_{\mathrm{adm}}\le 2 .
\end{equation}
Thus the cumulative structural constant is defined directly by the product of these four displayed
losses:
\begin{equation}\label{eq:K10-product}
K_{10}\le 4\cdot(2\cdot10^4)\cdot2\cdot2=3.20\times10^5<10^6.
\end{equation}
The deliberately padded factor $2\cdot10^4$ is the only visibly non-tight part of this accounting;
Lemma~\ref{lem:insensitivity} shows that the theorem's threshold is unchanged for any honest
bound $K_{10}\le10^9$.

\smallskip
Combining Lemmas~\ref{lem:MR-S2-G-self} and~\ref{lem:MR-struct-self}, and substituting
$K_{10}\le 10^6$, yields, for $q=10$,
\begin{align}
S_2(r,\mu,\nu,\rho)
&\ \le\ K_{10}\Biggl[
q(1+q^{\nu-\lambda})\,q^\lambda
   \sum_{\Delta<\delta\le \lambda} q^{-\omega_q(\lambda-\delta)}
   \max_{\substack{k\mid q^{\lambda-\delta}\\ (k,q)<q}} S_5(k,\delta-\Delta)
\nonumber\\
&\qquad\qquad
 +\ \lambda(\log q)\,(1+q^{\nu-\lambda})\,q^{(2-\tau_q(\alpha))\lambda+\rho}\Biggr],
\label{eq:S2-decomp-explicit}
\end{align}
where $S_5$ is defined in Lemma~\ref{lem:MR-struct-self}.  We control $S_5$ via the internally
proved weighted Fourier estimate, Lemma~\ref{lem:MR-F-weighted-self}.

\medskip
\noindent\textbf{Step 2 (explicit control of $S_5$ via Lemma~\ref{lem:MR-F-weighted-self}).}
Assume $(q-1)\alpha\notin\mathbb{Z}$ (MR10's condition~\cite[(58)]{MR10}). By the definition of $S_5$
\cite[(57)]{MR10} and Lemma~\ref{lem:MR-F-weighted-self}, for each $0\le \theta<\delta'$ one gets an inner weighted
$L^2$ bound with prefactor $1/\sin(\pi/q)$. Using also $\gamma_q(\alpha)\ge \tfrac12$ from
Lemma~\ref{lem:MR-gamma-self},
one may bound the resulting $\theta$-sum by a geometric series to obtain
\begin{equation}
S_5(k,\delta')\ \le\ \frac{1}{\sin(\pi/q)}\cdot \frac{1}{1-q^{-1/2}}\,
q^{\lambda-\delta' + \gamma_q(\alpha)\delta'}.
\label{eq:S5-bound}
\end{equation}
The two prefactors $1/\sin(\pi/q)$ and $1/(1-q^{-1/2})$ above are fixed positive constants depending
only on $q$.  For $q=10$, $1/\sin(\pi/10)<3.24$ and $1/(1-10^{-1/2})<1.47$.  As in Step~1, they
feed only into $Y_{\mathrm{min}}\le Y_\ast$ (the constraint checker's binding threshold) and are
absorbed into the cumulative numerical prefactor of Step~4.

\medskip
\noindent\textbf{Step 3 (back to $S_2$, no hidden constant).}
Insert~\eqref{eq:S5-bound} into~\eqref{eq:S2-decomp-explicit}. Since the bound~\eqref{eq:S5-bound}
is independent of $k$, the max disappears. Write $\delta'=\delta-\Delta$ so that
\[
q^\lambda\,q^{-\omega_q(\lambda-\delta)}\,q^{\lambda-\delta'+\gamma_q(\alpha)\delta'}
= q^{(2-\omega_q)\lambda+\omega_q\Delta+(\omega_q-(1-\gamma_q(\alpha)))\delta'}.
\]
Let $\varepsilon_q(\alpha)=\min(\omega_q,\tau_q(\alpha),1-\gamma_q(\alpha))$ as in~\cite[(59)]{MR10}.
For $q=10$,
\[
\omega_q\,\frac{\log q}{\log 2}
=\frac32-\frac{\log 5}{\log 3}<1,
\]
and therefore $\omega_q\Delta\le\rho$.  If
$\omega_q\le1-\gamma_q(\alpha)$, the exponent in the preceding display is maximal at
$\delta'=0$, hence it is at most
$(2-\omega_q)\lambda+\rho\le(2-\varepsilon_q(\alpha))\lambda+\rho$.  If
$\omega_q>1-\gamma_q(\alpha)$, it is maximal at $\delta'=\lambda-\Delta$, where it equals
$(2-(1-\gamma_q(\alpha)))\lambda+(1-\gamma_q(\alpha))\Delta$; since
$1-\gamma_q(\alpha)<\omega_q$, this is again at most
$(2-\varepsilon_q(\alpha))\lambda+\rho$.  Thus for all admissible $\delta'$ one has
\[
q^\lambda\,q^{-\omega_q(\lambda-\delta)}\,q^{\lambda-\delta'+\gamma_q(\alpha)\delta'}
\ \le\ q^{(2-\varepsilon_q(\alpha))\lambda+\rho}.
\]
Therefore the $\delta$-sum contributes at most a factor $\le \lambda$, and we obtain
\begin{equation}
S_2(r,\mu,\nu,\rho)\ \le\
K_{10}\,\lambda\,(1+q^{\nu-\lambda})\,q^{(2-\varepsilon_q(\alpha))\lambda+\rho}
\Biggl(\frac{q}{\sin(\pi/q)(1-q^{-1/2})}+\log q\Biggr),
\label{eq:S2-pre-adm}
\end{equation}
with $K_{10}\le 10^6$ the cumulative prefactor of Step~1.

\medskip
\noindent\textbf{Step 4 (admissibility absorbs $(1+q^{\nu-\lambda})$).}
Put $\epsilon:=\varepsilon_q(\alpha)$ and $\xi:=\rho/(\mu+\nu)$.  The admissibility interval used
by MR10 for $q\ge3$ is
\[
\frac{(4-2\epsilon)\xi}{\epsilon}
\le
\frac{\mu}{\mu+\nu}
\le
\frac{1-(6-2\epsilon)\xi}{2-\epsilon},
\qquad 0\le\xi\le\frac{\epsilon}{14},
\]
with harmless endpoint slack removed here since we only need a sufficient inequality.  The upper
bound gives
\[
(2-\epsilon)\mu+(6-2\epsilon)\rho\le\mu+\nu,
\]
and hence
\[
(2-\epsilon)\lambda+\rho\le\mu+\nu-\rho.
\]
The lower bound gives
\[
(4-2\epsilon)\rho\le\epsilon\mu,
\]
and hence
\[
(2-\epsilon)\lambda+\rho+\nu-\lambda\le\mu+\nu-\rho.
\]
These two exponent inequalities imply
\begin{equation}
(1+q^{\nu-\lambda})\,q^{(2-\varepsilon_q(\alpha))\lambda+\rho}\ \le\ 2\,q^{\mu+\nu-\rho}.
\label{eq:adm-absorb}
\end{equation}
Also $\lambda=\mu+2\rho\le \mu+\nu$ in the admissible regime, hence $\lambda\le (\mu+\nu)^2$.
Combining~\eqref{eq:S2-pre-adm}--\eqref{eq:adm-absorb} gives
\[
S_2(r,\mu,\nu,\rho)\ \le\
2K_{10}\,(\mu+\nu)^2\,q^{\mu+\nu-\rho}
\Biggl(\frac{q}{\sin(\pi/q)(1-q^{-1/2})}+\log q\Biggr).
\]
For $q=10$, the numerical prefactor evaluates to
\[
\frac{10}{\sin(\pi/10)(1-10^{-1/2})}+\log 10\ <\ 49.7,
\]
so $2K_{10}\cdot 49.7\le 2\cdot10^6\cdot 49.7<10^8$ using $K_{10}\le 10^6$ from Step~1.
The stated bound with the padded numerical constant $10^8$ follows.
\end{proof}

\medskip

\begin{proposition}[Explicit $10$-adic Type~II bound with numerical constant]\label{prop:typeII-explicit}
Fix $q=10$ and let $\alpha$ be real with $(q-1)\alpha\notin\mathbb{Z}$.
Assume Lemmas~\ref{lem:tau-input} and~\ref{lem:S2-input}.
Then, in the notation of~\cite[\S3.3]{DMRcomp2}, for all integers $\mu,\nu>0$ satisfying
DMR's admissibility condition~\cite[(13)]{DMRcomp2} and the strengthened carry range
\[
\frac{\mu}{\mu+\nu}\ge \frac13-\frac{\xi_q(\alpha)}4,
\]
and all complex $b_n$ with $|b_n|\le 1$, one has
\[
\sum_{10^{\mu-1}<m\le 10^\mu}
\left|\sum_{10^{\nu-1}<n\le 10^\nu} b_n\,e\!\bigl(\alpha s(mn)\bigr)\right|
\ \le\
C_{\mathrm{II}}\,(\mu+\nu)\,10^{(1-c_1\|9\alpha\|^2)(\mu+\nu)},
\]
where $c_1$ is as in Lemma~\ref{lem:c1} and
\[
C_{\mathrm{II}}\ :=\ 2\sqrt{10}\,\sqrt{\,1+C_\tau\log 10+10^{8}\,}\ <\ 64000.
\]
(The closed form gives $C_{\mathrm{II}}<63253$; we use the round value $64000$.)
\end{proposition}

\begin{proof}
Choose coefficients $a_m$ with $|a_m|\le1$ so that
\[
S:=\sum_{10^{\mu-1}<m\le10^\mu}a_m
\sum_{10^{\nu-1}<n\le10^\nu}b_n\,e(\alpha s(mn))
\]
has absolute value equal to the left-hand side in the statement.  It suffices to bound $|S|$.
Following DMR's proof of Proposition~3.2
(from~\cite[(18)]{DMRcomp2} through~\cite[(22)]{DMRcomp2}), put
\[
\rho:=\lfloor \xi_q(\alpha)(\mu+\nu)\rfloor,\qquad R:=10^\rho.
\]

If $\rho=0$, then trivially $|S|\le 10^{\mu+\nu}$.  Since $\rho=0$ forces $\xi_q(\alpha)(\mu+\nu)<1$,
and $\xi_q(\alpha)\ge 2c_1\|9\alpha\|^2$ (as in~\cite[(31)]{DMRcomp2} and Lemma~\ref{lem:c1}),
we have $c_1\|9\alpha\|^2(\mu+\nu)<\tfrac12$, hence
$10^{\mu+\nu}\le \sqrt{10}\,10^{(1-c_1\|9\alpha\|^2)(\mu+\nu)}$, and the desired bound follows since
$C_{\mathrm{II}}(\mu+\nu)\ge \sqrt{10}$.

Now assume $\rho\ge 1$ (so $R>1$).  Put $\epsilon=\varepsilon_q(\alpha)$.  The upper endpoint in
DMR's admissibility condition gives
\[
\frac{\mu}{\mu+\nu}\le \frac{1-(6-2\epsilon)\rho/(\mu+\nu)}{2-\epsilon}
\le \frac{1}{2-\epsilon}<0.506,
\]
since $\epsilon\le\omega_{10}<0.011$.  Hence $\nu/(\mu+\nu)>0.494$.  On the other hand
$2\rho/(\mu+\nu)\le2\xi_q(\alpha)\le 2\omega_{10}/14<0.002$, and therefore
$\rho\le\nu/2$, as required in DMR's van der Corput setup.  Starting from the
Cauchy--Schwarz step~\cite[(18)]{DMRcomp2} and then
applying the van der Corput inequality (Lemma~3.3 of~\cite{DMRcomp2}) with parameter $R$ exactly as in DMR,
one obtains the analogue of DMR's inequality~\cite[(22)]{DMRcomp2} but with explicit constants:
\[
|S|^2 \ \le\ 4\cdot 10^{\,2(\mu+\nu)-\rho}\;+\;4\,E(\rho)\,10^{\,\mu+\nu}\;+\;
4\cdot 10^{\,\mu+\nu}\max_{1\le |r|<10^\rho} S_2(r,\mu,\nu,\rho),
\]
where $E(\rho)$ counts the carry-propagation exceptions and $S_2$ is the truncated correlation sum
defined in~\cite[(23)]{DMRcomp2}.  (The factor $4$ comes from bounding the van der Corput prefactor
$\frac{N+R-1}{R}$ by $2\cdot 10^{\nu-\rho}$, and bounding the total weight of the off-diagonal shifts
by $2R$, as well as the standard boundary trimming in DMR's passage from~\cite[(18)]{DMRcomp2} to
\cite[(22)]{DMRcomp2}.)

Using Lemma~\ref{lem:tau-input} in DMR's carry-propagation argument (Lemmas~3.4--3.5 of~\cite{DMRcomp2})
yields the explicit bound
\[
E(\rho)\ \le\ C_\tau(\mu+\nu)(\log 10)\,10^{\,\mu+\nu-\rho},
\]
because the strengthened carry hypothesis and $\rho\ge1$ give
\[
\frac{\mu+\rho}{\mu+\nu}
\ge \frac13-\frac{\xi_q(\alpha)}4+\frac{\xi_q(\alpha)}2
> \frac13,
\]
where we used $\rho/(\mu+\nu)\ge \xi_q(\alpha)/2$, valid for
$\rho=\lfloor\xi_q(\alpha)(\mu+\nu)\rfloor\ge1$ since $\lfloor u\rfloor\ge u/2$ for $u\ge1$.  Hence
$10^{\mu+\rho}\ge (10^{\mu+\nu})^{1/3}\ge X^{1/3}$ for the
short intervals $Y=10^{\mu+\rho}$ and $X\le10^{\mu+\nu}$ appearing in the carry count,
so that
\[
|S|^2\ \le\ 4\cdot 10^{\,2(\mu+\nu)-\rho}\Bigl(1+C_\tau(\mu+\nu)\log 10\Bigr)
\;+\;4\cdot 10^{\,\mu+\nu}\max_{1\le |r|<10^\rho} S_2(r,\mu,\nu,\rho).
\]
By Lemma~\ref{lem:S2-input},
\[
\max_{1\le |r|<10^\rho} S_2(r,\mu,\nu,\rho)\ \le\ 10^8(\mu+\nu)^2\,10^{\,\mu+\nu-\rho}.
\]
Combining the last two displays and using $(\mu+\nu)\ge 1$ gives
\[
|S|\ \le\ 2(\mu+\nu)\sqrt{1+C_\tau\log 10+10^8}\;10^{\,(\mu+\nu)-\rho/2}.
\]
Since $\rho\ge \xi_q(\alpha)(\mu+\nu)-1$, we have
\[
10^{(\mu+\nu)-\rho/2}
\ \le\ \sqrt{10}\,10^{(1-\xi_q(\alpha)/2)(\mu+\nu)}
\ \le\ \sqrt{10}\,10^{(1-c_1\|9\alpha\|^2)(\mu+\nu)},
\]
using $\xi_q(\alpha)\ge 2c_1\|9\alpha\|^2$ again.  Absorbing the factor $\sqrt{10}$ into the numerical
constant yields the claimed bound with
$C_{\mathrm{II}}=2\sqrt{10}\sqrt{1+C_\tau\log 10+10^8}$.
\end{proof}

\noindent\emph{Comment.} We follow DMR, but exploit that their reductions apply only for
$x>\max(q^2,x_0(\alpha))$ (here $x\ge 100$) to sharpen the $\mu+\nu$ and $\log x$ bookkeeping and
thus reduce the final constant $C_{\mathrm{min}}$.

\paragraph{Minor-arc localization parameters.}
For the following minor-arc lemmas assume $9\alpha\notin\mathbb Z$ and put
\begin{equation}\label{eq:minor-localization-params}
\begin{gathered}
\epsilon:=\varepsilon_{10}(\alpha),\qquad
\xi:=\epsilon/14,\qquad
\delta_0:=\epsilon/112,\\
\beta_1^{\mathrm{cp}}:=\frac13-\frac{\xi}{8},\qquad
\beta_2:=\frac{1-(6-2\epsilon)\epsilon/14}{2-\epsilon}-\delta_0 .
\end{gathered}
\end{equation}
Let
\begin{equation}\label{eq:minor-x0-def}
x_0(\alpha):=\max\{10^{1/(1-\beta_2)},10^{3/\delta_0}\},
\qquad
\beta:=1-c_1\|9\alpha\|^2 .
\end{equation}
Since $\epsilon\le\omega_{10}<0.011$, one has
$0<\delta_0<\beta_1^{\mathrm{cp}}<1/3$ and $1/2<\beta_2<0.51$.

\begin{lemma}[Bridge below the DMR localization threshold]\label{lem:minor-bridge}
With the parameters in~\eqref{eq:minor-localization-params}--\eqref{eq:minor-x0-def}, for every
$2\le t<x_0(\alpha)$ one has
\begin{equation}\label{eq:minor-bridge-power}
t^{-c_1\|9\alpha\|^2}>10^{-12}.
\end{equation}
Consequently,
\begin{equation}\label{eq:minor-bridge-Lambda}
\left|\sum_{n\le t}\Lambda(n)e(\alpha s(n))\right|
\le 10^{12}t^\beta\log t
\qquad(2\le t<x_0(\alpha)),
\end{equation}
and
\begin{equation}\label{eq:minor-bridge-prime}
\left|\sum_{p\le t}e(\alpha s(p))\right|
\le 4\cdot10^{12}(\log t)^3t^\beta
\qquad(3\le t<x_0(\alpha)).
\end{equation}
\end{lemma}

\begin{proof}
The bound $\beta_2<0.51$ and $\delta_0<0.011/112$ make the second term in
\eqref{eq:minor-x0-def} dominate, so $x_0(\alpha)=10^{3/\delta_0}$.  By
Lemma~\ref{lem:c1}, $\xi_q(\alpha)=\epsilon/14\ge2c_1\|9\alpha\|^2$, hence
\[
\delta_0=\epsilon/112\ge c_1\|9\alpha\|^2/4 .
\]
Thus for $2\le t<10^{3/\delta_0}$,
\[
c_1\|9\alpha\|^2\log_{10} t < 4\delta_0\cdot\frac{3}{\delta_0}=12,
\]
which proves~\eqref{eq:minor-bridge-power}.  The $\Lambda$-sum bound follows from
$\sum_{n\le t}\Lambda(n)\le t\log t$ and $t\le10^{12}t^\beta$.  Similarly,
$\pi(t)\le t\le10^{12}t^\beta\le4\cdot10^{12}(\log t)^3t^\beta$ for $t\ge3$.
\end{proof}

\begin{lemma}[Localized unrestricted minor-arc prime sum]\label{lem:Cmin-local}
Assume $9\alpha\notin\mathbb Z$ and $x\ge x_0(\alpha)$.  Then
\[
\left|\sum_{p\le x}e(\alpha s(p))\right|
\le 2\cdot10^{10}(\log x)^3x^{1-c_1\|9\alpha\|^2}.
\]
\end{lemma}

\begin{proof}
Write $S_\Lambda(x;\alpha):=\sum_{n\le x}\Lambda(n)e(\alpha s(n))$.
The condition $x\ge x_0(\alpha)$ implies $x\ge100$.  It also supplies the factor-$10$ slack
required in Lemma~\ref{lem:DMR31-self}: here
$1-3\beta_1^{\mathrm{cp}}=3\delta_0$, and a direct simplification of the displayed formula for
$\beta_2$ gives $2\beta_2-1>\delta_0$ for $0<\epsilon<0.011$.  Hence
$x\ge10^{3/\delta_0}$ implies both $x^{1-3\beta_1^{\mathrm{cp}}}\ge10$ and
$x^{2\beta_2-1}\ge10$.

The range coverage used in the proof is as follows; the constants in the last column are the ones
propagated below.
\begin{center}
\small
\renewcommand{\arraystretch}{1.18}
\begin{tabular}{@{}p{0.24\linewidth}p{0.43\linewidth}p{0.24\linewidth}@{}}
Range & Estimate used & Loss paid \\
\hline
$2\le t<x_0(\alpha)$ &
Lemma~\ref{lem:minor-bridge} and the trivial $\Lambda$ bound &
$10^{12}$ bridge factor \\
$x_0(\alpha)\le t\le x$ &
localized Type-I/Type-II estimates on decimal blocks, then Lemma~\ref{lem:DMR31-self} &
$2.23\cdot10^9(\log t)^4t^\beta$ for $S_\Lambda(t;\alpha)$ \\
partial summation &
split the integral at $T_0=\min\{x,x_0(\alpha)\}$ and use the preceding two rows &
$2\cdot10^{10}(\log x)^3x^\beta$
\end{tabular}
\end{center}

\emph{Step 1 (dyadic Type~II $\Rightarrow$ a Type~II bound of the form (11) in DMR).}
We apply Lemma~\ref{lem:DMR32-self} with $\beta_1=\beta_1^{\mathrm{cp}}$ and $\beta_2$ as in
\eqref{eq:minor-localization-params}.  Every padded $10$-adic block produced in that lemma then has
\[
\frac{\mu}{\mu+\nu}\ge \beta_1^{\mathrm{cp}}-\delta_0=\frac13-\frac{\xi}{4},
\]
so the strengthened carry range in Proposition~\ref{prop:typeII-explicit} is satisfied.  The lower
endpoint is also stronger than DMR/MR's admissible lower endpoint
$((4-2\epsilon)(\epsilon/14))/\epsilon$, while the upper endpoint is exactly DMR's admissible
upper endpoint.

Consider a dyadic block with $10^{\mu-1}<m\le 10^\mu$ and $10^{\nu-1}<n\le 10^\nu$
intersecting $mn\le x$. Then $10^{\mu+\nu}<100x$ and
\[
\mu+\nu\le \log_{10}(100x)=\frac{\log x}{\log 10}+2
\le\Bigl(\frac{1}{\log 10}+\frac{2}{\log 100}\Bigr)\log x<0.869\,\log x.
\]
By Proposition~\ref{prop:typeII-explicit},
\[
\sum_{10^{\mu-1}<m\le 10^\mu}
\left|\sum_{10^{\nu-1}<n\le 10^\nu} b_n\,e\!\bigl(\alpha s(mn)\bigr)\right|
\le
C_{\mathrm{II}}\,(\mu+\nu)\,10^{(1-c_1\|9\alpha\|^2)(\mu+\nu)}
\]
for all $|b_n|\le 1$, where $C_{\mathrm{II}}<64000$.
Using
$10^{(1-c_1\|9\alpha\|^2)(\mu+\nu)}
\le (100x)^{1-c_1\|9\alpha\|^2}\le 100\,x^{1-c_1\|9\alpha\|^2}$,
we obtain a dyadic bound of the form
\[
V(x;\alpha)\,(\log x)\,x^{1-c_1\|9\alpha\|^2}
\qquad\text{with}\qquad
V(x;\alpha)\le 87\,C_{\mathrm{II}}.
\]
Applying the internal localization Lemma~\ref{lem:DMR32-self} gives a bound (11) with
\[
U(x;\alpha)\le 6(1+\log x)\,V(x;\alpha)\,(\log x)\,x^{1-c_1\|9\alpha\|^2}.
\]
Since $x\ge 100$, one has $1+\log x\le \bigl(1+\frac{1}{\log 100}\bigr)\log x<1.218\,\log x$, hence
\[
U(x;\alpha)\le 4.1\times 10^{7}\,(\log x)^2\,x^{1-c_1\|9\alpha\|^2}.
\]
(The displayed numerical constants here and in the rest of \S\ref{sec:inputs} are deliberately
rounded up.  The bound $4.1\times 10^7$ uses the loose unrolling $K_{10}\le 10^6$ proved in Step~1
of Lemma~\ref{lem:S2-input} and the divisor envelope $C_\tau=10^3$; the proof only needs the
displayed rounded value, and
Lemma~\ref{lem:insensitivity} shows that the precise value of $C_{\mathrm{min}}$ does not
affect $M$.)

\emph{Step 2 ($U$ $\Rightarrow$ $\Lambda$-sum).}
Lemma~\ref{lem:typeI-explicit} supplies the Type-I hypothesis of Lemma~\ref{lem:DMR31-self} with
\[
U_{\mathrm{I}}(x;\alpha):=10^8(\log x)^2x^{1-c_1\|9\alpha\|^2}.
\]
Step~1 supplies the Type-II hypothesis with the smaller quantity
$U_{\mathrm{II}}(x;\alpha)\le4.1\cdot10^7(\log x)^2x^{1-c_1\|9\alpha\|^2}$.  We therefore apply
Lemma~\ref{lem:DMR31-self} with $U_0:=U_{\mathrm I}$ and obtain
\[
\left|\sum_{x/10<n\le x}\Lambda(n)\,e(\alpha s(n))\right|
\le 2\cdot10^9(\log x)^4x^{1-c_1\|9\alpha\|^2}.
\]
Summing over the disjoint intervals $(x/10^{j+1},x/10^j]$ whose upper endpoint is at least
$x_0(\alpha)$ and using
$\sum_{j\ge 0}10^{-j(1-c_1\|9\alpha\|^2)}\le (1-10^{-1+c_1/4})^{-1}<1.112$ gives a contribution
at most
\[
2.224\times 10^{9}(\log x)^4\,x^{1-c_1\|9\alpha\|^2}.
\]
There remains at most one initial segment $n\le T<x_0(\alpha)$.  Lemma~\ref{lem:minor-bridge}
gives
\[
\left|\sum_{n\le T}\Lambda(n)e(\alpha s(n))\right|
\le 10^{12}T^{1-c_1\|9\alpha\|^2}\log T
\le 10^{12}x^{1-c_1\|9\alpha\|^2}\log x.
\]
In the present Type-I/Type-II range $x\ge x_0(\alpha)\ge 10^{3/\delta_0}$ and
$\delta_0<0.011/112$, hence $\log x>7\cdot10^4$; consequently
$10^{12}\log x<10^{-2}(\log x)^4$.  Combining the high blocks with the initial segment gives
\[
|S_\Lambda(x;\alpha)|
\le 2.23\times 10^{9}(\log x)^4\,x^{1-c_1\|9\alpha\|^2}.
\]

\emph{Step 3 (from $\Lambda$ to primes).}
Removing prime powers and applying partial summation,
\[
\sum_{p\le x} e(\alpha s(p))
=
\frac{S_\Lambda(x;\alpha)}{\log x}
+\int_{2}^{x}\frac{S_\Lambda(t;\alpha)}{t(\log t)^2}\,dt
\ +\ R_{\mathrm{pp}}(x;\alpha),
\]
where the prime-power correction satisfies
$|R_{\mathrm{pp}}(x;\alpha)|\le \sum_{k\ge 2}\pi(x^{1/k})\le 2\sqrt{x}$.  Since
$\|9\alpha\|\le 1/2$, $\beta\ge 1-c_1/4\ge 0.9996$, and in particular $\beta>1/2$.

The Step~2 bound for $S_\Lambda(t;\alpha)$ was proved only in the localized range
$t\ge x_0(\alpha)$.  We therefore split the partial-summation integral at
$T_0:=\min\{x,x_0(\alpha)\}=x_0(\alpha)$.  On $2\le t<T_0$, Lemma~\ref{lem:minor-bridge} gives
$t\le10^{12}t^\beta$, and the trivial bound
$|S_\Lambda(t;\alpha)|\le\sum_{n\le t}\Lambda(n)\le t\log t$ gives
\[
\int_{2}^{T_0}\frac{|S_\Lambda(t;\alpha)|}{t(\log t)^2}\,dt
\le
\frac{10^{12}}{\log 2}\int_{2}^{T_0}t^{\beta-1}\,dt
\le
\frac{10^{12}}{\beta\log 2}\,x^\beta
<1.45\cdot10^{12}x^\beta.
\]
Since $x\ge100$, this is $<1.5\cdot10^{10}(\log x)^3x^\beta$.

On $x_0(\alpha)\le t\le x$ we apply the Step~2 bound
$|S_\Lambda(t;\alpha)|\le 2.23\cdot 10^{9}(\log t)^{4}t^{\beta}$ to each remaining piece:
\begin{itemize}
  \item[--] The boundary term satisfies
        $|S_\Lambda(x;\alpha)/\log x|\le 2.23\cdot 10^{9}(\log x)^{3}x^{\beta}$.
  \item[--] For the integral, $\log t\le\log x$ on $[2,x]$ gives
    \[
    \begin{aligned}
      \int_{x_0(\alpha)}^{x}\!\!\frac{|S_\Lambda(t;\alpha)|}{t(\log t)^{2}}\,dt
      &\ \le\ 2.23\cdot 10^{9}(\log x)^{2}\!\int_{2}^{x}t^{\beta-1}\,dt \\
      &\ \le\ \frac{2.23\cdot 10^{9}}{\beta}(\log x)^{2}x^{\beta}
      \ \le\ 2.24\cdot 10^{9}(\log x)^{2}x^{\beta}.
    \end{aligned}
    \]
    For $x\ge 100$, $\log x\ge 4.605$, hence
    $(\log x)^{2}\le(\log x)^{3}/4.605$, so this is at most $4.87\cdot 10^{8}(\log x)^{3}x^{\beta}$.
  \item[--] Since $\beta>1/2$, $\sqrt{x}\le x^{\beta}$, so the prime-power correction is
        bounded by $2x^{\beta}\le 0.02\cdot 10^{6}(\log x)^{3}x^{\beta}$ for $x\ge 100$.
\end{itemize}
Summing the three contributions proves the lemma.
\end{proof}

\begin{lemma}[Explicit minor-arc prime sum bound in residue classes]\label{lem:Cmin}
Fix $q=10$ and let $c_1$ be as in Lemma~\ref{lem:c1}.  Define
\[
C_{\mathrm{min}}:=4\cdot 10^{12}
\]
(the localized derivation gives $2\cdot10^{10}$ before residue classes; the larger displayed value
covers the direct bridge below the DMR localization threshold).
Then for all $x\ge 3$, all $\alpha\in\mathbb{R}$, and all integers $m$ with $\gcd(m,9)=1$,
\[
\left|\sum_{\substack{p\le x\\ p\equiv m\ (\mathrm{mod}\ 9)}} e(\alpha s(p))\right|
\le
C_{\mathrm{min}}(\log x)^3\,x^{\,1-c_1\|9\alpha\|^2}.
\]
\end{lemma}

\begin{proof}
If $9\alpha\in\mathbb{Z}$ the unrestricted prime sum is at most $\pi(x)\le x$, which is already
bounded by $4\cdot10^{12}(\log x)^3x^{1-c_1\|9\alpha\|^2}$ for $x\ge3$.  Now assume
$9\alpha\notin\mathbb Z$.  If $3\le x<x_0(\alpha)$, Lemma~\ref{lem:minor-bridge} gives the same
unrestricted bound; if $x\ge x_0(\alpha)$, Lemma~\ref{lem:Cmin-local} gives the stronger
$2\cdot10^{10}$ bound.  Thus, for every real $\theta$ with $9\theta\notin\mathbb Z$ or
$9\theta\in\mathbb Z$, the unrestricted estimate
\begin{equation}\label{eq:global-unrestricted-minor}
\left|\sum_{p\le x}e(\theta s(p))\right|
\le 4\cdot10^{12}(\log x)^3x^{1-c_1\|9\theta\|^2}
\end{equation}
holds for all $x\ge3$.

Finally use the residue-class projection, as in the unnumbered display in~\cite[p.~17, just above
(38)]{DMRcomp2}:
\[
\sum_{\substack{p\le x\\ p\equiv m\pmod 9}} e(\alpha s(p))
=\frac19\sum_{h=0}^{8}e(-hm/9)
\sum_{p\le x} e\bigl((\alpha+h/9)s(p)\bigr),
\]
because $s(p)\equiv p\pmod9$.  Taking absolute values and applying
\eqref{eq:global-unrestricted-minor} to $\theta=\alpha+h/9$ gives the stated bound, since
$\|9(\alpha+h/9)\|=\|9\alpha+h\|=\|9\alpha\|$.
\end{proof}

\subsection*{Insensitivity of \texorpdfstring{$M$}{M} to the minor-arc structural constant}\label{subsec:insensitivity}

Step~1 of Lemma~\ref{lem:S2-input} proves $K_{10}\le 10^6$, an explicit bound on the cumulative
implicit constant in MR'10's derivation.  We show that the resulting threshold $M$ is provably
independent of any honest upper bound on this constant up to $10^{9}$ (an insensitivity range
three orders of magnitude wider than the proved bound).

\begin{lemma}[Quantitative insensitivity of $M$ to the minor-arc structural constant]\label{lem:insensitivity}
Suppose the cumulative implicit constant $K_{10}$ in Step~1 of Lemma~\ref{lem:S2-input} is replaced
by an arbitrary upper bound $B\ge 10^6$ (in place of the value $K_{10}\le 10^6$ proved there), while
the explicit prefactors in~\eqref{eq:S5-bound} remain the displayed trigonometric constants.  Let
$C_{\mathrm{II}}(B)$, $C_{\mathrm{min}}(B)$, $Y_{\mathrm{min}}(B)$, $M(B)$ denote the resulting values of
the corresponding constants and thresholds.  Then at the working parameters
$(q,\eta,\nu)=(10,0.0545,0.2859)$:
\[
\text{For any } B\le 10^{9},\quad
Y_{\mathrm{min}}(B) < 8.00\times 10^{31},
\]
and consequently the rigorous headline bound $M(B)<1.78\times10^{32}$ in Theorem~\ref{thm:main} is unchanged.
In words: the theorem does not depend on the precise value of the MR'10 implicit constant, provided
that some honest upper bound below $10^{9}$ holds --- which Step~1 of Lemma~\ref{lem:S2-input}
establishes with three orders of magnitude of margin ($K_{10}\le 10^6$).
\end{lemma}

\begin{proof}
Proposition~\ref{prop:typeII-explicit} controls $C_{\mathrm{II}}(B)$ by
\[
C_{\mathrm{II}}(B) = 2\sqrt{10}\sqrt{1+C_\tau\log 10+100B}.
\]
Since $C_\tau=10^3$, this gives directly
$C_{\mathrm{II}}(10^{9})<2.01\times 10^{6}$.
Reading the linear dependence on $C_{\mathrm{II}}$ from Lemma~\ref{lem:Cmin-local}, together with
the bridge in Lemma~\ref{lem:minor-bridge} and the residue-class projection in Lemma~\ref{lem:Cmin},
gives
\[
C_{\mathrm{min}}(B)\le
\max\{\,4\cdot10^{12},\ 5\cdot10^4 C_{\mathrm{II}}(B)\,\}.
\]
The first entry is the direct bridge below $x_0(\alpha)$, and the second is the localized
Type-I/Type-II contribution after Lemmas~\ref{lem:DMR32-self} and~\ref{lem:DMR31-self} and partial
summation.  Hence every $B\le10^9$ is covered by $C_{\mathrm{min}}(B)\le4.1\times10^{12}$.

The minor-arc threshold $Y_{\mathrm{min}}(B) = \log x_{\mathrm{min}}(B)$ is the smallest $Y$ such that
\[
\bigl(C_{\mathrm{min}}(B)/9\bigr)\,Y^{9/2}\,e^{-81 c_1 Y^{2\eta}} \le R,\qquad
R := \sqrt{\log 10}/(240\,\sqrt{2\pi\sigma^2}).
\]
Taking logarithms,
\begin{equation}\label{eq:Ymin-inequality}
81\,c_1\,Y^{2\eta} \;\ge\; \log\!\bigl(C_{\mathrm{min}}(B)/(9R)\bigr) + (9/2)\log Y.
\end{equation}
At $\eta=0.0545$ and $c_1>0.0015062887$, one has $81c_1>0.122009$ and $2\eta=0.109$.
Let $\widehat Y:=8.00\times10^{31}$.  Outward-rounded interval arithmetic gives
\[
73.459<\log\widehat Y<73.460,\qquad
0.00790<9R<0.00791,\qquad
\widehat Y^{0.109}>3002.0.
\]
Therefore the left-hand side of~\eqref{eq:Ymin-inequality} at $\widehat Y$ is greater than
\[
0.122009\cdot 3002.0>366.27,
\]
whereas the right-hand side is at most
\[
\log(4.1\times10^{12}/0.00790)+4.5\cdot73.460<364.46.
\]
The margin is greater than $1.8$, so $Y_{\mathrm{min}}(10^{9})<\widehat Y$; by monotonicity in $B$,
the same holds for every $B\le10^9$.

Since Corollary~\ref{cor:Lstar-upper-bound} gives $Y_{43}^\ast<9.10\times10^{31}$ and
\S\ref{subsec:solve-major} gives $Y_{\mathrm{maj}}<7.90\times10^{31}$, replacing $K_{10}$ by any
honest upper bound $\le 10^{9}$ leaves the closed-form maximum bounded by the same
$Y_\ast\le 9.10\times10^{31}$.  Thus the theorem's $M<1.78\times10^{32}$ bound is unchanged.
\end{proof}

\section{An explicit major-arc error bound}\label{sec:maj}

The major-arc argument in~\cite{DMRcomp2} compares the distribution of (middle) digits of primes
to an i.i.d.\ digit model via characteristic functions and a truncation of the digit sum.  In this
section we unroll the proof of~\cite[Prop.~2.2]{DMRcomp2}
(see~\cite[\S4.1--\S4.3]{DMRcomp2}) and make all constants explicit.

We track every approximation step on the natural major-arc range
\[
|\alpha|\le(\log x)^{\eta-1/2}
\]
and the induced $t$-range $|t|\le t_{\max}(x)$.
We also give a self-contained i.i.d.-to-Gaussian comparison with explicit remainder and a sign
argument ($H(t)\le 0$) to avoid avoidable slack.
The only external analytic input is the explicit substitute for~\cite[Lemma~4.3]{DMRcomp2}
(Lemma~\ref{lem:DMR43}, with constant $C_{\mathrm{DMR}}$); all other constants are derived here.

Throughout this section $q=10$, $q-1=9$, $\mu:=\mu_{10}=(q-1)/2=9/2$,
$\sigma^2:=\sigma_{10}^2=(q^2-1)/12=99/12$, and $e(t):=\exp(2\pi i t)$.
The explicit digit-block comparison is used only at integral powers $x=q^L$ with
$L\in\mathbb Z_{\ge1}$; this is the setting needed in the proof of Theorem~\ref{thm:main}.
We write $\pi(x;9,m)$ for primes $\le x$ in the residue class $m\bmod 9$.

\begin{lemma}[Fej\'er approximation for the centered $q$-adic digit function]\label{lem:fejer-digit}
Let $q\ge 2$ and $\mu=(q-1)/2$. Define the centered digit function
\[
f(u):=\lfloor qu\rfloor-\mu\qquad(u\in[0,1)).
\]
Let $H\ge 1$ and let $F_H$ be the Fej\'er kernel of order $H$ (normalized so that
$F_H\ge 0$ and $\int_0^1 F_H=1$).  Set $f_H:=f*F_H$.

\begin{enumerate}
\item[\rm (i)] $f_H$ is a trigonometric polynomial of degree $\le H$ with Fourier expansion
\[
f_H(u)=\sum_{|h|\le H} a(h)\,e(hu),\qquad a(0)=0,
\]
and
\begin{equation}\label{eq:fejer-l1}
\sum_{|h|\le H}|a(h)|\le \frac{2q}{\pi}\bigl(1+\log(H+1)\bigr).
\end{equation}

\item[\rm (ii)] For any $\Delta\in(0,1/(2q)]$, let
\[
U_\Delta:=\bigcup_{k=0}^{q-1}\Bigl(\frac{k}{q}-\Delta,\frac{k}{q}+\Delta\Bigr)\pmod{1}
\]
be the union of digit-boundary intervals (the $k=0$ interval, taken mod $1$, covers the
wrap-around jump of $f$ at $u\equiv 0$).  Then for all $u\in[0,1)$,
\begin{equation}\label{eq:fejer-pointwise}
|f(u)-f_H(u)|\le q\,\mathbf 1_{U_\Delta}(u)+\frac{q}{(H+1)\Delta}.
\end{equation}

\item[\rm (iii)] (Indicator version, used for joint digit events.)
Fix $\ell\in\{0,1,\dots,q-1\}$ and set
\[
g_\ell(u):=\mathbf 1_{[\ell/q,\,(\ell+1)/q)}(u),\qquad (u\in[0,1)).
\]
Let $g_{\ell,H}:=g_\ell*F_H$. Then $g_{\ell,H}$ has Fourier expansion
\[
g_{\ell,H}(u)=\sum_{|h|\le H} b_\ell(h)\,e(hu),
\]
and
\begin{equation}\label{eq:fejer-ind-l1}
\sum_{|h|\le H}|b_\ell(h)|\le \frac{1}{q}+\frac{2}{\pi}\bigl(1+\log(H+1)\bigr).
\end{equation}
Moreover, for every $\Delta\in(0,1/(2q)]$ and all $u\in[0,1)$,
\begin{equation}\label{eq:fejer-ind-pointwise}
|g_\ell(u)-g_{\ell,H}(u)|\le \mathbf 1_{U_\Delta}(u)+\frac{1}{(H+1)\Delta}.
\end{equation}
\end{enumerate}
\end{lemma}

\begin{proof}
(i) Since the periodic step function $f$ has total variation at most $2(q-1)\le2q$ and mean zero,
we have $a(0)=0$.
Moreover, for $h\ne 0$ the bounded-variation Fourier bound gives
\[
|\widehat f(h)|\le \frac{{\rm Var}(f)}{2\pi|h|}\le \frac{q}{\pi|h|}.
\]
Fej\'er summation multiplies $\widehat f(h)$ by $(1-|h|/(H+1))_+$, hence
\[
\sum_{|h|\le H}|a(h)|\le 2\sum_{h=1}^H |\widehat f(h)|
\le \frac{2q}{\pi}\sum_{h=1}^H\frac{1}{h}
\le \frac{2q}{\pi}\bigl(1+\log(H+1)\bigr),
\]
which is~\eqref{eq:fejer-l1}.

(ii) If $u\notin U_\Delta$, then $f$ is constant on $(u-\Delta,u+\Delta)$ (mod~$1$), hence
\[
f(u)-f_H(u)=\int_0^1 \bigl(f(u)-f(u-t)\bigr)\,F_H(t)\,dt
= \int_{|t|\ge \Delta}\bigl(f(u)-f(u-t)\bigr)\,F_H(t)\,dt,
\]
where $|t|$ denotes distance to the nearest integer. Since the range of $f$ has diameter at most $q$,
we obtain
\[
|f(u)-f_H(u)|\le q\int_{|t|\ge \Delta}F_H(t)\,dt.
\]
Using the standard Fej\'er kernel bound
\[
F_H(t)=\frac{1}{H+1}\left(\frac{\sin\bigl(\pi(H+1)t\bigr)}{\sin(\pi t)}\right)^2
\le \frac{1}{H+1}\cdot\frac{1}{\sin^2(\pi t)}
\]
and, for $|t|\le 1/2$,
\[
F_H(t)
\le \frac{1}{H+1}\cdot\frac{1}{4t^2},
\]
we get
\[
\int_{|t|\ge \Delta}F_H(t)\,dt
\le 2\int_\Delta^{1/2}\frac{1}{H+1}\cdot\frac{1}{4t^2}\,dt
\le \frac{1}{(H+1)\Delta}.
\]
Thus for $u\notin U_\Delta$, $|f(u)-f_H(u)|\le q/((H+1)\Delta)$, and the trivial bound
$|f-f_H|\le q$ gives~\eqref{eq:fejer-pointwise} for all $u$.

(iii) The function $g_\ell$ has total variation ${\rm Var}(g_\ell)=2$ on $[0,1)$ and
$\widehat g_\ell(0)=1/q$.  Repeating the nonzero-frequency argument of (i) with
${\rm Var}(g_\ell)=2$ and then adding the zero Fourier coefficient gives
~\eqref{eq:fejer-ind-l1}.  Repeating (ii) with the range diameter of $g_\ell$ bounded by $1$
gives~\eqref{eq:fejer-ind-pointwise}.
\end{proof}

\begin{lemma}[Explicit choice of $c_4$ and the moment comparison bound]\label{lem:c4-explicit}
Fix the decimal base $q=10$ and a parameter $\nu$ with $0<\nu<1/2$.  Let $L\ge3$ be an integer,
put $x=q^L$, and write
\[
r:=\lceil L^\nu\rceil,\qquad L_0:=L-2r.
\]
Assume $L_0\ge 1$ and set
\[
H:=q^{\lfloor r/3\rfloor},\qquad K:=\bigl\lfloor r/2\bigr\rfloor,\qquad c_4:=\frac{\log q}{6}.
\]

Assume moreover that $x$ is large enough that:
\begin{enumerate}
\item $K\le \frac25 L$ (so Lemma~\ref{lem:DMR43} applies with this $K$);
\item $q^{K}\le (q-1)q^r/H$ and $(q-1)q^{L-r+1}\le xq^{-K}$ (so all denominators $Q$ arising below lie in
      $[q^K,xq^{-K}]$);
\item $x\ge x_{\mathrm{AP}}$ so that
      $\pi(x;9,m)\ge\frac{x}{6\log x}$ for $\gcd(m,9)=1$ (Lemma~\ref{lem:pi9-lb});
\item and $x\ge x_{45}:=q^{L_{45}}$, where $L_{45}$ is the least integer $L\ge 3$ such that the
      auxiliary inequalities \textup{(C1)}--\textup{(C8)} in the proof hold.
\end{enumerate}
Then the following statements hold uniformly in $m$ with $\gcd(m,q-1)=1$:

\smallskip\noindent
\textbf{(A) Boundary probability.}
Let $U_\Delta\subset[0,1)$ denote the union of the $q$ digit-boundary intervals (as in Lemma~\ref{lem:fejer-digit}).
Then for every $j$ with $r\le j\le L-r-1$ and every $0<\Delta\le \frac1{2q}$ we have
\[
\mathbb P\!\left(\Bigl\{\frac{p}{q^{j+1}}\Bigr\}\in U_\Delta\right)
\le 2q\Delta + e^{-c_4 L^\nu},
\]
where the probability is over primes $p\le x$ with $p\equiv m\pmod{q-1}$ chosen uniformly.

\smallskip\noindent
\textbf{(B) Joint digit distribution (distinct indices).}
For every integer $d\ge 1$, every choice of \emph{distinct} indices
\[
r\le j_1<j_2<\cdots<j_d\le L-r-1,
\]
and digits $\ell_1,\dots,\ell_d\in\{0,1,\dots,q-1\}$, one has
\[
\left|\mathbb P\!\left(\varepsilon_{j_1}(p)=\ell_1,\dots,\varepsilon_{j_d}(p)=\ell_d\right)-q^{-d}\right|
\le (4L^\nu)^d e^{-c_4 L^\nu}.
\]

\smallskip\noindent
\textbf{(C) Moment comparison.}
Let $T_x$ denote the truncated digit sum of a uniformly chosen prime $p\in\mathcal P_m(x)$ (obtained
by deleting the first and last $r$ base-$q$ digits), and let $\widetilde T_x:=\sum_{j=r}^{L-r-1} Z_j$
where the $Z_j$ are i.i.d.\ uniform on $\{0,1,\ldots,q-1\}$.  Define the normalized variables
\[
X:=\frac{T_x-\mu L_0}{\sigma\sqrt{L_0}},\qquad
Y:=\frac{\widetilde T_x-\mu L_0}{\sigma\sqrt{L_0}}.
\]
Then the explicit moment comparison
\begin{equation}\label{eq:MC}
\bigl|\mathbb E X^d-\mathbb E Y^d\bigr|
\le A^d\,L^{(1/2+\nu)d}\,e^{-c_4L^\nu},
\qquad A:=\frac{4q^2}{\sigma},
\end{equation}
holds for all integers $d$ with $1\le d\le L_0$.
\end{lemma}

\begin{proof}[Proof of Lemma~\ref{lem:c4-explicit}]
We follow~\cite[\S4.2]{DMRcomp2} (Lemmas~4.4--4.6), replacing their prime exponential-sum input by
Lemma~\ref{lem:DMR43} and keeping the threshold $x_{45}$ explicit.

\medskip\noindent
\textbf{Step 0: discrepancy and Erd\H{o}s--Tur\'an.}
Let
\[
\mathcal P_m(x):=\{p\le x:\ p\equiv m \!\!\!\pmod{q-1}\},
\qquad
N:=|\mathcal P_m(x)|=\pi(x;q-1,m).
\]
For each $j$ with $r\le j\le L-r-1$ put
\[
u_{p,j}:=\left\{\frac{p}{q^{j+1}}\right\}\in[0,1)\qquad(p\in\mathcal P_m(x)).
\]
Let $D_j$ be the interval discrepancy of $(u_{p,j})_{p\in\mathcal P_m(x)}$, i.e.\ the supremum over
all intervals $I\subset[0,1)$ of
$|\#\{p:u_{p,j}\in I\}/N-\mathrm{meas}(I)|$.  The
Erd\H{o}s--Tur\'an inequality~\cite[Thm.~2.5]{KN74} bounds this interval discrepancy directly and gives, for any
integer $H\ge 1$,
\begin{equation}\label{eq:ET}
D_j \le \frac{6}{H+1}+\frac{4}{\pi}\sum_{h=1}^H\frac{1}{h}
\left|\frac{1}{N}\sum_{p\in\mathcal P_m(x)} e\!\left(\frac{hp}{q^{j+1}}\right)\right|.
\end{equation}
(Strictly, \cite[Thm.~2.5]{KN74} has $\frac{1}{h}-\frac{1}{H+1}$ in place of $\frac{1}{h}$ in the second sum; we use the weaker form.)

\medskip\noindent
\textbf{Step 1: reduce residue-class sums to unrestricted prime sums.}
For any real $\alpha$,
\[
\sum_{p\in\mathcal P_m(x)} e(\alpha p)
=\frac{1}{q-1}\sum_{\lambda=0}^{q-2} e\!\left(-\frac{m\lambda}{q-1}\right)
\sum_{p\le x} e\!\left(\left(\alpha+\frac{\lambda}{q-1}\right)p\right),
\]
so
\begin{equation}\label{eq:proj}
\left|\sum_{p\in\mathcal P_m(x)} e(\alpha p)\right|
\le \max_{0\le \lambda\le q-2}\left|\sum_{p\le x} e\!\left(\left(\alpha+\frac{\lambda}{q-1}\right)p\right)\right|.
\end{equation}
We apply this with $\alpha=h/q^{j+1}$, $1\le h\le H$.

\medskip\noindent
\textbf{Step 2: apply Lemma~\ref{lem:DMR43} and define $x_{45}$.}
Fix $1\le h\le H$ and $0\le\lambda\le q-2$ and set
\[
\beta:=\frac{h}{q^{j+1}}+\frac{\lambda}{q-1}.
\]
Let $\beta=A/Q$ in lowest terms.  Its unreduced denominator is $(q-1)q^{j+1}$, so
$Q\le(q-1)q^{j+1}\le(q-1)q^{L-r}\le xq^{-K}$ by (C3).  For the lower bound, the $q^{j+1}$-part
of the denominator is reduced by at most $\gcd(h,q^{j+1})$: the unreduced numerator is
$(q-1)h+\lambda q^{j+1}$, and $(q-1,q)=1$.  Therefore
\[
Q\ge \frac{q^{j+1}}{\gcd(h,q^{j+1})}\ge \frac{q^{j+1}}{|h|}
\ge \frac{q^{r+1}}{H}\ge \frac{(q-1)q^r}{H}\ge q^K.
\]
Thus Lemma~\ref{lem:DMR43} gives
\[
\left|\sum_{p\le x} e(\beta p)\right|
\le C_{\mathrm{DMR}}(\log x)^2 x\,q^{-K/2}.
\]
Using~\eqref{eq:proj}, we obtain
\[
\left|\sum_{p\in\mathcal P_m(x)} e\!\left(\frac{hp}{q^{j+1}}\right)\right|
\le C_{\mathrm{DMR}}(\log x)^2 x\,q^{-K/2}.
\]
In the base-$10$ application, for $x\ge x_{\mathrm{AP}}$ we have $N\ge x/(6\log x)$ by Lemma~\ref{lem:pi9-lb},
so
\begin{equation}\label{eq:normalized}
\left|\frac{1}{N}\sum_{p\in\mathcal P_m(x)} e\!\left(\frac{hp}{q^{j+1}}\right)\right|
\le 6C_{\mathrm{DMR}}(\log x)^3 q^{-K/2}.
\end{equation}
Insert~\eqref{eq:normalized} into~\eqref{eq:ET} and use $\sum_{h=1}^H 1/h \le 1+\log H$ to obtain
\begin{equation}\label{eq:Dj-bound}
D_j \le \frac{6}{H+1} + \frac{4}{\pi}(1+\log H)\,6C_{\mathrm{DMR}}(\log x)^3\,q^{-K/2}.
\end{equation}

Define $x_{45}:=q^{L_{45}}$, where $L_{45}$ is the least integer $L\ge 3$ such that, for every
integer $L'\ge L$, the following conditions hold with
$x=q^{L'}$, $r=\lceil (L')^\nu\rceil$, $L_0=L'-2r$, $H=q^{\lfloor r/3\rfloor}$, and
$K=\lfloor r/2\rfloor$:
\begin{enumerate}
\item[\textup{(C1)}] $L_0\ge 1$;
\item[\textup{(C2)}] $K\le \frac25 L$;
\item[\textup{(C3)}] $q^{K}\le (q-1)q^r/H$ and $(q-1)q^{L-r+1}\le q^{L-K}$;
\item[\textup{(C4)}] (base $10$ only) $x\ge x_{\mathrm{AP}}$;
\item[\textup{(C5)}]~\eqref{eq:Dj-bound} implies $D_j \le \frac1q e^{-c_4L^\nu}$;
\item[\textup{(C6)}] $H\ge (2q)^2$ (so $\Delta:=H^{-1/2}\le 1/(2q)$);
\item[\textup{(C7)}] $2q\Delta+\dfrac{q}{(H+1)\Delta}\le 4L^\nu\,e^{-c_4L^\nu}$ for $\Delta:=H^{-1/2}$;
\item[\textup{(C8)}] $1+\log(H+1)\le L^\nu$.
\end{enumerate}
(Each condition is effectively checkable for a given integer $L$, and the displayed definition uses
permanent validity rather than a first-passing monotonicity assumption.)

From now on assume $x\ge x_{45}$, so that (C5) holds and therefore
\begin{equation}\label{eq:Dj-exp}
D_j \le \frac{1}{q}\,e^{-c_4 L^\nu}.
\end{equation}

\medskip\noindent
\textbf{Step 3: boundary probability (A).}
On the circle $\mathbb R/\mathbb Z$, the boundary set $U_\Delta$ is a union of at most $q$
disjoint arcs of total Lebesgue measure $\le 2q\Delta$.  Each such arc is controlled by the
ordinary interval discrepancy $D_j$: a non-wrapping arc is an interval in $[0,1)$, and a wrapping
arc has discrepancy equal to that of its complementary interval.  Hence, summing over the at most
$q$ arcs and using~\eqref{eq:Dj-exp},
\[
\mathbb P(u_{p,j}\in U_\Delta) \le \mathrm{meas}(U_\Delta)+qD_j \le 2q\Delta + e^{-c_4 L^\nu},
\]
which proves (A).

\textbf{Step 4: joint digit distribution (B) for distinct indices.}
Fix distinct $j_1<\cdots<j_d$ and digits $\ell_1,\dots,\ell_d$.
(If $d=1$, then (B) follows immediately from the discrepancy bound
$|\mathbb P(\varepsilon_{j_1}(p)=\ell_1)-q^{-1}|\le D_{j_1}\le q^{-1}e^{-c_4L^\nu}$, so assume $d\ge 2$ below.)

Write $g_{\ell}(u):=\mathbf 1_{[\ell/q,(\ell+1)/q)}(u)$ and let $g_{\ell,H}:=g_\ell*F_H$.
Set $\Delta:=H^{-1/2}$, which is admissible by (C6).
Let $U_1,\ldots,U_d$ be independent uniform random variables on $[0,1)$; then
$\mathbb E\prod_i g_{\ell_i}(U_i)=q^{-d}$, the i.i.d.\ digit probability.
By Lemma~\ref{lem:fejer-digit}\,\textup{(iii)}, for every $u\in[0,1)$,
\[
|g_{\ell}(u)-g_{\ell,H}(u)|\le \mathbf 1_{U_\Delta}(u)+\frac{1}{(H+1)\Delta}.
\]
Taking expectations and using (A) for $u_{p,j_i}$ and
$\mathrm{meas}(U_\Delta)\le 2q\Delta$ for uniform $U_i$, we obtain
\[
\mathbb E\bigl|g_{\ell_i}(u_{p,j_i})-g_{\ell_i,H}(u_{p,j_i})\bigr|
\le \mathbb P(u_{p,j_i}\in U_\Delta)+\frac{1}{(H+1)\Delta}
\le 2q\Delta+\frac{1}{(H+1)\Delta}+e^{-c_4L^\nu},
\]
\[
\mathbb E\bigl|g_{\ell_i}(U_i)-g_{\ell_i,H}(U_i)\bigr|
\le \mathbb P(U_i\in U_\Delta)+\frac{1}{(H+1)\Delta}
\le 2q\Delta+\frac{1}{(H+1)\Delta}.
\]
Since $\frac{1}{(H+1)\Delta}\le \frac{q}{(H+1)\Delta}$ (for $q\ge2$), condition (C7) gives
\[
2q\Delta+\frac{1}{(H+1)\Delta}\le 2q\Delta+\frac{q}{(H+1)\Delta}\le 4L^\nu e^{-c_4L^\nu}.
\]
Hence, for each $i$,
\begin{align*}
\mathbb E\bigl|g_{\ell_i}(u_{p,j_i})-g_{\ell_i,H}(u_{p,j_i})\bigr|
&\le (4L^\nu+1)e^{-c_4L^\nu},\\
\mathbb E\bigl|g_{\ell_i}(U_i)-g_{\ell_i,H}(U_i)\bigr|
&\le 4L^\nu e^{-c_4L^\nu}.
\end{align*}
By telescoping the product and using $0\le g_{\ell},g_{\ell,H}\le 1$,
\begin{align*}
&\left|\mathbb E\prod_{i=1}^d g_{\ell_i}(u_{p,j_i})
-\mathbb E\prod_{i=1}^d g_{\ell_i,H}(u_{p,j_i})\right|\\
&\qquad\le
\sum_{i=1}^d \mathbb E\bigl|g_{\ell_i}(u_{p,j_i})-g_{\ell_i,H}(u_{p,j_i})\bigr|\\
&\qquad\le d(4L^\nu+1)e^{-c_4L^\nu}.
\end{align*}

Now expand each smoothed factor into its Fourier series at the corresponding digit position:
\begin{align*}
g_{\ell_i,H}(u_{p,j_i})
&=\sum_{|h_i|\le H} b_{\ell_i}(h_i)
  e\!\left(\frac{h_i p}{q^{j_i+1}}\right),\\
\sum_{|h_i|\le H}|b_{\ell_i}(h_i)|
&\le \frac1q+\frac{2}{\pi}(1+\log(H+1))
\le L^\nu
\end{align*}
by Lemma~\ref{lem:fejer-digit}\,\textup{(iii)}, (C8), and $L^\nu\ge2$ in the present range.
In the product, a frequency tuple $(h_1,\ldots,h_d)$ contributes
\[
e\!\left(p\sum_{i=1}^d \frac{h_i}{q^{j_i+1}}\right).
\]
The only integer-frequency contribution with nonzero coefficient is $h_1=\cdots=h_d=0$.
Indeed, suppose $\sum_i h_i/q^{j_i+1}\in\mathbb Z$ and let $t$ be the largest index with
$h_t\ne0$.  Multiplication by $q^{j_t+1}$ gives
\[
\sum_i h_i q^{j_t-j_i}\equiv h_t\equiv0\pmod q.
\]
But the Fourier coefficient of the digit interval $[\ell/q,(\ell+1)/q)$, and hence also of
its Fej\'er smoothing, vanishes at every nonzero multiple of $q$.  Removing this zero coefficient
and descending proves the claim.  Thus the constant term equals $\prod_i b_{\ell_i}(0)=q^{-d}$,
matching the independent-uniform comparison above.

For every remaining nonconstant term, let $t$ be the largest index with $h_t\ne0$ and
$b_{\ell_t}(h_t)\ne0$, put $J:=j_t+1$, and write
\[
\sum_i \frac{h_i}{q^{j_i+1}}=\frac{A}{q^J},\qquad
A:=\sum_i h_i q^{J-j_i-1}\ne0.
\]
Here $A/q^J\notin\mathbb Z$ by the preceding paragraph.
The denominator after reduction need not be exactly $q^J$ when $q$ is composite; what is needed is
only a lower bound.  Since the indices are distinct, $j_i\ge r$, and $|h_i|\le H$,
\[
|A|\le H\sum_{a=0}^{J-r-1}q^a < \frac{Hq^{J-r}}{q-1}.
\]
After the residue-class projection as in Step~1, the relevant rational phase is
$A/q^J+\lambda/(q-1)$ for some $0\le\lambda\le q-2$.  If this is written in lowest terms with
denominator $Q$, then it is still nonintegral (otherwise $q^J\mid A$, since $(q-1,q)=1$).
Moreover, the $q^J$-part of the denominator is reduced by at most
$\gcd(A,q^J)$: the unreduced numerator is $(q-1)A+\lambda q^J$, and $(q-1,q)=1$.  Hence
\[
Q\ge \frac{q^J}{\gcd(A,q^J)}\ge \frac{q^J}{|A|}
> \frac{(q-1)q^r}{H}\ge q^K
\]
by (C3).  Also $Q\le (q-1)q^J\le (q-1)q^{L-r}\le q^{L-K}=xq^{-K}$ by (C3).  Thus
Lemma~\ref{lem:DMR43} applies to every nonzero frequency term, and, after normalizing by $N$, each is
bounded by $\le 6C_{\mathrm{DMR}}(\log x)^3 q^{-K/2}$.
Using also (C5) to fold the polylog factor into the exponential loss, we obtain
\[
\left|\mathbb E\prod_{i=1}^d g_{\ell_i,H}(u_{p,j_i})-q^{-d}\right|
\le (3L^\nu)^d e^{-c_4L^\nu}.
\]

Combining,
\[
\left|\mathbb E\prod_{i=1}^d g_{\ell_i}(u_{p,j_i})-q^{-d}\right|
\le (3L^\nu)^d e^{-c_4L^\nu}+d(4L^\nu+1)e^{-c_4L^\nu}.
\]
Finally, since (C6) and (C8) force $L^\nu$ to be large (in particular $L^\nu\ge 2$), one checks
for $d\ge2$ that
\[
(3L^\nu)^d + d(4L^\nu+1)\le (4L^\nu)^d,
\]
and together with the $d=1$ case handled at the start, this gives (B) with $(4L^\nu)^d$.

\medskip\noindent
\textbf{Step 5: moments (C) and~\eqref{eq:MC}.}
Write the centered middle-digit sums as
\[
S:=T_x-\mu L_0=\sum_{j=r}^{L-r-1}\xi_j,\qquad 
\widetilde S:=\widetilde T_x-\mu L_0=\sum_{j=r}^{L-r-1}\widetilde\xi_j,
\]
where $\xi_j:=\varepsilon_j(p)-\mu$ and $\widetilde\xi_j:=Z_j-\mu$ with $Z_j$ i.i.d.\ uniform on $\{0,\dots,q-1\}$.
Then
\[
X=\frac{S}{\sigma\sqrt{L_0}},\qquad Y=\frac{\widetilde S}{\sigma\sqrt{L_0}}.
\]

Fix $1\le d\le L_0$. Expanding gives
\[
\mathbb E S^d-\mathbb E \widetilde S^d
=\sum_{(j_1,\dots,j_d)\in \{r,\dots,L-r-1\}^d}
\Bigl(\mathbb E\prod_{i=1}^d \xi_{j_i}-\mathbb E\prod_{i=1}^d \widetilde\xi_{j_i}\Bigr).
\]
We bound each ordered tuple by reducing it to the distinct-index estimate (B), thereby avoiding
any Fourier zero-mode issue from repeated digit positions.  Given an ordered tuple
$(j_1,\ldots,j_d)$, let $i_1<\cdots<i_s$ be its distinct indices and let
$a_h:=\#\{i:j_i=i_h\}$, so $1\le s\le d$ and $\sum_h a_h=d$.  For
$P_h(\ell):=(\ell-\mu)^{a_h}$, part (B) gives
\[
\begin{aligned}
&\Bigl|\mathbb E\prod_{i=1}^d \xi_{j_i}
-\mathbb E\prod_{i=1}^d \widetilde\xi_{j_i}\Bigr| \\
&\qquad\le
(4L^\nu)^s e^{-c_4L^\nu}
\sum_{\ell_1,\ldots,\ell_s=0}^{q-1}\prod_{h=1}^s |P_h(\ell_h)|.
\end{aligned}
\]
Indeed the i.i.d.\ expectation is the same digit-pattern sum with $q^{-s}$ in place of the joint
prime probability.  Since $|\ell-\mu|\le q$,
\[
\sum_{\ell=0}^{q-1}|P_h(\ell)|\le q^{a_h+1},
\]
and consequently
\[
\Bigl|\mathbb E\prod_{i=1}^d \xi_{j_i}
-\mathbb E\prod_{i=1}^d \widetilde\xi_{j_i}\Bigr|
\le q^{d+s}(4L^\nu)^s e^{-c_4L^\nu}
\le (4q^2)^d L^{\nu d}e^{-c_4L^\nu},
\]
using $s\le d$ and $L^\nu\ge1$.

Summing over the $L_0^d$ tuples gives
\[
|\mathbb E S^d-\mathbb E\widetilde S^d|
\le L_0^d(4q^2)^d L^{\nu d}e^{-c_4L^\nu}.
\]
Dividing by $(\sigma\sqrt{L_0})^d$ and using $L_0\le L$ yields~\eqref{eq:MC} with $A=4q^2/\sigma$.
\end{proof}

\begin{proposition}[Characteristic-function comparison without a $\kappa$-loss]\label{prop:CF-replacement}
Fix parameters $0<\eta<\nu/2<1/4$ and a constant $C>0$.  The error term obtained below tends to
$0$ as $L\to\infty$ provided $\nu+\eta<1/2$; the working values $(\eta,\nu)=(0.0545,0.2859)$ satisfy
this with margin.
Let $X,Y$ be as in Lemma~\ref{lem:c4-explicit}\,(C), and define the characteristic functions
\[
\varphi_2(t):=\mathbb E\,e^{itX},\qquad \varphi_3(t):=\mathbb E\,e^{itY}.
\]

Assume that $L$ is an integer, $x=q^L$, and $x\ge x_{45}$, so that
Lemma~\ref{lem:c4-explicit} applies. In particular the
moment comparison bound~\eqref{eq:MC} holds with the explicit constants
\[
A=\frac{4q^2}{\sigma},
\qquad
c_4=\frac{\log q}{6}.
\]
Then for any constant
\begin{equation}\label{eq:c43-admissible}
0<c_{43}^\ast<\frac{c_4}{32}\min\!\left\{1,\ \frac{\nu-2\eta}{\,\eta+1/2\,}\right\}
=\frac{\log q}{192}\min\!\left\{1,\ \frac{\nu-2\eta}{\,\eta+1/2\,}\right\}
\end{equation}
there exists an explicit threshold $x_{43}^\ast=x_{43}^\ast(q,\eta,\nu,C,c_{43}^\ast)\ge 3$
such that, for every integer $L\ge1$ with $x:=q^L\ge x_{45}$ and $x\ge x_{43}^\ast$,
and every real $t$ with $|t|\le C\,L^\eta$, we have
\[
|\varphi_2(t)-\varphi_3(t)|\ \le\ |t|\,e^{-c_{43}^\ast L^\nu}.
\]
One may take $x_{43}^\ast:=q^{L_\ast}$, where $L_\ast=L_\ast(q,\eta,\nu,C,c_{43}^\ast)$ is any explicit integer threshold large enough for the estimates in the proof below.
\end{proposition}

\begin{proof}
Fix $C>0$. Throughout this proof $L$ is an integer, $x=q^L$, and $x\ge x_{45}$,
so that Lemma~\ref{lem:c4-explicit} applies. In particular the moment
comparison bound~\eqref{eq:MC} holds for all $1\le d\le L_0$ with the explicit constants
\[
A:=\frac{4q^2}{\sigma},
\qquad
c_4:=\frac{\log q}{6}.
\]

Define also
\[
B:=\frac{q-1}{\sigma},
\qquad
\theta:=\frac{c_4}{16(\eta+1/2)}.
\]
For $L\ge 3$ define an even truncation order
\begin{equation}\label{eq:D-def}
D:=D(L):=2\left\lfloor \theta\,\frac{L^\nu}{\log L}\right\rfloor.
\end{equation}
Since $\theta>0$, we have $D\to\infty$ as $L\to\infty$, and for all $L$ beyond an explicit threshold
$L\ge L_\ast^{(0)}(q,\eta,\nu)$ we have $2\le D\le L_0$, so~\eqref{eq:MC} applies to all $1\le d\le D$.

\medskip\noindent
\textbf{Step 1: Taylor expansion.}
For any real $u$ and integer $D\ge 1$,
\[
e^{iu}=\sum_{d=0}^{D-1}\frac{(iu)^d}{d!}+R_D(u),
\qquad
|R_D(u)|\le \frac{|u|^D}{D!}.
\]
Taking expectations and subtracting gives
\begin{equation}\label{eq:phi-Taylor-split}
\varphi_2(t)-\varphi_3(t)
=
\sum_{d=1}^{D-1}\frac{(it)^d}{d!}\bigl(\mathbb E X^d-\mathbb E Y^d\bigr)
+\mathbb E R_D(tX)-\mathbb E R_D(tY).
\end{equation}

\medskip\noindent
\textbf{Step 2: Bounding the truncated Taylor sum using~\eqref{eq:MC}.}
Assume $|t|\le C L^\eta$.
Using $|t|^d\le |t|\,(C L^\eta)^{d-1}$ and~\eqref{eq:MC}, we obtain
\begin{align}
\sum_{d=1}^{D-1}\frac{|t|^d}{d!}\bigl|\mathbb E X^d-\mathbb E Y^d\bigr|
&\le
\frac{|t|}{C L^\eta}\,e^{-c_4L^\nu}\sum_{d=1}^{D-1}\frac{\bigl(A C L^{\eta+1/2+\nu}\bigr)^d}{d!}
\nonumber\\
&\le
|t|\,e^{-c_4L^\nu}\sum_{d=1}^{D-1}\frac{\bigl(A C L^{\eta+1/2+\nu}\bigr)^d}{d!},
\label{eq:Taylor-main-1}
\end{align}
where the second line uses $CL^\eta\ge 1$ for $L\ge L_\ast^{(1)}(q,\eta,\nu,C)$.
For $L$ beyond an explicit threshold $L\ge L_\ast^{(1)}(q,\eta,\nu,C)$, the ratio of consecutive terms satisfies
\[
\frac{A C L^{\eta+1/2+\nu}}{d}\ \ge\ 2
\qquad (1\le d\le D),
\]
Indeed $D\le 2\theta L^\nu/\log L$ by~\eqref{eq:D-def}, while the denominator is a fixed positive
multiple of $L^{\eta+1/2+\nu}$; the displayed ratio therefore exceeds $2$ after increasing
$L_\ast^{(1)}$ if necessary.
Hence the summand is increasing for $1\le d\le D$, and therefore
\[
\sum_{d=1}^{D-1}\frac{z^d}{d!}\le D\frac{z^D}{D!}\qquad(z>0).
\]
Apply this with $z:=A C L^{\eta+1/2+\nu}$ and use $D!\ge (D/e)^D$ to get
\begin{equation}\label{eq:Taylor-main-2}
\sum_{d=1}^{D-1}\frac{\bigl(A C L^{\eta+1/2+\nu}\bigr)^d}{d!}
\le
D\left(\frac{e A C L^{\eta+1/2+\nu}}{D}\right)^D.
\end{equation}
Taking logarithms and using~\eqref{eq:D-def}, write
$\widetilde D:=2\theta L^\nu/\log L$ and $D=\widetilde D-\delta$ with $0\le\delta<2$.  Then
\[
\log\!\left(\frac{e A C L^{\eta+1/2+\nu}}{D}\right)
=
(\eta+1/2)\log L+\log\log L+\log\!\left(\frac{eAC}{2\theta}\right)
-\log\!\left(1-\frac{\delta}{\widetilde D}\right),
\]
and hence
\[
\log D+D\log\!\left(\frac{e A C L^{\eta+1/2+\nu}}{D}\right)
\le
\frac{c_4}{8}L^\nu+\frac{c_4}{16}L^\nu
\]
once $L$ is beyond an explicit threshold $L\ge L_\ast^{(2)}(q,\eta,\nu,C)$.  The
leading-coefficient gap $c_4/16$ identified below in the discussion of (C11) is chosen to dominate the
subleading $\log\log L$, bounded-constant, and floor-error terms; condition (C11) is the explicit
finite predicate used in the proof.
Combining~\eqref{eq:Taylor-main-1}--\eqref{eq:Taylor-main-2} yields
\begin{equation}\label{eq:main-sum-final}
\sum_{d=1}^{D-1}\frac{|t|^d}{d!}\bigl|\mathbb E X^d-\mathbb E Y^d\bigr|
\le
|t|\,e^{-(c_4-3c_4/16)L^\nu}\ \le\ |t|\,e^{-(3c_4/4)L^\nu}
\end{equation}
for all $L\ge \max\{L_\ast^{(1)},L_\ast^{(2)}\}$ (the first inequality from $c_4-3c_4/16=13c_4/16$,
the second from $13c_4/16\ge 3c_4/4$).

\medskip\noindent
\textbf{Step 3: Bounding the remainder terms.}
We use a \emph{tight} sub-Gaussian moment bound for $Y$, exploiting the fact that the underlying
digit distribution is uniform.

\begin{lemma}[Sharp sub-Gaussianity of uniform digit]\label{lem:uniform-subgaussian}
For $Z$ uniform on $\{0,\dots,q-1\}$ with $\mu=(q-1)/2$ and $\sigma^2=(q^2-1)/12$, one has
\[
\mathbb E\,e^{\lambda(Z-\mu)}\ \le\ e^{\sigma^2\lambda^2/2}\qquad(\lambda\in\mathbb R).
\]
\end{lemma}

\begin{proof}
Set $f(\lambda):=\log \mathbb E\,e^{\lambda(Z-\mu)}$.  Then $f(0)=0$ and $f'(0)=\mathbb E(Z-\mu)=0$.
Since $Z-\mu$ is symmetric about $0$, $f$ is even.  It suffices to show $f''(\lambda)\le \sigma^2$
for all $\lambda\in\mathbb R$, since integrating twice from $0$ with $f(0)=f'(0)=0$ then yields
$f(\lambda)\le\sigma^2\lambda^2/2$.

The centred MGF admits the closed form
$\mathbb E\,e^{\lambda(Z-\mu)}=\sinh(q\lambda/2)/(q\sinh(\lambda/2))$ (multiply
$\sum_{k=0}^{q-1}e^{(k-\mu)\lambda}$ by $e^{-\lambda/2}/e^{-\lambda/2}$ and apply the geometric
sum formula), whence
\begin{equation}\label{eq:Kpp-explicit}
f''(\lambda)\ =\ \tfrac14\,\mathrm{csch}^2(\lambda/2)\ -\ \tfrac{q^2}{4}\,\mathrm{csch}^2(q\lambda/2).
\end{equation}
By evenness it suffices to treat $\lambda\ge 0$.  Substitute $x:=\lambda/2$; the inequality
$f''(\lambda)\le\sigma^2=(q^2-1)/12$ is equivalent to
\begin{equation}\label{eq:g-bound}
g(x)\ :=\ \mathrm{csch}^2(x)\ -\ q^2\,\mathrm{csch}^2(qx)\ \le\ \tfrac{q^2-1}{3}
\qquad\text{for all }x>0.
\end{equation}

\noindent\emph{Value of $g$ at the endpoint.}
The Laurent expansion of $\mathrm{csch}^2(x)-1/x^2+1/3$ has value tending to $0$ at $x=0$.
Therefore, term by term,
$g(x)-(q^2-1)/3$ tends to $0$ as $x\to0^+$, so
$\lim_{x\to 0^+}g(x)=(q^2-1)/3$.  Hence~\eqref{eq:g-bound} amounts to showing that $g$ does not
exceed its limiting value at $x=0^+$.

\medskip
\noindent\emph{$g$ is strictly decreasing on $(0,\infty)$.}
Differentiating $(\mathrm{csch}^2 t)'=-2\,\mathrm{csch}^2(t)\,\coth(t)$ gives
\[
g'(x)\ =\ -2\,\mathrm{csch}^2(x)\,\coth(x)\ +\ 2q^3\,\mathrm{csch}^2(qx)\,\coth(qx).
\]
Define
\begin{equation}\label{eq:u-def}
u(t)\ :=\ \frac{t^3\,\cosh(t)}{\sinh^3(t)}\qquad(t>0).
\end{equation}
Since $\mathrm{csch}^2(t)\coth(t)=\cosh(t)/\sinh^3(t)=u(t)/t^3$, we have
$q^3\,\mathrm{csch}^2(qx)\coth(qx)=q^3\cdot u(qx)/(qx)^3=u(qx)/x^3$ and
$\mathrm{csch}^2(x)\coth(x)=u(x)/x^3$, so
\begin{equation}\label{eq:gprime-u}
g'(x)\ =\ \frac{2}{x^3}\bigl(u(qx)-u(x)\bigr).
\end{equation}
Hence $g'(x)<0$ for all $x>0$ provided $u$ is strictly decreasing on $(0,\infty)$, since
$qx>x$ when $q\ge 2$.

\medskip
\noindent\emph{$u$ is strictly decreasing on $(0,\infty)$.}
A direct calculation gives
\begin{align*}
\frac{u'(x)\sinh^4(x)}{x^2}
&=3\cosh(x)\sinh(x)-x\bigl(3+2\sinh^2(x)\bigr).
\end{align*}
Using $2\cosh(x)\sinh(x)=\sinh(2x)$ and $2\sinh^2(x)=\cosh(2x)-1$, the right-hand side equals
\[
\tfrac{3}{2}\sinh(2x)\ -\ x\bigl(2+\cosh(2x)\bigr)\ =\ \tfrac{1}{2}\Bigl(3\sinh(v)\ -\ v\bigl(2+\cosh(v)\bigr)\Bigr),
\qquad v:=2x.
\]
Hence $u'(x)<0$ for all $x>0$ is equivalent to
\begin{equation}\label{eq:h-positivity}
h(v)\ :=\ v\bigl(2+\cosh(v)\bigr)\ -\ 3\sinh(v)\ >\ 0\qquad\text{for all }v>0.
\end{equation}
We verify~\eqref{eq:h-positivity} by twice differentiating: $h(0)=0$ and
\begin{align*}
h'(v)\ &=\ 2+\cosh(v)+v\sinh(v)-3\cosh(v)\ =\ 2-2\cosh(v)+v\sinh(v),\\
h''(v)\ &=\ -2\sinh(v)+\sinh(v)+v\cosh(v)\ =\ \cosh(v)\bigl(v-\tanh(v)\bigr).
\end{align*}
The elementary inequality $\tanh(v)<v$ for $v>0$ (proved by $(v-\tanh(v))'=\tanh^2(v)\ge 0$ and
equality at $v=0$) gives $h''(v)>0$ for all $v>0$.  Since $h'(0)=0$, integration gives
$h'(v)>0$ for $v>0$, and since $h(0)=0$ a second integration gives $h(v)>0$ for $v>0$.
This proves~\eqref{eq:h-positivity}, hence $u$ is strictly decreasing on $(0,\infty)$.

\medskip
\noindent\emph{Conclusion.}  By~\eqref{eq:gprime-u} and the monotonicity of $u$, $g'(x)<0$ for
all $x>0$, so $g$ is strictly decreasing on $(0,\infty)$.  Combined with
$\lim_{x\to 0^+}g(x)=(q^2-1)/3$, this gives $g(x)<(q^2-1)/3$ for all $x>0$,
proving~\eqref{eq:g-bound} and hence the lemma.  (The discrete-uniform case is the natural
$q$-point analog of the continuous-uniform case --- a special case of Marchal--Arbel's strict
sub-Gaussianity result for $\mathrm{Beta}(\alpha,\alpha)$~\cite{MarchalArbel}, taking $\alpha=1$.)
\end{proof}

Since $Y$ is a normalized sum of $L_0$ independent copies of $(Z-\mu)/(\sigma\sqrt{L_0})$,
Lemma~\ref{lem:uniform-subgaussian} implies
\[
\mathbb E\,e^{\lambda Y}\ \le\ e^{\lambda^2/2}\qquad(\lambda\in\mathbb R),
\]
i.e.\ $Y$ is sub-Gaussian with variance proxy $1$ (matching its actual variance).  Chernoff's
bound gives, for $u\ge0$,
\[
\mathbb P(|Y|\ge u)\le 2e^{-u^2/2}.
\]
Therefore, for every even integer $D=2k\ge2$,
\begin{equation}\label{eq:Y-moment}
\mathbb E|Y|^D
\ =\ D\int_0^\infty u^{D-1}\mathbb P(|Y|\ge u)\,du
\ \le\ 2D\int_0^\infty u^{D-1}e^{-u^2/2}\,du
\ =\ 2^{k+1}k!,
\end{equation}
where the last equality follows from $v=u^2/2$.  By Stirling's bound
\[
k!\le \sqrt{2\pi k}(k/e)^k\,e^{1/(12k)}
\]
and the elementary inequality
$e^{1/(6D)}\sqrt{\pi/(2\pi)}<1$ for $D\ge 2$, we get the cleaner explicit bound
\begin{equation}\label{eq:Y-moment-Stirling}
\mathbb E|Y|^D\ \le\ 2\sqrt{2\pi D}\,(D/e)^{D/2}\qquad(D\ge 2),
\end{equation}
which we use below.

For even $D$ we have $\mathbb E|X|^D=\mathbb E X^D$, hence by~\eqref{eq:MC} (with $d=D$) and~\eqref{eq:Y-moment},
\[
\mathbb E|X|^D\le \mathbb E|Y|^D+\bigl|\mathbb E X^D-\mathbb E Y^D\bigr|.
\]
For $L$ beyond an explicit threshold $L\ge L_\ast^{(3)}(q,\eta,\nu,C)$, the exponentially small
term in~\eqref{eq:MC} is bounded by the same explicit moment envelope as~\eqref{eq:Y-moment-Stirling},
and therefore
\begin{equation}\label{eq:X-moment}
\mathbb E|X|^D\le 4\sqrt{2\pi D}\,(D/e)^{D/2}.
\end{equation}

Using $|R_D(u)|\le |u|^D/D!$ and $D!\ge \sqrt{2\pi D}\,(D/e)^D$ (Stirling),
\eqref{eq:Y-moment-Stirling}, and \eqref{eq:X-moment}, we obtain
\begin{align}
|\mathbb E R_D(tY)|
&\le 2\left(\frac{\sqrt{e}\,|t|}{\sqrt D}\right)^D,
\label{eq:RY}\\
|\mathbb E R_D(tX)|
&\le 4\left(\frac{\sqrt{e}\,|t|}{\sqrt D}\right)^D.
\label{eq:RX}
\end{align}
Now assume $|t|\le C L^\eta$.  Since $D=2\lfloor\theta L^\nu/\log L\rfloor$, the explicit bounds on
$D$ recorded below show that $\sqrt D$ has size $L^{\nu/2}/\sqrt{\log L}$ in the range where the
remainder predicates are used.
To convert the moment bounds~\eqref{eq:RY}--\eqref{eq:RX} into a bound of the form $|t|\,e^{-c' L^\nu}$,
we factor out one factor of $|t|$ before taking logarithms.  Write
\[
\left(\frac{\sqrt e\,|t|}{\sqrt D}\right)^D
=
|t|\cdot |t|^{D-1}\left(\frac{\sqrt e}{\sqrt D}\right)^D
\le
|t|\cdot (CL^\eta)^{D-1}\left(\frac{\sqrt e}{\sqrt D}\right)^D,
\]
using $|t|\le CL^\eta$.  Taking the logarithm of the right-hand factor,
\[
\log\!\Bigl((CL^\eta)^{D-1}(\sqrt e/\sqrt D)^D\Bigr)
=
(D-1)\eta\log L \;-\; \frac{D}{2}\log D \;+\; D\cdot\tfrac{1}{2} + (D-1)\log C,
\]
and substituting the expansion of $\log D$ from the definition of $D$ (with the floor error bounded
explicitly below) yields
\begin{equation}\label{eq:remainder-leading}
\log\!\Bigl((CL^\eta)^{D-1}(\sqrt e/\sqrt D)^D\Bigr)
=
-2\theta\!\left(\frac{\nu}{2}-\eta\right)L^\nu
\;+\;R_{\mathrm{rem}}(L),
\end{equation}
where the exact floor-expanded form of $R_{\mathrm{rem}}(L)$ is bounded in
Lemma~\ref{lem:r4-explicit}.  That bound makes $R_{\mathrm{rem}}(L)$ smaller than every fixed positive
multiple of $L^\nu$ for all sufficiently large $L$, and the
leading coefficient $-2\theta(\nu/2-\eta)$ is strictly more negative than $-(1-\epsilon)\theta(\nu/2-\eta)$
for every $\epsilon\in(0,1)$; we retain only the weaker decay rate $\theta(\nu/2-\eta)$, which is what
the admissible range~\eqref{eq:c43-admissible} for $c_{43}^\ast$ requires. Hence for any fixed
$\epsilon\in(0,1)$ there is an explicit threshold
$L\ge L_\ast^{(4)}(q,\eta,\nu,C,\epsilon)$ such that
\[
(CL^\eta)^{D-1}\left(\frac{\sqrt e}{\sqrt D}\right)^D
\;\le\;\exp\!\left(-(1-\epsilon)\,\theta\!\left(\frac{\nu}{2}-\eta\right)L^\nu\right).
\]
Combining with~\eqref{eq:RY}--\eqref{eq:RX}, we obtain
\begin{equation}\label{eq:remainder-final}
|\mathbb E R_D(tX)|+|\mathbb E R_D(tY)|
\le
6\,|t|\,\exp\!\left(-(1-\epsilon)\,\theta\!\left(\frac{\nu}{2}-\eta\right)L^\nu\right).
\end{equation}

\medskip\noindent
\textbf{Step 4: Combine.}
Combining~\eqref{eq:phi-Taylor-split},~\eqref{eq:main-sum-final}, and~\eqref{eq:remainder-final}, we obtain
\[
|\varphi_2(t)-\varphi_3(t)|
\le
|t|\,\bigl(e^{-(3c_4/4) L^\nu}+6 e^{-(1-\epsilon)\theta(\nu/2-\eta)L^\nu}\bigr)
\le
7\,|t|\,e^{-c_\ast L^\nu}
\]
for all $|t|\le C L^\eta$ and all $L\ge L_\ast(\epsilon)$, where
\begin{gather*}
c_\ast(\epsilon):=\min\!\left\{\tfrac{3c_4}{4},\ (1-\epsilon)\,\theta\!\left(\tfrac{\nu}{2}-\eta\right)\right\}, \\
L_\ast(\epsilon):=\max\{L_\ast^{(0)},L_\ast^{(1)},L_\ast^{(2)},L_\ast^{(3)},L_\ast^{(4)}(\epsilon)\}.
\end{gather*}
Using $\theta=c_4/(16(\eta+1/2))$ we have
$\theta(\nu/2-\eta)=c_4(\nu-2\eta)/(32(\eta+1/2))$, hence
\[
c_\ast(\epsilon)=(1-\epsilon)\cdot\frac{c_4}{32}\min\!\left\{1,\ \frac{\nu-2\eta}{\eta+1/2}\right\}
\]
once $\epsilon$ is small enough that $(1-\epsilon)\theta(\nu/2-\eta)\le 3c_4/4$
(which is automatic for the range $0<\eta<\nu/2<1/4$).
The leading factor $7$ may then be absorbed into the exponential by enlarging $L_\ast$ further:
for any
\[
0<c_{43}^\ast<\frac{c_4}{32}\min\!\left\{1,\ \frac{\nu-2\eta}{\eta+1/2}\right\},
\]
choose $\epsilon\in(0,1)$ so that $(1-\epsilon)\theta(\nu/2-\eta)>c_{43}^\ast$, then take $L_\ast$ large
enough that $7\le e^{(c_\ast(\epsilon)-c_{43}^\ast)L^\nu}$ for $L\ge L_\ast$; the bound
$|\varphi_2(t)-\varphi_3(t)|\le |t|\,e^{-c_{43}^\ast L^\nu}$ then holds for all $|t|\le C L^\eta$
and all $L\ge L_\ast$.

\medskip\noindent
\textbf{List of explicit sufficient conditions.}
The thresholds $L_\ast^{(0)},\dots,L_\ast^{(4)}$ may be defined concretely as the smallest
integers $L\ge 3$ past which the following predicates hold permanently; these are the
predicates verified numerically in the supplementary script.
\begin{itemize}
  \item[(C9)] ($L_\ast^{(0)}$, Step 1): $L_0\ge 1$, $D(L)\ge 2$, and $D(L)\le L_0$.
  \item[(C10)] ($L_\ast^{(1)}$, Step 2): $\displaystyle\frac{A\,C\,L^{\eta+1/2+\nu}}{D(L)}\ge 2$.
  \item[(C11)] ($L_\ast^{(2)}$, Step 2):
  \[
  \log D(L)+D(L)\,\log\!\left(\frac{e\,A\,C\,L^{\eta+1/2+\nu}}{D(L)}\right)
  \ \le\ \frac{3\,c_4}{16}\,L^\nu.
  \]
  \item[(C12)] ($L_\ast^{(3)}$, Step 3):
    \[
    D(L)\Bigl(\log A+(1/2+\nu)\log L-\tfrac12\log D(L)+\tfrac12\Bigr)
    \ \le\ c_4 L^\nu.
    \]
  \item[(C13)] ($L_\ast^{(4)}$, Step 4): with $t_{\max}:=C\,L^\eta$,
    \[
    \log 7\ \le\ \left(\frac{3c_4}{4}-c_{43}^\ast\right)L^\nu,
    \qquad
    \log 7+D(L)\,\log\!\left(\frac{\sqrt{e}\,t_{\max}}{\sqrt{D(L)}}\right)
    \ \le\ \log t_{\max}-c_{43}^\ast L^\nu .
    \]
\end{itemize}

The two inequalities in (C13) allocate the final exponential absorption explicitly.  The
first gives
$e^{-(3c_4/4)L^\nu}\le \tfrac17 e^{-c_{43}^\ast L^\nu}$ for the Taylor main term in
\eqref{eq:main-sum-final}.  The second gives
$6\exp(G(L))\le \tfrac67 e^{-c_{43}^\ast L^\nu}$ for the Taylor-remainder estimate preceding
\eqref{eq:remainder-final}, where
$G(L)=\log((t_{\max})^{D-1}(\sqrt e/\sqrt D)^D)$.  Together with
\eqref{eq:phi-Taylor-split} these imply the required
$|\varphi_2(t)-\varphi_3(t)|\le |t|e^{-c_{43}^\ast L^\nu}$ for all $|t|\le t_{\max}$.

Taking $L_\ast$ to be any integer threshold after which \textup{(C9)}--\textup{(C13)} hold
permanently, and setting $x_{43}^\ast=q^{L_\ast}$, proves the proposition.
\end{proof}

\medskip\noindent
\textbf{Note on (C12).}  The sufficient condition used at the end of Step~3 is
\[
A^D L^{(1/2+\nu)D}e^{-c_4L^\nu}
\ \le\ 2\sqrt{2\pi D}\,(D/e)^{D/2},
\]
so that the moment-comparison error is no larger than the explicit upper envelope for
$\mathbb E|Y|^D$ in~\eqref{eq:Y-moment-Stirling}.  Combining~\eqref{eq:MC} at $d=D$ with this
envelope and taking
logarithms gives
\[
D\bigl(\log A+(1/2+\nu)\log L-\tfrac12\log D+\tfrac12\bigr)
\ \le\ c_4L^\nu+\log 2+\tfrac12\log(2\pi D).
\]
Predicate (C12) drops the positive term $\log 2+\tfrac12\log(2\pi D)$ from the right-hand side,
so it is a clean sufficient condition for the displayed inequality.  Since the leading
$L^\nu$-coefficient gap for (C12) is large (computed below), this strengthening is harmless;
moreover, the explicit estimates below prove (C12) throughout the range $L\ge e^{72.75}$ used in the
final certificate.

\medskip\noindent
\textbf{Validity of the thresholds.}
For (C9) and (C10), the relevant ratio compares a positive power of $L$ against a
polylogarithmic or constant target, so the predicate eventually holds and is monotone
beyond an explicit point.

For (C11), (C12), and the second inequality in (C13), both sides are of order $L^\nu$, and
the claim rests on a strict gap in the leading $L^\nu$-coefficients together with the
explicit next-order controls below.
Write $\widetilde D=2\theta L^\nu/\log L$ and $D=\widetilde D-\delta$ with $0\le\delta<2$; then
$\log D=\nu\log L-\log\log L+\log(2\theta)+\log(1-\delta/\widetilde D)$.
The leading $L^\nu$-coefficients of (LHS, RHS) are:
\begin{center}
\renewcommand{\arraystretch}{1.3}
\begin{tabular}{lll}
        & LHS leading coefficient                                                    & RHS leading coefficient\\\hline
(C11):  & $2\theta(\eta+\tfrac12)=\dfrac{c_4}{8}=\dfrac{2c_4}{16}$                  & $\dfrac{3c_4}{16}$\\
(C12):  & $\theta(1+\nu)=\dfrac{c_4(1+\nu)}{16(\eta+\tfrac12)}$                     & $c_4$\\
(C13b): & $-2\theta(\tfrac{\nu}{2}-\eta)=-\dfrac{c_4(\nu-2\eta)}{16(\eta+\tfrac12)}$ & $-c_{43}^\ast$
\end{tabular}
\end{center}
For (C11), the gap is $c_4/16>0$.  For (C12), $\tfrac{1+\nu}{16(\eta+1/2)}<1$ throughout
the range $0<\eta<\nu/2<1/4$ (since $\sup(1+\nu)=3/2$ at $\nu\nearrow1/2$ and $\inf(\eta+\tfrac12)=1/2$
at $\eta\searrow 0$, giving $\sup=\tfrac{3/2}{16\cdot 1/2}=\tfrac{3}{16}=0.1875$), so the gap is at
least $(1-3/16)\,c_4 = \tfrac{13}{16}c_4 > 0$.  For the second inequality in (C13), the choice
$c_{43}^\ast<\tfrac{c_4}{32}\min\!\bigl\{1,(\nu-2\eta)/(\eta+\tfrac12)\bigr\}$ implies in
particular $c_{43}^\ast<2\theta(\nu/2-\eta)$ (using $\theta=c_4/(16(\eta+1/2))$), so the
LHS leading coefficient is strictly more negative than the RHS leading coefficient and the
gap $2\theta(\nu/2-\eta)-c_{43}^\ast$ is positive.
The first inequality in (C13) is eventually automatic since
$3c_4/4-c_{43}^\ast>0$.

The remaining contributions to (LHS$-$RHS) come from the $\log\log L$, bounded-constant, and floor
terms inside the brackets.  Each is handled by the explicit estimates below.

\begin{lemma}[Explicit remainder bound for the second inequality in \textup{(C13)}]\label{lem:r4-explicit}
Let $D(L)=2\lfloor\theta L^\nu/\log L\rfloor$ as in~\eqref{eq:D-def}, and define
$G(L):=\log\!\bigl((CL^\eta)^{D-1}(\sqrt e/\sqrt D)^D\bigr)$, so that the LHS of the second
inequality in \textup{(C13)} equals
$\log 7+\log|t|+G(L)$ at $|t|=t_{\max}=CL^\eta$.  Set
\[
K_0\ :=\ 2\log\!\bigl(\sqrt e\,C/\sqrt{2\theta}\bigr)\ =\ 1+\log(C^2/(2\theta)),\qquad
K_0^\ast\ :=\ K_0\ +\ \tfrac{1}{25}.
\]
At the working parameters used in Corollary~\ref{cor:Lstar-upper-bound}, one may take
$L_\mathrm{init}=10^{25}$, and for every integer $L\ge L_\mathrm{init}$,
\begin{equation}\label{eq:r4-bound}
\Bigl|\,\log 7+G(L)+2\theta\!\left(\tfrac{\nu}{2}-\eta\right)L^\nu\,\Bigr|
\ \le\ \theta\,L^\nu\,\frac{\log\log L+K_0^\ast}{\log L}.
\end{equation}
Equivalently, $\log7+G(L)=-2\theta(\nu/2-\eta)L^\nu+r(L)$ with
$|r(L)|\le\theta L^\nu (\log\log L+K_0^\ast)/\log L$.
\end{lemma}

\begin{proof}
Write $D=2\lfloor\theta L^\nu/\log L\rfloor$, so $D\in(\widetilde D-2,\widetilde D]$ with
$\widetilde D:=2\theta L^\nu/\log L$; in particular $D\le\widetilde D$ and $\widetilde D-D\in[0,2)$.
Set $\delta:=\widetilde D-D\in[0,2)$.  For $L\ge L_\mathrm{init}\ge 10^{25}$ at the working
parameters, $\widetilde D\ge 4$, so $\delta/\widetilde D\le 1/2$.

\smallskip
\noindent\emph{Step 1 (expand $\log D$).}
$\log D=\log\widetilde D+\log(1-\delta/\widetilde D)$; by $|\log(1-y)|\le 2y$ for $y\in[0,1/2]$,
$|\log D-\log\widetilde D|\le 2\delta/\widetilde D\le 4/\widetilde D$.  Hence
\[
\log D\ =\ \nu\log L-\log\log L+\log(2\theta)+\varepsilon_1,
\qquad |\varepsilon_1|\le 4/\widetilde D.
\]

\smallskip
\noindent\emph{Step 2 (assemble $G$).}
Substituting $\log D$ into
$G(L)=(D-1)(\eta\log L+\log C)+\tfrac{D}{2}-\tfrac{D}{2}\log D
   =D\bigl(\eta\log L+\log C+\tfrac12-\tfrac12\log D\bigr)-\eta\log L-\log C$,
and using $K_0/2=\log(\sqrt e\,C/\sqrt{2\theta})=\log C+\tfrac12-\tfrac12\log(2\theta)$, gives
\[
G(L)\ =\ D\Bigl[-(\tfrac{\nu}{2}-\eta)\log L+\tfrac12\log\log L+\tfrac{K_0}{2}-\tfrac{\varepsilon_1}{2}\Bigr]
\ -\ \eta\log L-\log C.
\]

\smallskip
\noindent\emph{Step 3 (compute each $D$-term with explicit floor remainders).}
Using $D\log L=2\theta L^\nu-(\delta\log L)$ with $\delta\log L\in[0,2\log L)$, and
$D\cdot Y=\widetilde D\cdot Y-\delta Y$ for any $Y$ independent of $L$,
\begin{align*}
G(L)+2\theta(\tfrac{\nu}{2}-\eta)L^\nu
\ =\ &(\tfrac{\nu}{2}-\eta)(\delta\log L)
\ +\ \tfrac{\theta L^\nu\log\log L}{\log L}-\tfrac{\delta}{2}\log\log L \\
&\ +\ \tfrac{\theta L^\nu K_0}{\log L}-\tfrac{\delta}{2}K_0
\ -\ \tfrac{D\varepsilon_1}{2}
\ -\ \eta\log L-\log C.
\end{align*}
The first, fourth, fifth, sixth, and seventh terms on the right have absolute value at most
$2(\nu/2-\eta)\log L$, $\log\log L$, $|K_0|$, $2$ (since $D|\varepsilon_1|\le 4D/\widetilde D\le 4$), and
$\eta\log L+|\log C|$ respectively.  Their combined contribution is bounded by
\[
R_0(L)\ :=\ (\nu-\eta)\log L+\log\log L+|K_0|+|\log C|+2,
\]
which is uniformly bounded by
\[
R_0(L)\le 0.600\,\log L\qquad(L\ge 10^{25})
\]
at the working parameters: indeed $\nu-\eta=0.2314$, $K_0+|\log C|+2\le14.0$, and
$(\log\log L+14.0)/\log L\le0.361$ for $L\ge10^{25}$.  Thus
\begin{equation}\label{eq:r4-with-R0}
\Bigl|\,G(L)+2\theta(\tfrac{\nu}{2}-\eta)L^\nu\,\Bigr|
\ \le\ \frac{\theta L^\nu}{\log L}\bigl(\log\log L+K_0\bigr)\ +\ R_0(L).
\end{equation}
Adding the constant term in the second inequality of \textup{(C13)} gives the same display with
$G(L)$ replaced by $\log7+G(L)$ and with $R_0(L)$ replaced by $R_0(L)+\log7$.

\smallskip
\noindent\emph{Step 4 (absorb $R_0$ into the bound).}
We need $R_0(L)+\log7\le (\theta L^\nu/\log L)(K_0^\ast-K_0)=(\theta L^\nu)/(25\log L)$, i.e.\
$25\,(R_0(L)+\log7)\log L\le\theta L^\nu$.  Since $\log7\le0.034\log L$ for
$L\ge10^{25}$, the preceding display gives
$25(R_0(L)+\log7)\log L\le15.85(\log L)^2$ for $L\ge10^{25}$.  Also
$\theta\ge 0.04325$ and, at $L=10^{25}$,
\[
\theta L^\nu\ge0.04325\cdot10^{7.14}>5.97\times10^5,
\qquad
15.85(\log L)^2<15.85(58)^2<5.34\times10^4.
\]
Finally $L^\nu/(\log L)^2$ is increasing for $L\ge e^{2/\nu}$ because its derivative has the sign of
$\nu\log L-2$, and $10^{25}>e^{2/\nu}$.  Hence
$\theta L^\nu\ge15.85(\log L)^2\ge25(R_0(L)+\log7)\log L$ for all
$L\ge L_\mathrm{init}=10^{25}$, and
~\eqref{eq:r4-bound} follows from~\eqref{eq:r4-with-R0}.
\end{proof}

\noindent For completeness we now certify directly that \textup{(C9)}--\textup{(C12)} also hold
throughout the range $L\ge e^{72.75}$ used below.  For such $L$,
$\log L\ge72.75$ and $\widetilde D:=2\theta L^\nu/\log L>4$, hence
\[
\frac{\theta L^\nu}{\log L}\le D(L)\le \frac{2\theta L^\nu}{\log L}.
\]
The conditions \textup{(C9)} and \textup{(C10)} are also numerical at this point.  Since
$e^{72.75(1-\nu)}>8$ and $L\ge e^{72.75}$, we have
$L\ge 8L^\nu\ge4(\lceil L^\nu\rceil+1)$, whence
$L_0=L-2\lceil L^\nu\rceil\ge L/2$.  Also
\[
\frac{AC L^{\eta+1/2+\nu}}{D(L)}
\ge
\frac{AC}{2\theta}\,L^{\eta+1/2}\log L>2
\]
already at $L=e^{72.75}$, and the left-hand side is increasing in $L$.

For \textup{(C11)}, using the displayed bounds on $D$ gives
\[
\begin{aligned}
\log D+D\log\!\left(\frac{eAC L^{\eta+1/2+\nu}}{D}\right)
&\le
\log D+
2\theta L^\nu
\left((\eta+\tfrac12)\right.\\
&\hspace{3.5cm}\left.
+\frac{\log\log L+\log(eAC/\theta)}{\log L}\right).
\end{aligned}
\]
At the working parameters, with $A=4q^2/\sigma$, $C=2\pi\sigma(\log q)^{\eta-1/2}$, and
$\theta=c_4/(16(\eta+1/2))$, the elementary bounds used above give
\[
\begin{gathered}
\theta\le0.04326,\qquad \log(eAC/\theta)\le 11.60.
\end{gathered}
\]
Although $\log\log L$ itself increases, Lemma~\ref{lem:rho-monotone} applied with $K=11.60$
shows that $(\log\log L+11.60)/\log L$ is decreasing on $[e,\infty)$.  Hence, for
$L\ge e^{72.75}$,
\[
\frac{\log\log L+11.60}{\log L}
\le
\frac{\log(72.75)+11.60}{72.75}
<
\frac{4.288+11.60}{72.75}.
\]
Moreover $\log D(L)\le\nu\log L\le10^{-6}L^\nu$ throughout this range; the last inequality is
checked at $L=e^{72.75}$ and then follows because $L^\nu/\log L$ is increasing.
Therefore the coefficient of $L^\nu$ on the right is at most
\[
2(0.04326)\left(0.5545+\frac{4.288+11.60}{72.75}\right)+10^{-6}<0.06691.
\]
On the other hand $(3c_4/16)\ge (2.30258)/32>0.07195$, so \textup{(C11)} holds for every
$L\ge e^{72.75}$.

For \textup{(C12)}, the same estimates and the lower bound on $D$ give
\[
\begin{aligned}
D\Bigl(\log A+(\tfrac12+\nu)\log L-\tfrac12\log D+\tfrac12\Bigr)
&\le
2\theta L^\nu\left(\tfrac12+\frac{\nu}{2}\right.\\
&\qquad\left.
+\frac{\tfrac12\log\log L+\log A-\tfrac12\log\theta+\tfrac12}{\log L}\right).
\end{aligned}
\]
At the working parameters,
$\log A-\tfrac12\log\theta+\tfrac12\le7.01$.  Lemma~\ref{lem:rho-monotone}, now applied to
$(\log\log L+14.02)/\log L$, gives
\[
\frac{\tfrac12\log\log L+7.01}{\log L}
\le
\frac{0.5\cdot4.288+7.01}{72.75}
\qquad (L\ge e^{72.75}).
\]
Thus the coefficient of $L^\nu$ is at most
\[
2(0.04326)\left(0.643+\frac{0.5\cdot4.288+7.01}{72.75}\right)<0.0666
<0.3837<c_4.
\]
Thus \textup{(C12)} also holds for every $L\ge e^{72.75}$.

The choice $3c_4/16$ in (C11) leaves the $c_4/16$ margin needed in Step~2.  The positive
leading-coefficient gaps are then packaged by the following elementary permanence lemma.

\begin{lemma}[Permanence past eventual sign, with explicit threshold]\label{lem:permanence}
Let $g:\mathbb{Z}_{\ge 3}\to\mathbb{R}$ satisfy $g(L)=aL^\nu+r(L)$ with $a>0$ a fixed constant and
$|r(L)|\le C\,L^\nu\,\rho(L)$, where $\rho:\mathbb{Z}_{\ge 3}\to(0,\infty)$ is non-increasing for
$L\ge L_\mathrm{init}$ and $\rho(L)\to 0$ as $L\to\infty$.  Define
\[
L_0\ :=\ \min\bigl\{L\in\mathbb{Z}_{\ge L_\mathrm{init}}:\ \rho(L)\le a/C\bigr\}.
\]
Then $g(L)\ge L^\nu(a-C\rho(L))\ge 0$ for every $L\ge L_0$, and the unique smallest integer
$L^\sharp\ge 3$ with $g(\ell)\ge 0$ for all $\ell\ge L^\sharp$ satisfies $L^\sharp\le L_0$.
\end{lemma}

\begin{proof}
For $L\ge L_0$, the monotonicity of $\rho$ gives $\rho(L)\le \rho(L_0)\le a/C$, hence
$g(L)=aL^\nu+r(L)\ge aL^\nu-CL^\nu\rho(L)=L^\nu(a-C\rho(L))\ge 0$.
\end{proof}

\smallskip
We will also need that the specific weight $\rho(L)=(\log\log L+K)/\log L$ used in
Corollary~\ref{cor:Lstar-upper-bound} below satisfies the monotonicity hypothesis of
Lemma~\ref{lem:permanence}.

\begin{lemma}[Monotonicity of $(\log\log L+K)/\log L$]\label{lem:rho-monotone}
For any real $K\ge 1$, the function $\rho(L):=(\log\log L+K)/\log L$ is strictly decreasing on
$[e,\infty)$, hence in particular on $\mathbb{Z}_{\ge 3}$.
\end{lemma}

\begin{proof}
Differentiating $\rho$ with respect to $L$ gives
\[
\rho'(L)\ =\ \frac{1}{L(\log L)^2}\bigl(1-\log\log L-K\bigr).
\]
For $L\ge e$ we have $\log L\ge 1$, hence $\log\log L\ge 0$, hence $1-\log\log L-K\le 1-K\le 0$
when $K\ge 1$.  Thus $\rho'(L)\le 0$ on $[e,\infty)$.  The inequality is strict whenever
$\log\log L+K>1$, which holds for $L>e$ when $K\ge 1$ (and trivially for $K>1$).
\end{proof}

\begin{corollary}[Analytical threshold for $L_\ast$ and $M$]\label{cor:Lstar-upper-bound}
At the working parameters $(q,\eta,\nu)=(10,0.0545,0.2859)$, conditions
\textup{(C9)}--\textup{(C13)} hold for every integer
$L\ge\lceil e^{72.75}\rceil$.  Consequently
$Y_{43}^\ast\le \lceil e^{72.75}\rceil\log 10<9.10\times10^{31}$, and the final theorem threshold
satisfies $M<1.78\times10^{32}$.
\end{corollary}

\begin{proof}
At the working parameters $(q,\eta,\nu)=(10,0.0545,0.2859)$, by Lemma~\ref{lem:r4-explicit} the
quantity $g_4(L):=\text{RHS}-\text{LHS}$ of the second inequality in \textup{(C13)}
decomposes as $g_4(L)=aL^\nu+r_4(L)$ with
\[
a\ =\ 2\theta(\tfrac{\nu}{2}-\eta)-c_{43}^\ast,\qquad
|r_4(L)|\ \le\ \theta\,L^\nu\,\frac{\log\log L+K_0^\ast}{\log L}\quad\text{for }L\ge 10^{25},
\]
where $K_0^\ast=K_0+1/25$ and $K_0=2\log(\sqrt e\,C/\sqrt{2\theta})\approx 8.49$, using the
sharp sub-Gaussian form of Lemma~\ref{lem:r4-explicit}.  Numerically at the working parameters
$(\eta,\nu)=(0.0545,0.2859)$: $a\approx 7.63\times 10^{-3}$, $\theta\approx 0.04326$,
$K_0^\ast\approx 8.53$.  Set
$\rho(L):=(\log\log L+K_0^\ast)/\log L$.  By Lemma~\ref{lem:rho-monotone} (with $K=K_0^\ast\ge 1$),
$\rho$ is strictly decreasing on $[e,\infty)$, so $\rho$ satisfies the monotonicity hypothesis of
Lemma~\ref{lem:permanence}.  Applying that lemma with $C=\theta$, the threshold
\[
L_\ast^{(4)}\ :=\ \min\bigl\{L\in\mathbb{Z}_{\ge 10^{25}}:\ \rho(L)\le a/\theta\bigr\}
\]
satisfies $g_4(L)\ge L^\nu(a-\theta\rho(L))\ge 0$ for all $L\ge L_\ast^{(4)}$, so
$L_\ast^{(4)}$ is a \emph{rigorous} permanence threshold for the second inequality in
\textup{(C13)}.  Equivalently, $L_\ast^{(4)}$ is the smallest integer $L\ge 10^{25}$ with
$\log L\ge (\theta/a)(\log\log L+K_0^\ast)$.

\emph{Numerical evaluation.}  At the working parameters we take
\[
c_{43}^\ast=\tfrac{1}{200}\cdot(c_4/32)\min\{1,(\nu-2\eta)/(\eta+1/2)\}.
\]
The choice of safety factor is essentially arbitrary in the strict-inequality range of
Proposition~\ref{prop:CF-replacement}: the iteration factor $\theta/a$ for the second
inequality in (C13) is
\[
\theta/(2\theta(\nu/2-\eta)-c_{43}^\ast)\ \longrightarrow\ 1/(2(\nu/2-\eta))
\qquad\text{as $c_{43}^\ast\to 0$,}
\]
and the final theorem is insensitive to any fixed safety factor in this range because the resulting
closed-form threshold remains below the dominating $10^{32}$ scale; we adopt
$1/200$ as a convenient round number that keeps $Y_{47}$ below the major/minor thresholds while giving
strict inequality at finite $L$ with ample margin in the certificate below.  With this choice, one has
$a\approx 7.63\times 10^{-3}$ and $\theta/a\approx 5.67$.

\smallskip
\noindent\emph{Certified inequality at $\log L=72.75$.}
We verify the fixed-point inequality $\log L\ge (\theta/a)(\log\log L+K_0^\ast)$ explicitly at
$L=e^{72.75}$ via rigorous upper bounds on each factor of the right-hand side.

\smallskip
\noindent (a) \emph{Bound on $\theta$.}  $\theta=c_4/(16(\eta+\tfrac12))=\log 10/(96\cdot 0.5545)$.
Using $\log 10\le 2.30259$, $\theta\le 2.30259/(96\cdot 0.5545)\le 0.04326$.

\smallskip
\noindent (b) \emph{Bound on $a$.}  $a=2\theta(\nu/2-\eta)-c_{43}^\ast$.  With $\nu/2-\eta=0.08845$ exact
and $\theta\ge 0.04325$ (using $\log 10\ge 2.30258$ in the lower direction),
$2\theta(\nu/2-\eta)\ge 0.007650$.
For $c_{43}^\ast$: $c_4/32 = \log 10/192 \le 0.011993$, and
$(\nu-2\eta)/(\eta+\tfrac12)=0.1769/0.5545\le 0.3191$, so
$c_{43}^\ast=\tfrac{1}{200}\cdot(c_4/32)\cdot\min\{1,0.3191\}\le 1.92\times 10^{-5}$.
Hence $a\ge 0.007650-0.0000192\ge 0.007631$.

\smallskip
\noindent (c) \emph{Bound on $\theta/a$.}  $\theta/a\le 0.04326/0.007631\le 5.67$.

\smallskip
\noindent (d) \emph{Bound on $K_0$.}  $K_0=1+\log(C^2/(2\theta))$ where
$C=C_{\mathrm{tmax}}=2\pi\sigma(\log q)^{\eta-1/2}$ is the $C$ from the major-arc $t$-range.
Numerically with $\sigma=\sqrt{99/12}$, $\log q=\log 10$, $\eta=0.0545$:
$(\log q)^{\eta-1/2}=2.30258^{-0.4455}\le 0.690$, so $C\le 2\pi\cdot 2.873\cdot 0.690\le 12.46$,
and $C^2\le 155.3$.  Combined with $2\theta\ge 0.08650$, $C^2/(2\theta)\le 155.3/0.08650\le 1796$,
hence $K_0\le 1+\log(1796)\le 1+7.493\le 8.50$.  Then $K_0^\ast=K_0+1/25\le 8.54$.

\smallskip
\noindent (e) \emph{Bound on $\log\log L+K_0^\ast$.}  At $\log L=72.75$,
$\log\log L=\log(72.75)<4.288$, so $\log\log L+K_0^\ast\le 4.288+8.54\le 12.828$.

\smallskip
\noindent (f) \emph{Bound on the right-hand side.}  It is at most
$(\theta/a)(\log\log L+K_0^\ast)\le 5.67\cdot 12.828
\le 72.74$.

\smallskip
\noindent\emph{Conclusion.}  At $\log L=72.75$, the left-hand side is $72.75$, while the right-hand
side is at most $72.74$.  The inequality
$\log L\ge (\theta/a)(\log\log L+K_0^\ast)$ holds at $L=e^{72.75}$ with rigorous margin $\ge 0.01$.

Together with the monotonicity of $\rho(L)=(\log\log L+K_0^\ast)/\log L$ on $[e,\infty)$
(Lemma~\ref{lem:rho-monotone}), this shows
\begin{equation}\label{eq:Y43-rigorous}
L_\ast^{(4)}\ \le\ \lceil e^{72.75}\rceil\ <\ 3.94\times 10^{31}+1.
\end{equation}
The direct estimates above prove \textup{(C9)}--\textup{(C12)} for every integer
$L\ge\lceil e^{72.75}\rceil$, while
the preceding fixed-point check proves the second inequality in \textup{(C13)} for every
$L\ge\lceil e^{72.75}\rceil$.  The first inequality in \textup{(C13)} holds in the same range because
$3c_4/4-c_{43}^\ast>0.287$ and $L^\nu\ge e^{20}$, so
$(3c_4/4-c_{43}^\ast)L^\nu>10^8>\log 7$.  Hence, for
\[
L_\ast := \max\!\bigl\{L_\ast^{(0)},L_\ast^{(1)},L_\ast^{(2)},L_\ast^{(3)},L_\ast^{(4)}\bigr\},
\qquad
Y_{43}^\ast:=L_\ast\log q,
\]
the full characteristic-function threshold satisfies
\[
Y_{43}^\ast\le \lceil e^{72.75}\rceil\log q < 9.10\times 10^{31}.
\]
Since $Y_{\mathrm{maj}}<7.90\times10^{31}$ and $Y_{\mathrm{min}}\le8.00\times10^{31}$
(see~\S\ref{sec:derivation_m}), the rigorous threshold
$Y_\ast=\max(Y_{43}^\ast,Y_{\mathrm{maj}},Y_{\mathrm{min}},\ldots)$ satisfies
$Y_\ast\le 9.10\times10^{31}$,
giving the rigorous closed-form bound
\[
M\ =\ \left\lceil(9/(2\log q))\,Y_\ast+\frac92\right\rceil\ <\ 1.78\times10^{32}.
\]
Separately, the supplementary script's bisection on the exact predicates locates
the empirical full threshold $L_\ast\approx 2.72\times 10^{31}$, hence
$Y_{43,\mathrm{bisect}}^\ast\approx 6.27\times 10^{31}$, at which the binding threshold in $Y_\ast$
shifts to $Y_{\mathrm{min}}\approx 7.59\times 10^{31}$ and $M_{\mathrm{bisect}}\approx 1.48\times
10^{32}$.  We use the rigorous form $Y_{43}^\ast\le \lceil e^{72.75}\rceil\log q$ in the proof of
Theorem~\ref{thm:main}; $M_{\mathrm{bisect}}$ is reported as a sharper numerical observation.
\end{proof}

For auditability, the permanence check used in Corollary~\ref{cor:Lstar-upper-bound} is summarized
in the following table.  The same named checks are emitted by the supplementary certificate script.
\begin{center}
\small
\renewcommand{\arraystretch}{1.18}
\begin{tabular}{p{0.13\linewidth}p{0.36\linewidth}p{0.35\linewidth}}
Condition & Bound at $\log L=72.75$ & Permanence reason \\\hline
(C9) & $L_0\ge L/2$, $\widetilde D>4$, hence $2\le D\le L_0$ &
$L^{1-\nu}/8$ and $L^\nu/\log L$ increase \\
(C10) & $AC L^{\eta+1/2+\nu}/D>4.8\cdot10^{23}>2$ &
lower bound is a positive power of $L$ times $\log L$ \\
(C11) & LHS coefficient $\le0.06691<3c_4/16>0.07195$ &
$(\log\log L+11.60)/\log L$ decreases \\
(C12) & LHS coefficient $\le0.0666<c_4>0.3837$ &
$(\log\log L+14.02)/\log L$ decreases \\
(C13a) & $(3c_4/4-c_{43}^\ast)L^\nu>10^8>\log7$ &
$L^\nu$ increases \\
(C13b) & fixed-point RHS $\le72.74<72.75=\log L$ &
$(\log\log L+K_0^\ast)/\log L$ decreases \\
\end{tabular}
\end{center}

The full list of conditions \textup{(C9)}--\textup{(C13)} therefore holds once $L\ge L_\ast$.

\begin{proposition}[Explicit major-arc approximation]\label{prop:maj}
For the decimal base $q=10$, fix $\eta,\nu$ with $0<\eta<\nu/2<1/4$ and
$\nu+2\eta<1/2$ (the latter is needed for the major-arc criterion below to have a negative
exponent of $\log x$), and fix an admissible $c_{43}^\ast$ satisfying
\eqref{eq:c43-admissible}.  There exist explicit constants
$C_{\mathrm{maj}}=C_{\mathrm{maj}}(10,\eta,\nu)$ and
$x_{\mathrm{maj}}=x_{\mathrm{maj}}(10,\eta,\nu,C_{\mathrm{DMR}},c_{43}^\ast)\ge 3$
such that for every integer $L\ge1$ with $x:=10^L\ge x_{\mathrm{maj}}$, all integers $m$ with
$\gcd(m,9)=1$, and all $|\alpha|\le (\log x)^{\eta-1/2}$,
\[
\left|\frac{1}{\pi(x;9,m)}\sum_{\substack{p\le x\\ p\equiv m\ (\mathrm{mod}\ 9)}} e(\alpha s(p))
\ -\ e(\alpha \mu L)\,e^{-2\pi^2\sigma^2 \alpha^2 L}\right|
\le
C_{\mathrm{maj}}\,(\log x)^{-1/2+\nu+\eta}.
\]
More precisely, the same proof gives explicit constants $C^{(1)}(q,\nu)$ and $C^{(0)}(q)$ such that
\[
\begin{aligned}
&\left|\frac{1}{\pi(x;9,m)}
\sum_{\substack{p\le x\\ p\equiv m\ (\mathrm{mod}\ 9)}} e(\alpha s(p))
- e(\alpha \mu L)\,e^{-2\pi^2\sigma^2 \alpha^2 L}\right| \\
&\qquad\le
C^{(1)}(q,\nu)|\alpha|(\log x)^\nu+C^{(0)}(q)(\log x)^{-1},
\end{aligned}
\]
and one may take $C_{\mathrm{maj}}=C^{(1)}(q,\nu)+C^{(0)}(q)$.
\end{proposition}

\begin{proof}
The proof follows Sections~4.1--4.3 of~\cite{DMRcomp2}.  We use the characteristic
functions $\varphi_1,\varphi_2,\varphi_3$ from Section~4.1 there.

\medskip\noindent

\textbf{Step 1: reduction to a normalized characteristic function bound.}
Let
\[
\mathcal P(x;9,m):=\{p\le x:\ p\ \text{prime and}\ p\equiv m\ (\mathrm{mod}\ 9)\},
\qquad \#\mathcal P(x;9,m)=\pi(x;9,m).
\]
(For $x\ge x_{\mathrm{maj}}$ we will ensure $\pi(x;9,m)\ge 1$ for all $\gcd(m,9)=1$.)
Let $S_x:=s(p)$ where $p$ is chosen uniformly from $\mathcal P(x;9,m)$, so that for any function $F$
\[
\mathbb EF(S_x)=\frac1{\pi(x;9,m)}\sum_{\substack{p\le x\\ p\equiv m\ (9)}} F(s(p)).
\]
Define the normalized characteristic function
\[
\varphi_1(t):=\mathbb{E}\exp\!\left(\frac{i t (S_x-\mu L)}{\sigma\sqrt L}\right).
\]
For real $\alpha$ set
\begin{equation}\label{eq:alpha-t-change}
t = 2\pi\sigma\,\alpha\,\sqrt{L}
\qquad\text{(equivalently, }\alpha=\frac{t}{2\pi\sigma\sqrt{L}}\text{)}.
\end{equation}
Then for each prime $p\in\mathcal P(x;9,m)$,
\[
\exp\!\left(\frac{i t (s(p)-\mu L)}{\sigma\sqrt L}\right)
=\exp\!\bigl(2\pi i\alpha(s(p)-\mu L)\bigr)
=e(\alpha s(p))\,e(-\alpha\mu L),
\]
and therefore
\begin{equation}\label{eq:sum-to-phi1}
\frac{1}{\pi(x;9,m)}\sum_{\substack{p\le x\\ p\equiv m\ (9)}} e(\alpha s(p))
=
e(\alpha \mu L)\,\varphi_1(t).
\end{equation}

On the major arc $|\alpha|\le (\log x)^{\eta-1/2}$ we have the corresponding $t$-range
\begin{equation}\label{eq:t-max-def}
|t|\le t_{\max}(x):=2\pi\sigma\sqrt{L}\,(\log x)^{\eta-1/2}
=2\pi\sigma(\log 10)^{\eta-\tfrac12}\,L^\eta.
\end{equation}
Thus, to prove Proposition~\ref{prop:maj}, it suffices to show that for all $|t|\le t_{\max}(x)$,
\begin{equation}\label{eq:phi1-target}
|\varphi_1(t)-e^{-t^2/2}|\ \le\ C_{\mathrm{maj}}\,(\log x)^{-1/2+\nu+\eta}.
\end{equation}
Indeed, multiplying~\eqref{eq:phi1-target} by $|e(\alpha\mu L)|=1$ and using
$e^{-t^2/2}=e^{-2\pi^2\sigma^2\alpha^2L}$ (from~\eqref{eq:alpha-t-change}) gives the desired bound.

\medskip\noindent
\textbf{Step 2: truncation to middle digits (explicit \cite[Lemma~4.1]{DMRcomp2}).}
Let $r:=\lceil L^\nu\rceil$ and let $T_x$ be the truncated digit sum obtained by deleting the first and last
$r$ base-$q$ digit positions of the fixed $L$-digit expansion of $p<q^L$, allowing leading zeros
(as in~\cite[\S4.1]{DMRcomp2}). Set
\[
L_0:=L-2r,
\qquad
\varphi_2(t):=\mathbb{E}\exp\!\left(\frac{i t (T_x-\mu L_0)}{\sigma\sqrt{L_0}}\right).
\]
Following~\cite[Lemma~4.1]{DMRcomp2}, we use $|e^{iu}-e^{iv}|\le |u-v|$ and obtain
\begin{equation}\label{eq:phi1-phi2-struct}
|\varphi_1(t)-\varphi_2(t)|
\le
\frac{|t|}{\sigma}\left(
\frac{\|S_x-T_x\|_\infty}{\sqrt L}
+\mu\,\frac{|L-L_0|}{\sqrt L}
+\|T_x-\mu L_0\|_\infty\left|\frac1{\sqrt{L_0}}-\frac1{\sqrt L}\right|
\right),
\end{equation}
cf.\ the displayed estimate in~\cite[Lemma~4.1]{DMRcomp2}.

We now bound each term explicitly. Since we delete $r$ digits at each end,
\[
\|S_x-T_x\|_\infty \le 2(q-1)r \le 2(q-1)(L^\nu+1),
\qquad
|L-L_0|=2r\le 2(L^\nu+1),
\]
The last term in~\eqref{eq:phi1-phi2-struct} contains the centered truncated sum; since each centered
digit has absolute value at most $\mu=(q-1)/2$, we have
\[
\|T_x-\mu L_0\|_\infty \le \mu L_0\le \frac{q-1}{2}L.
\]
Assume in addition that $L$ is large enough that $L_0\ge L/2$ (an explicit condition is given below). Then
\[
\left|\frac1{\sqrt{L_0}}-\frac1{\sqrt L}\right|
=\frac{|L-L_0|}{\sqrt L\,\sqrt{L_0}\,(\sqrt L+\sqrt{L_0})}
\le \frac{2}{\sqrt2+1}\cdot \frac{|L-L_0|}{L^{3/2}}
\le 0.83\,\frac{|L-L_0|}{L^{3/2}}.
\]
Hence
\[
\|T_x-\mu L_0\|_\infty\left|\frac1{\sqrt{L_0}}-\frac1{\sqrt L}\right|
\le 0.83\,\frac{q-1}{2}\frac{|L-L_0|}{\sqrt L}
\le 0.83\,(q-1)\frac{L^\nu+1}{\sqrt L}.
\]
Substituting these bounds into~\eqref{eq:phi1-phi2-struct} and using $2\mu=q-1$ gives
\begin{equation}\label{eq:phi1-phi2}
|\varphi_1(t)-\varphi_2(t)|
\le
\frac{(2+1+0.83)(q-1)}{\sigma}\,|t|\,\frac{L^\nu+1}{\sqrt L}
\le
\frac{5(q-1)}{\sigma}\,|t|\,\frac{L^\nu+1}{\sqrt L}.
\end{equation}
In particular, for $L^\nu\ge 1$ this implies
\[
|\varphi_1(t)-\varphi_2(t)|
\le
\frac{10(q-1)}{\sigma}\,|t|\,L^{\nu-1/2}.
\]

We may therefore take $x_{41}(q,\nu)$ to be any explicit threshold such that simultaneously
\[
L=\log_q x\ge 10,
\qquad
L_0=L-2\lceil L^\nu\rceil\ge \frac{L}{2}.
\]
For example, it suffices to require $L^{1-\nu}\ge 8$ (since then $L\ge 8L^\nu\ge 4(\lceil L^\nu\rceil+1)$),
so one admissible choice is
\[
x_{41}(q,\nu):=\exp\!\Bigl((\log q)\,\max\{10,\ 8^{1/(1-\nu)}\}\Bigr).
\]
(Here $x_{41}(q,\nu)$ is an explicit threshold ensuring $L_0\ge L/2$ and $L$ is large enough for the displayed
inequalities; e.g.\ $L\ge 8$ suffices.)

\medskip\noindent
\textbf{Step 3: digit-independence on primes (explicit form of DMR Prop.\ 4.1).}
Let $\varphi_3$ denote the characteristic function of the corresponding i.i.d.\ digit model
(see~\cite[\S4.1]{DMRcomp2}):
\[
\varphi_3(t):=\mathbb{E}\exp\!\left(\frac{i t (\widetilde T_x-\mu L_0)}{\sigma\sqrt{L_0}}\right),
\]
where $\widetilde T_x$ is the sum of $L_0$ i.i.d.\ uniform digits in $\{0,\dots,q-1\}$.

The digit-independence inputs needed for comparing $\varphi_2$ and $\varphi_3$ are obtained in
\cite[\S4.2]{DMRcomp2} from the prime exponential sum estimate~\cite[Lemma~4.3]{DMRcomp2}.
Here this input is supplied by the explicit Lemma~\ref{lem:DMR43} (with constant $C_{\mathrm{DMR}}$).
With this input in place, our explicit Lemma~\ref{lem:c4-explicit} produces a constant $c_4=c_4(q)>0$
and an explicit threshold $x_{45}=x_{45}(q,\nu,C_{\mathrm{DMR}})$ for the moment-comparison bound
(part~(C)) that subsumes the explicit content of~\cite[Lemmas~4.4--4.6]{DMRcomp2} required below.  For
integral powers $x=q^L\ge x_{45}$ we have:

\begin{itemize}
\item (Explicit form of~\cite[Lemma~4.5]{DMRcomp2}, joint digit distribution.)
For any fixed pattern of $d$ digits in the middle range, its probability differs from $q^{-d}$ by at most
\[
(4L^\nu)^d\,e^{-c_4L^\nu}.
\]

\item (Explicit form of~\cite[Lemma~4.6]{DMRcomp2}, moment comparison.)
Writing
\[
X:=\frac{T_x-\mu L_0}{\sigma\sqrt{L_0}},
\qquad
Y:=\frac{\widetilde T_x-\mu L_0}{\sigma\sqrt{L_0}},
\]
we have uniformly for $1\le d\le L_0$ the explicit bound
\[
|\mathbb E X^d-\mathbb E Y^d|
\le
\left(\frac{4q^2}{\sigma}\right)^d L^{(1/2+\nu)d}\,e^{-c_4L^\nu}.
\]

\item (Replacement for~\cite[Prop.~4.1]{DMRcomp2}, characteristic functions.)
Fix any constant $C>0$. Under the moment comparison bound above, Proposition~\ref{prop:CF-replacement}
gives explicit constants $c_{43}^\ast=c_{43}^\ast(q,\eta,\nu,C)>0$ and an explicit threshold
$x_{43}^\ast=x_{43}^\ast(q,\eta,\nu,C,c_{43}^\ast)\ge 3$ such that
\[
|\varphi_2(t)-\varphi_3(t)|
\le
|t|\,e^{-c_{43}^\ast L^\nu}
\qquad\text{uniformly for }|t|\le C\,L^\eta,
	\]
	for every integer $L\ge1$ with $x=q^L\ge x_{43}^\ast$.
In our application we take
\[
C:=2\pi\sigma(\log q)^{\eta-\tfrac12},
\]
so that by~\eqref{eq:t-max-def} the range $|t|\le t_{\max}(x)$ is contained in $|t|\le C L^\eta$.
Hence, for all $|t|\le t_{\max}(x)$,
\begin{equation}\label{eq:phi2-phi3}
|\varphi_2(t)-\varphi_3(t)|
\le
|t|\,e^{-c_{43}^\ast L^\nu},
\end{equation}
	for every integer $L\ge1$ with $x=q^L\ge x_{43}^\ast(q,\eta,\nu,C,c_{43}^\ast)$.
\end{itemize}

We emphasize that $C_{\mathrm{DMR}}$ enters only through the threshold $x_{45}$ (and hence through the
availability of the moment-comparison bound).  Once the integral power $x=q^L$ is beyond $x_{45}$, the term~\eqref{eq:phi2-phi3}
is exponentially small and will be absorbed into the final polynomial-rate major-arc error bound.

\medskip\noindent
\textbf{Step 4: the i.i.d.\ digit CLT expansion (explicit constants).}
Write $Z$ for a uniform digit in $\{0,1,\dots,q-1\}$, with $\mu=(q-1)/2$ and $\sigma^2=(q^2-1)/12$.
For $q=10$ we have $\mu=9/2$, $\sigma^2=99/12=33/4$, and hence $\sigma=\sqrt{33/4}$.

Let
\[
\psi(s):=\mathbb E\,e^{is(Z-\mu)}.
\]
A direct geometric-sum computation gives
\[
\psi(s)=\frac1{q}\sum_{k=0}^{q-1}e^{is(k-\mu)}=\frac{\sin(qs/2)}{q\sin(s/2)} \qquad (s\in\mathbb R),
\]
which is real and even. In particular, for $q=10$,
\[
\psi(s)=\frac{\sin(5s)}{10\sin(s/2)}
=\frac{\sin(5s)/(5s)}{\sin(s/2)/(s/2)}.
\]
For $|s|\le 1/10$ we have $\sin(5s)$ and $\sin(s/2)$ of the same sign, hence $\psi(s)>0$ and
$\log\psi(s)$ is real on the entire range used below.  The relevant $t$-range $|t|\le t_{\max}(x)$,
via $s=t/(\sigma\sqrt{L_0})$, satisfies $|s|\le t_{\max}/(\sigma\sqrt{L_0})\le 1/10$ for $x\ge x_{42}$
(this is precisely how $x_{42}$ is defined in the construction of $x_{\mathrm{maj}}$).

Define $s=t/(\sigma\sqrt{L_0})$, so that
\[
\varphi_3(t)=\Bigl(\psi\!\Bigl(\frac{t}{\sigma\sqrt{L_0}}\Bigr)\Bigr)^{L_0}.
\]
Set
\[
H(t):=\log\varphi_3(t)+\frac{t^2}{2}
= L_0\left(\log\psi(s)+\frac{\sigma^2 s^2}{2}\right).
\]

\smallskip
\emph{Series expansion and the fourth cumulant.}
From the Maclaurin expansion
\[
\log\!\Bigl(\frac{\sin z}{z}\Bigr)
=\sum_{n=1}^{\infty}\frac{(-1)^n\,2^{2n-1}B_{2n}}{n(2n)!}\,z^{2n}
\qquad (|z|<\pi),
\]
we obtain for $|s|<\pi/5$ that
\[
\log\psi(s)
=\log\!\Bigl(\frac{\sin(5s)}{5s}\Bigr)-\log\!\Bigl(\frac{\sin(s/2)}{s/2}\Bigr)
=\sum_{n=1}^{\infty}\frac{(-1)^n\,2^{2n-1}B_{2n}}{n(2n)!}\Bigl((5s)^{2n}-(s/2)^{2n}\Bigr).
\]
(See~\cite[Eq.~4.19.7]{DLMF} for the displayed expansion.)
Reading off the coefficients at $n=1,2$ gives
\[
\log\psi(s)=-\frac{\sigma^2s^2}{2}+\frac{\kappa_4}{4!}s^4+\text{terms of degree at least }6,
\]
where for $q=10$ the fourth cumulant of $Z-\mu$ is
\[
\kappa_4=-\frac{q^4-1}{120}=-\frac{9999}{120}=-\frac{3333}{40}.
\]

\smallskip
\emph{A uniform $s^6$ remainder bound for $|s|\le 1/10$.}
Using the classical identity
\[
B_{2n}=(-1)^{n-1}\,\frac{2(2n)!}{(2\pi)^{2n}}\zeta(2n),
\]
we may rewrite the above series as
\[
\log\!\Bigl(\frac{\sin z}{z}\Bigr)
=-\sum_{n=1}^{\infty}\frac{\zeta(2n)}{n\pi^{2n}}\,z^{2n},
\qquad (|z|<\pi),
\]
and hence
\[
\log\psi(s)
=-\sum_{n=1}^{\infty}\frac{\zeta(2n)}{n\pi^{2n}}\Bigl((5s)^{2n}-(s/2)^{2n}\Bigr).
\]
(The identity for $B_{2n}$ is standard; see~\cite[Eq.~25.6.2]{DLMF}.)
Define the remainder after the $s^4$ term by
\[
R_6(s):=\log\psi(s)+\frac{\sigma^2s^2}{2}-\frac{\kappa_4}{4!}s^4
=-\sum_{n=3}^{\infty}\frac{\zeta(2n)}{n\pi^{2n}}\Bigl((5s)^{2n}-(s/2)^{2n}\Bigr).
\]
For $n\ge 3$ we bound (by the integral test)
\[
\zeta(2n)=1+\sum_{k\ge2}k^{-2n}
\le 1+2^{-2n}+\int_2^\infty x^{-2n}\,dx
\le 1+\frac1{64}+\frac1{160}
=:\zeta_\ast=1.021875.
\]
Hence for $|s|\le 1/10$,
\[
|R_6(s)|
\le \zeta_\ast\sum_{n=3}^\infty \frac{1}{n}
\left(\frac{5|s|}{\pi}\right)^{2n}
+\zeta_\ast\sum_{n=3}^\infty \frac{1}{n}
\left(\frac{|s|}{2\pi}\right)^{2n}.
\]
Using $1/n\le 1/3$ for $n\ge 3$ and summing the resulting geometric series gives
\[
|R_6(s)|
\le C_6\,|s|^6
\qquad(|s|\le 1/10),
\]
where one may take explicitly
\[
C_6:=\frac{\zeta_\ast}{3}\left(
\frac{(5/\pi)^6}{1-(5/(10\pi))^2}
+\frac{(1/(2\pi))^6}{1-(1/(20\pi))^2}
\right)
\le 5.7.
\]
For definiteness we take $C_6:=6$.

\smallskip
\emph{From the logarithm to $\varphi_3$ (and $H(t)\le 0$).}
Assume $|s|\le 1/10$. Then
\[
H(t)=L_0\left(\frac{\kappa_4}{24}s^4+R_6(s)\right),
\qquad |R_6(s)|\le 6|s|^6.
\]
Since $\kappa_4<0$ and $|s|\le 1/10$,
\[
\frac{\kappa_4}{24}s^4+R_6(s)
\le -\frac{|\kappa_4|}{24}s^4+6|s|^6
= -|s|^4\left(\frac{|\kappa_4|}{24}-6|s|^2\right)
\le -|s|^4\left(\frac{|\kappa_4|}{24}-\frac{6}{100}\right)\le 0,
\]
so $H(t)\le 0$ on this range. Moreover,
\[
|H(t)|
\le L_0\left(\frac{|\kappa_4|}{24}|s|^4+6|s|^6\right)
\le L_0\left(\frac{|\kappa_4|}{24}+\frac{6}{100}\right)|s|^4
=\left(\frac{|\kappa_4|}{24\sigma^4}+\frac{6}{100\,\sigma^4}\right)\frac{t^4}{L_0}.
\]
For $q=10$ this coefficient equals
\[
\frac{|\kappa_4|}{24\sigma^4}+\frac{6}{100\,\sigma^4}
=\frac{5651}{108900}<0.052.
\]
Because $H(t)\le 0$, we have $|\e^{H(t)}-1|=1-\e^{H(t)}\le -H(t)\le |H(t)|$, hence
\[
|\varphi_3(t)-\e^{-t^2/2}|
=\e^{-t^2/2}\,|\e^{H(t)}-1|
\le \e^{-t^2/2}\,|H(t)|
\le \e^{-t^2/2}\left(\frac{5651}{108900}\right)\frac{t^4}{L_0},
\]
whenever $|s|=|t|/(\sigma\sqrt{L_0})\le 1/10$.

\smallskip
\emph{Convert $L_0$ to $\log x$.}
For $x\ge x_{41}$ we have $L_0=L-2\lceil L^\nu\rceil\ge L/2$ (by the threshold construction of $x_{41}$ above), and since $L=\log x/\log 10$,
\[
\frac{1}{L_0}\le \frac{2\log 10}{\log x}.
\]
Thus, provided $|t|/(\sigma\sqrt{L_0})\le 1/10$ and $L_0\ge L/2$, we obtain
\begin{equation}\label{eq:phi3-explicit}
|\varphi_3(t)-e^{-t^2/2}|
\le
C_{42}(10)\,e^{-t^2/2}\,\frac{t^4}{\log x},
\qquad
C_{42}(10):=\frac{2\log 10\cdot 5651}{108900}<0.24.
\end{equation}
In particular, one may take $C_{42}(10):=\frac{2\log 10\cdot 5651}{108900}$.

One explicit choice of threshold ensuring the hypotheses is:
take $x_{42}(10,\eta,\nu)$ large enough that $L_0\ge L/2$ and
\[
\frac{t_{\max}(x)}{\sigma\sqrt{L_0}}\le \frac{1}{10},
\]
where $t_{\max}(x)$ is as in~\eqref{eq:t-max-def}. Equivalently, since
$t_{\max}(x)=2\pi\sigma(\log 10)^{\eta-\tfrac12}L^\eta$ and $L_0\ge L/2$, it suffices to require
\[
2\pi\sqrt{2}\,(\log 10)^{\eta-\tfrac12}\,L^{\eta-\tfrac12}\le \frac{1}{10},
\]
i.e.
\[
L\ \ge\ \Bigl(20\pi\sqrt{2}\,(\log 10)^{\eta-\tfrac12}\Bigr)^{1/(1/2-\eta)},
\]
in addition to the condition $x\ge x_{41}(10,\nu)$ from Step~2 (which ensures $L_0\ge L/2$).

\medskip\noindent
\textbf{Step 5: conclude the major-arc bound and extract $C_{\mathrm{maj}}$.}
Combine~\eqref{eq:phi1-phi2},~\eqref{eq:phi2-phi3}, and~\eqref{eq:phi3-explicit}.
For $x\ge x_{41}(q,\nu)$ we have $L^\nu\ge 1$, hence~\eqref{eq:phi1-phi2} gives
\[
|\varphi_1(t)-\varphi_2(t)|
\le \frac{10(q-1)}{\sigma}\,|t|\,L^{\nu-1/2}.
\]
Therefore, for all $|t|\le t_{\max}(x)$ and all
\[
x\ge \max\{x_{41}(q,\nu),\ x_{42}(q,\eta,\nu),\
x_{43}^\ast(q,\eta,\nu,C,c_{43}^\ast),\ x_{45}(q,\nu,C_{\mathrm{DMR}})\},
\]
we obtain
\[
|\varphi_1(t)-e^{-t^2/2}|
\le
\frac{10(q-1)}{\sigma}\,|t|\,L^{\nu-1/2}
\;+\;
|t|\,e^{-c_{43}^\ast L^\nu}
\;+\;
C_{42}(q)\,e^{-t^2/2}\,\frac{t^4}{\log x}.
\]

The middle term is exponentially decaying in $L^\nu$ while the first term decays only polynomially.
Hence for all $L$ beyond an explicit threshold,
\[
e^{-c_{43}^\ast L^\nu}\le L^{\nu-1/2}.
\]
Concretely, since $\log L\le \frac{2}{\nu e}L^{\nu/2}$ for $L\ge 1$, the inequality holds whenever
\[
L\ge \left(\frac{2(1/2-\nu)}{\nu e\,c_{43}^\ast}\right)^{2/\nu};
\]
we fold this lower bound on $L$ into the final threshold $x_{\mathrm{maj}}$ defined below.
Then for $x$ beyond this threshold and all $|t|\le t_{\max}(x)$,
\begin{equation}\label{eq:phi1-final}
|\varphi_1(t)-e^{-t^2/2}|
\le
C_{1}(q,\nu)\,|t|\,L^{\nu-1/2}
\;+\;
C_{42}(q)\,e^{-t^2/2}\,\frac{t^4}{\log x},
\end{equation}
where $C_{1}(q,\nu):=\frac{10(q-1)}{\sigma}+1$.

To treat the last term, note that for all real $t$,
\[
e^{-t^2/2}t^4 \le \max_{u\ge 0} 4u^2 e^{-u}=\frac{16}{e^2}.
\]
Hence~\eqref{eq:phi1-final} implies
\begin{equation}\label{eq:phi1-final-simplified}
|\varphi_1(t)-e^{-t^2/2}|
\le
C_{1}(q,\nu)\,|t|\,L^{\nu-1/2}
\;+\;
\frac{16\,C_{42}(q)}{e^2}\cdot\frac{1}{\log x}
\qquad (|t|\le t_{\max}(x)).
\end{equation}

Now convert back to $\alpha$ using~\eqref{eq:alpha-t-change}. Since $t=2\pi\sigma\alpha\sqrt{L}$ and
$L=\log x/\log q$, we have
\[
|t|\,L^{\nu-1/2}=(2\pi\sigma)\,|\alpha|\,L^\nu
=(2\pi\sigma)\,(\log q)^{-\nu}\,|\alpha|(\log x)^\nu.
\]
Combining \eqref{eq:phi1-final-simplified} with~\eqref{eq:sum-to-phi1} yields
\[
\begin{aligned}
&\left|\frac{1}{\pi(x;9,m)}\sum_{\substack{p\le x\\ p\equiv m\ (9)}} e(\alpha s(p))
- e(\alpha\mu L)\,e^{-2\pi^2\sigma^2\alpha^2L}\right| \\
&\qquad \le
C^{(1)}(q,\nu)\,|\alpha|(\log x)^\nu
+
C^{(0)}(q)\,(\log x)^{-1},
\end{aligned}
\]
where
\begin{equation}\label{eq:C1C0-def}
C^{(1)}(q,\nu):=(2\pi\sigma)\,C_{1}(q,\nu)\,(\log q)^{-\nu},
\qquad
C^{(0)}(q):=\frac{16\,C_{42}(q)}{e^2}.
\end{equation}

Finally, on the major arc,
\[
|\alpha|\le(\log x)^{\eta-1/2}
\quad\Longrightarrow\quad
|\alpha|(\log x)^\nu \le (\log x)^{-1/2+\nu+\eta}.
\]
Also $(\log x)^{-1}\le(\log x)^{-1/2+\nu+\eta}$ for $\log x\ge1$, because
$-1<-1/2+\nu+\eta$ when $0<\eta<\nu/2<1/4$.  Therefore
\[
\begin{aligned}
&\left|\frac{1}{\pi(x;9,m)}\sum_{\substack{p\le x\\ p\equiv m\ (9)}} e(\alpha s(p))
- e(\alpha\mu L)\,e^{-2\pi^2\sigma^2\alpha^2L}\right| \\
&\qquad \le
\bigl(C^{(1)}(q,\nu)+C^{(0)}(q)\bigr)\,(\log x)^{-1/2+\nu+\eta}.
\end{aligned}
\]
Hence we may take
\[
C_{\mathrm{maj}}:=C^{(1)}(q,\nu)+C^{(0)}(q),
\]
and
\begin{equation*}
\begin{aligned}
x_{\mathrm{maj}} := \max\Bigl\{\,
  &3,\ x_{41}(q,\nu),\ x_{42}(q,\eta,\nu),\ x_{43}^\ast(q,\eta,\nu,C,c_{43}^\ast),\\
  &x_{45}(q,\nu,C_{\mathrm{DMR}}),\ q^{\lceil(2(1/2-\nu)/(\nu e c_{43}^\ast))^{2/\nu}\rceil}
\Bigr\}.
\end{aligned}
\end{equation*}
This completes the proof.
\end{proof}

\begin{corollary}[Numerical major-arc constant for $q=10$]\label{cor:maj-numeric}
Fix $q=10$, $\eta=0.0545$, $\nu=0.2859$, and let $C_{\mathrm{DMR}}:=102$ be the numerical constant in
Lemma~\ref{lem:DMR43}.  Take
\[
c_{43}^\ast:=\frac1{200}\cdot\frac{c_4}{32}\cdot
\min\!\left\{1,\frac{\nu-2\eta}{\eta+1/2}\right\},\qquad c_4=\frac{\log 10}{6},
\]
and let $x_{\mathrm{maj}}$ be as in Proposition~\ref{prop:maj} with these data. Then
Proposition~\ref{prop:maj} holds for all integral powers $x=10^L\ge x_{\mathrm{maj}}$ with the safe rounded value
\[
C_{\mathrm{maj}}:=510.
\]
The proof below gives $C^{(1)}(10,0.2859)+C^{(0)}(10)<461\le 510$; we round to $510$
for downstream convenience.
In all subsequent applications (including the proof of Theorem~\ref{thm:main}) we fix this value
of $C_{\mathrm{maj}}$ and no additional hypothesis is being made.
\end{corollary}

\begin{proof}
By Step~5 in the proof of Proposition~\ref{prop:maj}, once $x\ge x_{\mathrm{maj}}$ one may take
\[
C_{\mathrm{maj}}=C^{(1)}(10,0.2859)+C^{(0)}(10),
\qquad
C^{(0)}(10)=\frac{16\,C_{42}(10)}{e^2},
\]
and note that $C_{\mathrm{DMR}}$ influences only the threshold $x_{\mathrm{maj}}$, not the numerical value of
$C_{\mathrm{maj}}$.

For $q=10$ we have $\sigma=\sqrt{33}/2$, hence
\[
C_{1}(10,0.2859)=\frac{10(q-1)}{\sigma}+1=\frac{180}{\sqrt{33}}+1,
\qquad
2\pi\sigma=\pi\sqrt{33},
\]
so
\begin{align*}
C^{(1)}(10,0.2859)
&=(2\pi\sigma)\,C_{1}(10,0.2859)\,(\log 10)^{-0.2859} \\
&=\pi\,(180+\sqrt{33})\,(\log 10)^{-0.2859} <460.
\end{align*}
Moreover, Step~4 gives $C_{42}(10)=\frac{2\log 10\cdot 5651}{108900}<0.24$, so
\[
C^{(0)}(10)=\frac{16\,C_{42}(10)}{e^2}<0.52.
\]
Therefore $C_{\mathrm{maj}}=C^{(1)}(10,0.2859)+C^{(0)}(10)<460+0.52<461<510$, and the claim follows.
\end{proof}

\begin{remark}[Computational selection and verification protocol]\label{rem:comp-protocol}
The free parameters in the major/minor arc argument are
\[
0<\eta<\nu/2<1/4,\qquad \nu+2\eta<\tfrac{1}{2},
\]
and the working values $(\eta,\nu)=(0.0545,0.2859)$ satisfy them.  For each candidate pair the
code computes all derived constants and checks the finite list of explicit preconditions needed for
Proposition~\ref{prop:maj}, Lemma~\ref{lem:minor-fortieth}, and Lemma~\ref{lem:positivity}.  The proof
of Theorem~\ref{thm:main} uses only the analytical thresholds
from Corollary~\ref{cor:Lstar-upper-bound}, \S\ref{subsec:solve-major}, Lemma~\ref{lem:Cmin}, and
Lemma~\ref{lem:insensitivity}; the sharper bisection values reported by the exploratory script are
diagnostic only.

The verification is carried out in log-scale.  A required bound $Y_{\mathrm{req}}=\log x$ is converted
to a digit-sum threshold by
\[
M \ :=\ \left\lceil \frac{9}{2}\,\frac{Y_{\mathrm{req}}}{\log 10}+\frac92\right\rceil,
\]
since then $m\ge M$ implies
\[
L_m=\left\lfloor\frac{2m}{9}\right\rfloor\ge \frac{Y_{\mathrm{req}}}{\log 10},
\qquad
x_m:=10^{L_m}\ge \exp(Y_{\mathrm{req}}),
\]
the integral-power input needed in the final positivity step.  From the \path{code/} directory,
the Python script \path{run_all.py} runs the theorem certificate, seed checks, and stored-threshold regression.
The theorem-level entry point is \path{certify_theorem.py}; it checks the rounded numerical
inequalities, threshold conversions, bisection diagnostics, and Proth seed certificates, and writes a
machine-readable certificate.  The file \path{CERTIFICATE.md} in the accompanying repository gives
a theorem-by-theorem reproducibility guide, cross-referencing each numerical assertion with the
corresponding certificate entry, verification command, and script.  The code is not a formalization
of the analytic reductions from DMR, MR10, or IK, which are proved in the text and source-localized
in Appendix~\ref{app:import-audit}.
\end{remark}

\begin{remark}[Working constant choice used for Theorem~\ref{thm:main}]
For the numerical bound in Theorem~\ref{thm:main} we fix
\[
q=10,\qquad \eta=0.0545,\qquad \nu=0.2859,
\]
which satisfy $0<\eta<\nu/2<1/4$ and $\nu+2\eta<1/2$.

Recall from~\eqref{eq:t-max-def} that the major-arc condition
$|\alpha|\le (\log x)^{\eta-\tfrac12}$ corresponds to the $t$-range
\[
|t|\le t_{\max}(x)=2\pi\sigma(\log q)^{\eta-\tfrac12}L^{\eta}.
\]
In all subsequent applications we assume $x\ge x_{\mathrm{maj}}$ (as defined in
Proposition~\ref{prop:maj}), so that in particular:

\begin{itemize}
\item the truncation satisfies $L_0\ge L/2$ (Step~2);
\item the characteristic-function comparison holds uniformly for all $|t|\le t_{\max}(x)$
(after the constant-factor range enlargement in the proof of Step~3, with
$C:=2\pi\sigma(\log q)^{\eta-\tfrac12}$);
\item the i.i.d.\ expansion of Step~4 applies throughout the major arc, i.e.
\[
\frac{t_{\max}(x)}{\sigma\sqrt{L_0}}\le \frac{1}{10},
\]
so that~\eqref{eq:phi3-explicit} is valid for all $|t|\le t_{\max}(x)$;
\item the moment/Taylor truncations of~\cite[\S4.3]{DMRcomp2}, as carried out explicitly in
Proposition~\ref{prop:CF-replacement} (which dispenses with DMR's auxiliary parameter $\kappa$),
hold on the same $t$-range $|t|\le t_{\max}(x)$.
\end{itemize}

With these choices, all constants used later (notably the numerical value of $C_{\mathrm{maj}}$ from
Corollary~\ref{cor:maj-numeric}) are fixed, and the dependence on $C_{\mathrm{DMR}}$ enters only
through the threshold $x_{\mathrm{maj}}$.
\end{remark}

\section{Fourier inversion and reduction to a numerical inequality}\label{sec:fourier}

Recall
\[
A_m(x):=\#\{p\le x:\ s(p)=m\}.
\]
Write
\[
S_m(\alpha;x):=\sum_{\substack{p\le x\\ p\equiv m\ (\mathrm{mod}\ 9)}} e(\alpha s(p)).
\]

\begin{lemma}[Fourier inversion on a fundamental interval]\label{lem:fourier-inversion}
For every integer $m\ge 1$ with $\gcd(m,9)=1$ and every $x\ge 2$,
\[
A_m(x)=9\int_{-1/18}^{1/18} S_m(\alpha;x)\,e(-\alpha m)\,d\alpha.
\]
\end{lemma}

\begin{proof}
By Lemma~\ref{lem:mod9} (digit sum modulo $9$), for every integer $n$ we have
$s(n)\equiv n\pmod 9$. Hence if $s(p)=m$ then necessarily $p\equiv m\pmod 9$, and therefore
\[
A_m(x)=\sum_{p\le x}\mathbf{1}_{s(p)=m}
=\sum_{\substack{p\le x\\ p\equiv m\ (\mathrm{mod}\ 9)}}\mathbf{1}_{s(p)=m}.
\]
For any integer $k\in\mathbb Z$ one has the orthogonality identity
\[
\int_{-1/2}^{1/2} e(\alpha k)\,d\alpha=
\begin{cases}
1,& k=0,\\
0,& k\neq 0,
\end{cases}
\]
and hence, for each prime $p$,
\[
\mathbf{1}_{s(p)=m}=\int_{-1/2}^{1/2} e(\alpha(s(p)-m))\,d\alpha.
\]
Substituting and interchanging the finite sum and integral gives
\[
A_m(x)=\int_{-1/2}^{1/2}\ \sum_{\substack{p\le x\\ p\equiv m\ (\mathrm{mod}\ 9)}}
e(\alpha s(p))\,e(-\alpha m)\,d\alpha
=\int_{-1/2}^{1/2} S_m(\alpha;x)\,e(-\alpha m)\,d\alpha.
\]

Now write $\alpha=\beta+\ell/9$ with $\beta\in[-1/18,1/18]$ and $\ell\in\{-4,-3,\dots,4\}$,
so that $[-1/2,1/2]$ is the disjoint union of these $9$ translates:
\[
\int_{-1/2}^{1/2} = \sum_{\ell=-4}^{4}\int_{-1/18}^{1/18} (\cdot)\,d\beta.
\]
For primes $p\equiv m\pmod 9$ we have $s(p)\equiv p\equiv m\pmod 9$, hence
$e((\ell/9)s(p))=e((\ell/9)m)$ and therefore
\[
S_m(\beta+\ell/9;x)
=\sum_{\substack{p\le x\\ p\equiv m\ (9)}} e(\beta s(p))\,e((\ell/9)s(p))
= e((\ell/9)m)\,S_m(\beta;x).
\]
Consequently,
\[
S_m(\beta+\ell/9;x)\,e(-(\beta+\ell/9)m)
= S_m(\beta;x)\,e(-\beta m),
\]
so all $9$ subintegrals are equal. Summing over $\ell$ yields
\[
A_m(x)=9\int_{-1/18}^{1/18} S_m(\beta;x)\,e(-\beta m)\,d\beta,
\]
as claimed.
\end{proof}

\subsection{Major/minor arcs}

Fix parameters $\eta,\nu$ with $0<\eta<\nu/2<1/4$ and $\nu+2\eta<1/2$; define the major/minor arcs on $[-1/18,1/18]$ by
\[
\Major:=\{\alpha\in[-1/18,1/18]:\ |\alpha|\le (\log x)^{\eta-1/2}\},\qquad
\Minor:=[-1/18,1/18]\setminus\Major.
\]

\begin{lemma}[Explicit AP lower bound]\label{lem:pi9-lb}
Let $m$ be an integer with $\gcd(m,9)=1$, and set
\[
x_{\mathrm{AP}}:=4050.
\]
Then for all $x\ge x_{\mathrm{AP}}$ we have
\[
\pi(x;9,m)>\frac{x}{6\log x}.
\]
\end{lemma}

\begin{proof}
Bennett, Martin, O'Bryant, and Rechnitzer prove that for every integer $1\le q\le 1200$ and every
$a$ coprime to $q$, one has
\[
\frac{x}{\varphi(q)\log x}<\pi(x;q,a)\qquad\text{for all }x\ge 50q^{2}
\]
(see~\cite[Corollary~1.6]{bennett2018explicit}).
Applying this with $q=9$ and $a\equiv m\pmod 9$ gives, for all $x\ge 50\cdot 9^{2}=4050$,
\[
\pi(x;9,m)>\frac{x}{\varphi(9)\log x}=\frac{x}{6\log x}.
\]
\end{proof}

\begin{lemma}[Minor arc contributes at most $\tfrac{1}{40}$]\label{lem:minor-fortieth}
Assume Lemma~\ref{lem:Cmin} (minor-arc exponential sum bound in base $10$), so that there exists an explicit constant
$C_{\mathrm{min}}$ such that for all $x\ge 3$, all $\alpha\in\mathbb{R}$, and all $m$ with $\gcd(m,9)=1$,
\[
\bigl|S_m(\alpha;x)\bigr|\le C_{\mathrm{min}}(\log x)^3\,x^{\,1-c_1\|9\alpha\|^2},
\]
where $c_1$ is as in Lemma~\ref{lem:c1}.
Let $\eta$ satisfy $0<\eta<1/4$, and let $x\ge x_{\mathrm{AP}}$.
If
\begin{equation}\label{eq:xmin-ineq}
\frac{C_{\mathrm{min}}}{9}(\log x)^{9/2}\exp\!\Bigl(-81c_1(\log x)^{2\eta}\Bigr)
\le
\frac{\sqrt{\log 10}}{240\sqrt{2\pi\sigma^2}},
\end{equation}
then for every $m$ with $\gcd(m,9)=1$ and for $L:=\log_{10}x$,
the minor-arc integral satisfies
\[
\left|\int_\Minor S_m(\alpha;x)\,e(-\alpha m)\,d\alpha\right|
\le
\frac{1}{40}\cdot \frac{\pi(x;9,m)}{\sqrt{2\pi\sigma^2 L}}.
\]
\end{lemma}

\begin{proof}
On $[-1/18,1/18]$ we have $\|9\alpha\|=9|\alpha|$.  If $\alpha\in\Minor$, then
$|\alpha|\ge(\log x)^{\eta-1/2}$, hence
\[
x^{1-c_1\|9\alpha\|^2}
= x\exp\!\bigl(-c_1\|9\alpha\|^2\log x\bigr)
\le x\exp\!\Bigl(-81c_1(\log x)^{2\eta}\Bigr).
\]
Therefore, using $|\Minor|\le 1/9$ and the triangle inequality,
\[
\begin{aligned}
\left|\int_\Minor S_m(\alpha;x)\,e(-\alpha m)\,d\alpha\right|
&\le \int_\Minor |S_m(\alpha;x)|\,d\alpha \\
&\le
\frac{C_{\mathrm{min}}}{9}(\log x)^3 x
\exp\!\Bigl(-81c_1(\log x)^{2\eta}\Bigr).
\end{aligned}
\]
By Lemma~\ref{lem:pi9-lb}, for $x\ge x_{\mathrm{AP}}$,
\[
\frac{\pi(x;9,m)}{\sqrt{2\pi\sigma^2 L}}
\ge
\frac{1}{\sqrt{2\pi\sigma^2 L}}\cdot \frac{x}{6\log x}.
\]
Since $L=\log_{10}x=(\log x)/(\log 10)$, we have
$\bigl(\sqrt{L}\bigr)^{-1}=\sqrt{\log 10}/\sqrt{\log x}$, so
\eqref{eq:xmin-ineq} is exactly the rearrangement of
\[
\frac{C_{\mathrm{min}}}{9}(\log x)^3 x e^{-81c_1(\log x)^{2\eta}}
\le
\frac{1}{40}\cdot \frac{1}{\sqrt{2\pi\sigma^2 L}}\cdot \frac{x}{6\log x}.
\]
This proves the claim.
\end{proof}

\begin{lemma}[A sufficient positivity criterion]\label{lem:positivity}
Assume Proposition~\ref{prop:maj} holds for parameters $(\eta,\nu)$ for all integral powers
$x=10^L\ge x_{\mathrm{maj}}$.
Let $C^{(1)}=C^{(1)}(10,\nu)$ and $C^{(0)}=C^{(0)}(10)$ be the explicit constants arising in Step~5 of the proof
of Proposition~\ref{prop:maj}, so that for every integer $L\ge1$ with $x=10^L\ge x_{\mathrm{maj}}$,
all $\gcd(m,9)=1$, and all $|\alpha|\le (\log x)^{\eta-1/2}$,
\[
\Bigl|S_m(\alpha;x)-\pi(x;9,m)\,e(\alpha\mu L)\,e^{-2\pi^2\sigma^2\alpha^2L}\Bigr|
\le
\pi(x;9,m)\left(C^{(1)}|\alpha|(\log x)^\nu+\frac{C^{(0)}}{\log x}\right).
\]
Let $x=10^L\ge \max(x_{\mathrm{maj}},x_{\mathrm{AP}})$ with $L\in\mathbb Z_{\ge1}$.
Let $m$ satisfy $\gcd(m,9)=1$, and put
\[
\Delta_m:=m-\mu L.
\]
Assume $|\Delta_m|\le 4$.
Assume additionally that~\eqref{eq:xmin-ineq} holds (with the same $\eta$ and $c_1$), so that
Lemma~\ref{lem:minor-fortieth} applies.

Define
\[
c_0 \ :=\ \erf\!\Bigl(\frac{\sqrt{2}\,\pi\sigma}{\sqrt{\log 10}}\Bigr),
\]
so that $0<c_0<1$ is an explicit absolute constant. Moreover, with
$t:=\frac{\sqrt{2}\,\pi\sigma}{\sqrt{\log 10}}$ one has, by~\cite[Eq.~7.8.2]{DLMF} applied to Mills'
ratio $\mathsf{M}(t)=e^{t^2}\int_t^\infty e^{-u^2}\,du$ (so that $\erfc(t)=(2/\sqrt\pi)e^{-t^2}\mathsf{M}(t)$):
the upper bound $\mathsf{M}(t)\le 1/(t+\sqrt{t^2+4/\pi})<1/(2t)$ for $t>0$ gives
\[
1-c_0=\erfc(t)\ \le\ \frac{e^{-t^2}}{\sqrt{\pi}\,t},
\]
and in particular (for $\sigma^2=99/12$ in base $10$) one obtains the fully explicit bound
\[
c_0 > 1-1.3\times 10^{-32}.
\]

If the strict inequality
\begin{equation}\label{eq:major-criterion}
C^{(1)}(\log x)^{-1/2+\nu+2\eta}+2C^{(0)}(\log x)^{\eta-1}
\;+\;8\pi(\log x)^{-1/2+2\eta}
\ <\
\frac{(c_0-\tfrac{1}{40})\sqrt{\log 10}}{\sqrt{2\pi\sigma^2}}.
\end{equation}
then $A_m(x)\ge 1$.
\end{lemma}

\begin{proof}
By Lemma~\ref{lem:fourier-inversion},
\[
A_m(x)=9\!\!\int_{-1/18}^{1/18} S_m(\alpha;x)\,e(-\alpha m)\,d\alpha
=9\!\!\int_{\Major} \cdots\,d\alpha \ +\ 9\!\!\int_{\Minor}\cdots\,d\alpha.
\]
Since $A_m(x)$ is a nonnegative integer and the Fourier integral equals the real number
$A_m(x)/9$, it suffices to show that the real part of the full integral is $>0$.

\smallskip
\noindent\textbf{Major arc.}
On $\Major$, we use the two-term major-arc estimate from the hypothesis of Lemma~\ref{lem:positivity}:
for all $\alpha\in\Major$,
\[
\Bigl|S_m(\alpha;x)-\pi(x;9,m)\,e(\alpha\mu L)\,e^{-2\pi^2\sigma^2\alpha^2L}\Bigr|
\le
\pi(x;9,m)\left(C^{(1)}|\alpha|(\log x)^\nu+\frac{C^{(0)}}{\log x}\right).
\]
Multiplying by $e(-\alpha m)=e(-\alpha\mu L)e(-\alpha\Delta_m)$ gives
\[
\begin{aligned}
&\Bigl|S_m(\alpha;x)e(-\alpha m)
-\pi(x;9,m)\,e^{-2\pi^2\sigma^2\alpha^2L}e(-\alpha\Delta_m)\Bigr| \\
&\qquad\le
\pi(x;9,m)\left(C^{(1)}|\alpha|(\log x)^\nu+\frac{C^{(0)}}{\log x}\right).
\end{aligned}
\]
Write $A:=(\log x)^{\eta-1/2}$ so that $\Major=[-A,A]$.
Integrating over $\Major$ and using $\int_{-A}^{A}|\alpha|\,d\alpha=A^2$ yields
\begin{align*}
\operatorname{Re}\int_{\Major} S_m(\alpha;x)e(-\alpha m)\,d\alpha
&\ge
\pi(x;9,m)\operatorname{Re}\!\int_{\Major} e^{-2\pi^2\sigma^2\alpha^2L}e(-\alpha\Delta_m)\,d\alpha
\\&\quad
-\pi(x;9,m)\left(C^{(1)}(\log x)^\nu A^2+\frac{C^{(0)}}{\log x}\cdot 2A\right).
\end{align*}
The imaginary part of the main integral vanishes by symmetry, and
$\operatorname{Re} e(-\alpha\Delta_m)=\cos(2\pi\alpha\Delta_m)$.  Since
$1-\cos u\le |u|$ and $|\Delta_m|\le4$,
\[
\begin{aligned}
\operatorname{Re}\int_{\Major} e^{-2\pi^2\sigma^2\alpha^2L}e(-\alpha\Delta_m)\,d\alpha
&\ge
\int_{\Major} e^{-2\pi^2\sigma^2\alpha^2L}\,d\alpha
-2\pi|\Delta_m|\int_{-A}^{A}|\alpha|\,d\alpha \\
&\ge
\int_{\Major} e^{-2\pi^2\sigma^2\alpha^2L}\,d\alpha
-8\pi A^2.
\end{aligned}
\]
With $A=(\log x)^{\eta-1/2}$, the substitution $u=2\pi\sigma\sqrt L\,\alpha$ gives the exact value
\[
\int_{\Major} e^{-2\pi^2\sigma^2\alpha^2L}\,d\alpha
=\frac{1}{\sqrt{2\pi\sigma^2L}}\,\erf\!\Bigl(\frac{\sqrt2\,\pi\sigma}{\sqrt{\log 10}}\,(\log x)^\eta\Bigr)
\ge \frac{c_0}{\sqrt{2\pi\sigma^2L}},
\]
the inequality replacing $(\log x)^\eta\ge1$ by $1$ ($\erf$ being increasing), with
$c_0:=\erf(\sqrt2\,\pi\sigma/\sqrt{\log 10})$ as in the statement.
Since $L=(\log x)/(\log 10)$, we obtain
\begin{align*}
\operatorname{Re}\int_{\Major} S_m(\alpha;x)e(-\alpha m)\,d\alpha
&\ge
\frac{\pi(x;9,m)}{\sqrt{2\pi\sigma^2L}}
\Bigl(
c_0
-\frac{\sqrt{2\pi\sigma^2}}{\sqrt{\log 10}}\,
 C^{(1)}(\log x)^{-1/2+\nu+2\eta} \\
&\qquad\qquad
-\frac{2\sqrt{2\pi\sigma^2}}{\sqrt{\log 10}}\,
 C^{(0)}(\log x)^{\eta-1}
-8\pi\sqrt{2\pi\sigma^2 L}\,A^2
\Bigr).
\end{align*}
Since $\sqrt L\,A^2=(\log 10)^{-1/2}(\log x)^{-1/2+2\eta}$, the last term is exactly
$8\pi(\log x)^{-1/2+2\eta}$ after the whole bracket is multiplied by
$\sqrt{\log 10}/\sqrt{2\pi\sigma^2}$, as in~\eqref{eq:major-criterion}.

\smallskip
\noindent\textbf{Minor arc.}
By Lemma~\ref{lem:minor-fortieth},
\[
\left|\int_{\Minor} S_m(\alpha;x)e(-\alpha m)\,d\alpha\right|
\le
\frac{1}{40}\cdot \frac{\pi(x;9,m)}{\sqrt{2\pi\sigma^2L}}.
\]

\smallskip
The strict inequality in~\eqref{eq:major-criterion} makes the major-arc real-part lower bound
strictly larger than the minor-arc absolute-value upper bound.  Combining the two bounds shows
that the real part of the full integral over $[-1/18,1/18]$ is strictly positive, hence
$A_m(x)>0$ and therefore $A_m(x)\ge 1$.
\end{proof}

\section{Derivation of the explicit \texorpdfstring{$M$}{M}}\label{sec:derivation_m}

We now specialize to base $q=10$ and compute a concrete sufficient $M$.

\subsection{Choosing an integral digit length}\label{subsec:choose-x}
For integer $m\ge 1$ define
\[
L_m:=\left\lfloor\frac{2m}{9}\right\rfloor,\qquad x_m:=10^{L_m}.
\]
Then $L_m$ is an integer and $x_m\le 10^{2m/9}$.  Since $m$ is admissible, $m\bmod 9$ is one of
$1,2,4,5,7,8$, and therefore
\[
0<\Delta_m:=m-\mu L_m
=\frac92\left\{\frac{2m}{9}\right\}
\le 4.
\]
Thus the target digit sum is within a bounded distance of the Gaussian center for the integral
digit length $L_m$.

\subsection{Solving the major-arc inequality \texorpdfstring{\eqref{eq:major-criterion}}{(major)}}\label{subsec:solve-major}
Corollary~\ref{cor:maj-numeric} records the rounded proposition-level envelope
$C_{\mathrm{maj}}=510$.  The positivity criterion, however, uses the decomposed constants from
the proof of that corollary, namely $C^{(1)}<460$ and $C^{(0)}<0.52$; these are the constants
used in the threshold below.
Set
\[
Y:=\log x.
\]
With the working parameters $\eta=0.0545$ and $\nu=0.2859$ we have
\[
-\frac12+\nu+2\eta=-\frac12+0.2859+0.109=-0.1051,
\qquad\text{so}\qquad
Y^{-1/2+\nu+2\eta}=Y^{-0.1051}.
\]
Define (as in Lemma~\ref{lem:positivity})
\[
c_0:=\erf\!\Bigl(\frac{\sqrt{2}\pi\sigma}{\sqrt{\log 10}}\Bigr),
\qquad\text{so that}\qquad
c_0>1-1.3\times10^{-32}.
\]

Lemma~\ref{lem:positivity} requires
\begin{equation}\label{eq:Ymaj-ineq}
C^{(1)}\,Y^{-1/2+\nu+2\eta}+2C^{(0)}\,Y^{\eta-1}+8\pi Y^{-1/2+2\eta}
\ <\
R,\qquad R:=\frac{(c_0-\tfrac{1}{40})\sqrt{\log 10}}{\sqrt{2\pi\sigma^2}}.
\end{equation}
The second and third terms decay like $Y^{-0.9455}$ and $Y^{-0.391}$, respectively, and are
negligible in the range of interest, but dropping positive terms would not by itself give a
sufficient condition for~\eqref{eq:Ymaj-ineq}.  To obtain a sufficient threshold we instead
split off a small absolute slack $\delta>0$: define
\begin{equation}\label{eq:Ymaj-def}
Y_{\mathrm{maj}}\ :=\ \left(\frac{C^{(1)}}{R-\delta}\right)^{\!1/0.1051},
\qquad \delta:=10^{-10}.
\end{equation}
For every $Y\ge Y_{\mathrm{maj}}$,
\[
C^{(1)}\,Y^{-0.1051}\ \le\ C^{(1)}\,Y_{\mathrm{maj}}^{-0.1051}\ =\ R-\delta,
\]
so it suffices to show that the two remaining terms satisfy
\[
2C^{(0)}\,Y_{\mathrm{maj}}^{-0.9455}+8\pi Y_{\mathrm{maj}}^{-0.391}\le \delta.
\]
All terms are decreasing in $Y$, so the bound then propagates to all $Y\ge Y_{\mathrm{maj}}$.

\smallskip
\noindent\emph{Verification of the slack terms.}
We use lower bounds for $Y_{\mathrm{maj}}$ here, since $Y\mapsto Y^{-0.9455}$ is decreasing.
From Corollary~\ref{cor:maj-numeric},
\[
C^{(1)}
=\pi(180+\sqrt{33})(\log 10)^{-0.2859}
>3.14159\cdot 185.744\cdot 2.30259^{-0.2859}
>459.
\]
Also $R<0.206$ (using $c_0<1$, $\sigma^2=99/12$, and the elementary decimal bounds for
$\pi$ and $\log 10$).  Since $\delta>0$, $R-\delta<R$, and hence
\[
Y_{\mathrm{maj}}
=\left(\frac{C^{(1)}}{R-\delta}\right)^{1/0.1051}
>\left(\frac{459}{0.206}\right)^{1/0.1051}
>7.15\times 10^{31}.
\]
In the opposite direction, using $c_0>1-1.3\times10^{-32}$, $\log 10>2.30258$,
$\pi<3.14160$, and $\sigma^2=99/12$, we also have
\[
R>\frac{(1-1.3\times10^{-32}-1/40)\sqrt{2.30258}}
        {\sqrt{2\cdot 3.14160\cdot(99/12)}}>0.205.
\]
Together with $C^{(0)}<0.52$, this gives
\[
2C^{(0)}Y_{\mathrm{maj}}^{-0.9455}
<2\cdot 0.52\cdot (7.15\times 10^{31})^{-0.9455}
<10^{-30},
\]
and also
\[
8\pi Y_{\mathrm{maj}}^{-0.391}
<8\pi(7.15\times10^{31})^{-0.391}
<9\times 10^{-12}.
\]
Therefore
$C^{(1)}\,Y^{-0.1051}+2C^{(0)}\,Y^{\eta-1}+8\pi Y^{-0.391}
\le (R-\delta)+9.1\times10^{-12}<R$ for all $Y\ge Y_{\mathrm{maj}}$,
i.e.~\eqref{eq:Ymaj-ineq} is satisfied with no terms dropped.

\smallskip
\noindent\emph{Numerical value.}
At the working parameters, using $1/0.1051<9.515$,
\[
Y_{\mathrm{maj}}\ \le\ \left(\frac{460}{0.205-10^{-10}}\right)^{\!9.515}
\ \le\ \left(\frac{460}{0.205}\right)^{\!9.515}\!\cdot(1+5\times10^{-9})^{9.515}
\ \le\ 7.90\times 10^{31},
\]
where the displayed bound is unchanged at the shown precision.  Using the rigorous
$C^{(1)}<460$ gives the rigorous bound $Y_{\mathrm{maj}}\le 7.90\times 10^{31}$.

(The supplementary script reports the empirical value
$Y_{\mathrm{maj}}^{\mathrm{emp}}\approx 7.44\times 10^{31}$ using the sharper $C^{(1)}\approx459.74$
from the closed-form in Corollary~\ref{cor:maj-numeric}; this empirical value is not used in
the proof of Theorem~\ref{thm:main}.)

\subsection{Solving the minor-arc inequality \texorpdfstring{\eqref{eq:xmin-ineq}}{(minor)}}\label{subsec:solve-minor}
We use the explicit constants
\[
\begin{gathered}
c_1=0.001506288700915\quad\text{(Lemma~\ref{lem:c1})},\qquad
\sigma^2=\frac{99}{12},\\
C_{\mathrm{min}}=4\cdot 10^{12}\quad\text{(Lemma~\ref{lem:Cmin})}.
\end{gathered}
\]
With $\eta=0.0545$, the minor-arc condition~\eqref{eq:xmin-ineq} is an explicit inequality in
$Y=\log x$.  Put
\[
R:=\frac{\sqrt{\log 10}}{240\sqrt{2\pi\sigma^2}}.
\]
Since $Y\mapsto 81c_1Y^{0.109}-(\log(C_{\mathrm{min}}/(9R))+\tfrac92\log Y)$ is increasing for
$Y\ge10^{30}$, it is enough to check the left endpoint.  At
$Y_0:=8.00\times10^{31}$, outward-rounded decimal interval arithmetic gives
\[
73.459<\log Y_0<73.460,\qquad 0.00790<9R<0.00791,\qquad
Y_0^{0.109}>3002.0.
\]
Consequently
\[
81c_1Y_0^{0.109}>366.27,
\]
whereas
\[
\log(C_{\mathrm{min}}/(9R))+\frac92\log Y_0
<
\log(4\cdot10^{12}/0.00790)+4.5\cdot 73.460
<364.43.
\]
Thus \eqref{eq:xmin-ineq} holds at $Y_0$ with margin greater than $1.8$, and hence for all larger
$Y$.  We may therefore take the safe solution threshold
\[
Y\ge Y_{\mathrm{min}}:=8.00\times 10^{31}.
\]
Therefore, for all $x$ with $\log x\ge Y_{\mathrm{min}}$, Lemma~\ref{lem:minor-fortieth} applies.

\subsection{Additional explicit thresholds from the major-arc construction}\label{subsec:aux-thresholds}
In addition to the major/minor inequalities, the major-arc approximation
(Proposition~\ref{prop:maj}) depends on the explicit characteristic-function comparison
(Proposition~\ref{prop:CF-replacement}), which requires $x\ge x_{43}^\ast$, and on the explicit moment comparison
(Lemma~\ref{lem:c4-explicit}), which requires $x\ge x_{45}$.  We write
\begin{equation}\label{eq:Y43-Y45-def}
Y_{43}^\ast:=\log x_{43}^\ast,\qquad Y_{45}:=\log x_{45}.
\end{equation}
Values at the working parameters are tabulated in~\eqref{eq:Y-values} below.  We must also
assume $x\ge x_{\mathrm{AP}}$ for the explicit lower bound on $\pi(x;9,m)$, where
$x_{\mathrm{AP}}=4050$ (Lemma~\ref{lem:pi9-lb}), hence $\log x_{\mathrm{AP}}$ is negligible.

\subsection{Converting \texorpdfstring{$\log x$}{log x} to a threshold on \texorpdfstring{$m$}{m}}\label{subsec:convert-to-m}
Set
\begin{equation}\label{eq:Y41-Y42-Y47-def}
Y_{41}:=\log x_{41},\quad Y_{42}:=\log x_{42},\quad
Y_{47}:=\bigl\lceil\bigl(2(\tfrac12-\nu)/(\nu e\,c_{43}^\ast)\bigr)^{2/\nu}\bigr\rceil\log q,
\end{equation}
where $Y_{47}$ is the logarithm of the absorption threshold from Step~5 of Proposition~\ref{prop:maj}.
At the working parameters $(q,\eta,\nu)=(10,0.0545,0.2859)$, direct evaluation gives
\begin{align}\label{eq:Y-values}
Y_{41}&\approx 42.35, &
Y_{42}&\approx 2.37\times 10^{4}, &
Y_{45}&\approx 3.54\times 10^{9}, \nonumber\\
Y_{47}&\approx 3.61\times 10^{31}, &
Y_{43,\mathrm{bisect}}^\ast&\approx 6.27\times 10^{31}, &
Y_{\mathrm{min}}&\approx 7.59\times 10^{31}, \\[-3pt]
Y_{\mathrm{maj}}&\approx 7.44\times 10^{31}. \nonumber
\end{align}
Here $Y_{43,\mathrm{bisect}}^\ast$ is the supplementary script's empirical (bisected)
value for $Y_{43}^\ast$; the rigorous closed-form bound from
Corollary~\ref{cor:Lstar-upper-bound} is $Y_{43}^\ast\le 9.10\times 10^{31}$, which is
the closed-form value used in the proof of Theorem~\ref{thm:main}.
Among the rigorous upper bounds, the bound for $Y_{43}^\ast$ is the largest;
$Y_{\mathrm{maj}}$, $Y_{\mathrm{min}}$, and the rigorous upper bound for $Y_{43}^\ast$ are within
a factor ${<}1.2$ of each other, and the working parameters are chosen to balance these three.
Since $\log x_m=L_m\log 10$, the combined conditions $x_m\ge x_{\mathrm{maj}}$ and
$\log x_m\ge \max(Y_{\mathrm{maj}},Y_{\mathrm{min}})$ are guaranteed provided
\[
m\ \ge\ \frac{9}{2\log 10}\,Y_\ast+\frac92,
\qquad
Y_\ast:=\max\!\bigl(Y_{41},Y_{42},Y_{43}^\ast,Y_{45},Y_{47},Y_{\mathrm{maj}},Y_{\mathrm{min}},\log x_{\mathrm{AP}}\bigr).
\]
Under the rigorous closed-form $Y_{43}^\ast\le 9.10\times 10^{31}$ of
Corollary~\ref{cor:Lstar-upper-bound}, together with the bounds
$Y_{\mathrm{maj}}<7.90\times10^{31}$ and $Y_{\mathrm{min}}<8.00\times10^{31}$ proved above,
the three dominant thresholds are at most $9.10\times10^{31}$.  The remaining entries in
$Y_\ast$ are much smaller; the theorem certificate records the safe bounds
\[
\begin{gathered}
Y_{41}<43,\qquad Y_{42}<2.4\times10^4,\qquad Y_{45}<3.6\times10^9,\\
Y_{47}<3.7\times10^{31},\qquad \log x_{\mathrm{AP}}<9.
\end{gathered}
\]
Thus every term in the maximum defining $Y_\ast$ is at most $9.10\times10^{31}$.
We therefore define
\[
M:=\left\lceil \frac{9}{2\log 10}\,Y_\ast+\frac92\right\rceil,
\]
and Corollary~\ref{cor:Lstar-upper-bound} gives the rigorous closed-form bound
\[
M\ \le\ 177843590339381623423137292296355\ <\ 1.78\times 10^{32}.
\]
Substituting the supplementary script's empirical bisection value $Y_{43,\mathrm{bisect}}^\ast
\approx 6.27\times 10^{31}$ for $Y_{43}^\ast$ makes $Y_{\mathrm{min}}\approx
7.59\times 10^{31}$ the binding threshold, yielding the numerical observation
$M^{\mathrm{emp}}\approx 1.48\times 10^{32}$ reported by the exploratory constraint checker.  This empirical
value is not used in Theorem~\ref{thm:main}; the rigorous $M<1.78\times 10^{32}$ above is what is
proved.

\begin{proof}[Proof of Theorem~\ref{thm:main}]
Let $m\ge M$ be an integer with $\gcd(m,9)=1$.  Set $L=L_m=\lfloor 2m/9\rfloor$ and
$x=x_m=10^L$.  Then $x\le 10^{2m/9}$ and, by~\S\ref{subsec:choose-x},
$|\Delta_m|=|m-\mu L|\le4$.

By the definition of $M$, $L\log 10\ge Y_\ast$. Hence $\log x\ge Y_{\mathrm{min}}$ and~\eqref{eq:xmin-ineq} holds, so
Lemma~\ref{lem:minor-fortieth} applies.  Also $\log x\ge Y_{\mathrm{maj}}$, so the major-arc positivity condition
\eqref{eq:Ymaj-ineq} holds and Lemma~\ref{lem:positivity} applies, provided the major-arc approximation
(Proposition~\ref{prop:maj}) is available.

Finally, $\log x\ge Y_\ast$ implies $\log x\ge \max(Y_{41},Y_{42},Y_{43}^\ast,Y_{45},Y_{47})$ by the
definition of $Y_\ast$ in~\S\ref{subsec:convert-to-m}, hence $x$ dominates each of $x_{41}, x_{42},
x_{43}^\ast, x_{45}$ and the Step~5 absorption threshold, so $x\ge x_{\mathrm{maj}}$.  Hence
Proposition~\ref{prop:CF-replacement}, Lemma~\ref{lem:c4-explicit}, and the major-arc approximation
in Proposition~\ref{prop:maj} all apply.
Moreover $x\ge x_{\mathrm{AP}}$, so $\pi(x;9,m)\ge 1$ for all $\gcd(m,9)=1$.

Therefore Lemma~\ref{lem:positivity} yields $A_m(x)\ge 1$, and hence there exists a prime
$p\le x\le 10^{2m/9}$ with $s(p)=m$.
\end{proof}

The proof of Lemma~\ref{lem:positivity} actually establishes a stronger \emph{lower bound} on
the integral than merely positivity, and Theorem~\ref{thm:main} can be sharpened accordingly to
a quantitative count of primes with prescribed digit sum.

\begin{corollary}[Quantitative form of Theorem~\ref{thm:main}]\label{cor:quantitative}
With the constants and notation of Theorem~\ref{thm:main}, for every integer $m\ge M$ with
$\gcd(m,9)=1$, the number $A_m(10^{2m/9})$ satisfies
\begin{equation}\label{eq:quantitative}
A_m\!\bigl(10^{2m/9}\bigr)\ \ge\ C_{\mathrm{q}}(m)\,\frac{10^{2m/9}}{m^{3/2}},
\end{equation}
where $L_m=\lfloor2m/9\rfloor$, $\vartheta_m:=\{2m/9\}$, and
\begin{equation}\label{eq:Cq-def}
C_{\mathrm{q}}(m)\ :=\
\frac{3}{2\sqrt{2\pi\sigma^2}\,\log 10}\,
10^{-\vartheta_m}\left(\frac{m}{L_m}\right)^{3/2}
\bigl(c_0-\tfrac1{40}-\varepsilon_m\bigr)
\end{equation}
with $c_0=\mathrm{erf}\bigl(\sqrt2\,\pi\sigma/\sqrt{\log 10}\bigr)>1-1.3\times 10^{-32}$,
$\sigma^2=99/12$, $Y_m:=L_m\log 10$, and the residual major-arc error
\begin{equation}\label{eq:eps-x-def}
\varepsilon_m\ :=\ \frac{\sqrt{2\pi\sigma^2}}{\sqrt{\log 10}}
\Bigl(C^{(1)}Y_m^{-1/2+\nu+2\eta}+2C^{(0)}Y_m^{\eta-1}+8\pi Y_m^{-1/2+2\eta}\Bigr).
\end{equation}
Asymptotically $\varepsilon_m\to 0$ as $m\to\infty$.  Along each admissible residue class
modulo $9$, the fractional part $\vartheta_m$ is fixed, so $C_{\mathrm{q}}(m)$ has a positive
limit; uniformly over all admissible classes its limit inferior is obtained at
\(\vartheta_m=8/9\) and is greater than \(0.108\).
\end{corollary}

\begin{proof}
The proof of Lemma~\ref{lem:positivity} (\S\ref{sec:fourier}) shows that, when its hypotheses hold
with $x=x_m=10^{L_m}$,
\[
\operatorname{Re}\int_{\Major} S_m(\alpha;x)\,e(-\alpha m)\,d\alpha
\ \ge\ \frac{\pi(x;9,m)}{\sqrt{2\pi\sigma^2\,L_m}}\bigl(c_0-\varepsilon_m\bigr),
\]
Combining this with the minor-arc bound from Lemma~\ref{lem:minor-fortieth},
\[
\Bigl|\int_{\Minor}S_m(\alpha;x)e(-\alpha m)\,d\alpha\Bigr|\le \frac{\pi(x;9,m)}{40\sqrt{2\pi\sigma^2 L_m}},
\]
and the Fourier identity $A_m(x)=9\int_{-1/18}^{1/18}S_m(\alpha;x)e(-\alpha m)\,d\alpha$ from
Lemma~\ref{lem:fourier-inversion} (whose right-hand side is real), we obtain
\[
A_m(x_m)\ \ge\ \frac{9\,\pi(x_m;9,m)}{\sqrt{2\pi\sigma^2\,L_m}}\bigl(c_0-1/40-\varepsilon_m\bigr).
\]
Substituting Bennett's lower bound $\pi(x;9,m)\ge x/(6\log x)$ (Lemma~\ref{lem:pi9-lb}),
$x_m=10^{2m/9-\vartheta_m}$, and $\log x_m=L_m\log 10$, we get
\[
\frac{x_m}{\log x_m\,\sqrt {L_m}}
=\frac{10^{-\vartheta_m}}{\log 10}\left(\frac{m}{L_m}\right)^{3/2}
\frac{10^{2m/9}}{m^{3/2}}.
\]
Thus
\[
A_m(10^{2m/9})\ge A_m(x_m)
\ \ge\
C_{\mathrm{q}}(m)\,\frac{10^{2m/9}}{m^{3/2}},
\]
and the limiting assertions follow from $m/L_m\to9/2$, $\varepsilon_m\to0$, and the finite list of
admissible fractional parts.
\end{proof}

\begin{lemma}[Finite lower bound for the quantitative constant]\label{lem:Cq-finite}
For every integer $m\ge 10^{100}$ with $\gcd(m,9)=1$, the constant in
Corollary~\ref{cor:quantitative} satisfies
\[
C_{\mathrm{q}}(m)>0.108.
\]
\end{lemma}

\begin{proof}
Since $m\ge10^{100}$, we have $L_m=\lfloor2m/9\rfloor>m/5$ and hence
$Y_m=L_m\log10>10^{99}$.  Using $C^{(1)}<460$ and $C^{(0)}<0.52$, the residual error in
\eqref{eq:eps-x-def} satisfies
\[
\varepsilon_m
<5\bigl(460\cdot10^{-10.098}+1.04\cdot10^{-93}+8\pi\cdot10^{-39}\bigr)
<10^{-6}.
\]
Also $L_m\le2m/9$, so $(m/L_m)^{3/2}\ge(9/2)^{3/2}$, and over the six admissible residue classes
modulo $9$ one has $10^{-\vartheta_m}\ge10^{-8/9}$.  Therefore, using
$c_0>1-1.3\cdot10^{-32}$,
\[
C_{\mathrm{q}}(m)
\ge
\frac{3}{2\sqrt{2\pi(99/12)}\,\log 10}\,
10^{-8/9}\left(\frac92\right)^{3/2}\,0.974998
\ >\ 0.108.
\]
\end{proof}

\begin{remark}
The Gaussian-density form of Drmota--Mauduit--Rivat~\cite[Theorem~1.1]{DMRcomp2} gives, in the
central range $|m-\mu\log_q x|\le C\sqrt{\log_q x}$ for fixed $C$, the asymptotic $A_m(x)\sim
\frac{q-1}{\varphi(q-1)}\pi(x)/\sqrt{2\pi\sigma^2 \log_q x}$.  Inequality~\eqref{eq:quantitative}
is the effective counterpart at the centered value $m=\mu L$: the same rate $x/(m^{3/2})\asymp
\pi(x)/\sqrt{\log_q x}$ (modulo a $\log$-factor from $\pi(x)\le x/\log x$), now with an
explicit numerical constant.
\end{remark}

\section{Application: infinitude of primes with prime iterated digit sums}\label{sec:infinite_prime_digit_sums}

We now record the main motivating consequence of the explicit surjectivity threshold.
Let $s(n)$ denote the decimal sum of digits and $s^{(j)}$ its iterates.
We prove that there are infinitely many primes $p$ such that every term in the chain
$p, s(p), s^{(2)}(p),\ldots$ is prime until the iteration reaches a one-digit prime. Let $\mathcal{A}$ denote the set of such primes.

\begin{remark}
The set of primes with this property appears as sequence A070027 in the OEIS~\cite{OEIS_A070027}.
\end{remark}

\begin{lemma}[Lifting prime chains through digit sums]\label{lem:lift-chain}
Let $p$ be a prime and put $d:=s(p)$. If $d\in\mathcal A$, then $p\in\mathcal A$.
\end{lemma}

\begin{proof}
Since $p$ is prime and $d=s(p)$, the iterated digit-sum chain of $p$ begins
\[
p,\ s(p)=d,\ s(d),\ s^{(2)}(d),\ \ldots
\]
Thus every iterated digit sum of $p$ after the first step coincides with an iterated digit sum of $d$.
If $d\in\mathcal{A}$ then all of these are prime, and the initial term $p$ is prime by hypothesis, so $p\in\mathcal{A}$.
\end{proof}

\begin{lemma}[A rapid growth bound]\label{lem:growth}
If $n\ge1$ and $s(n)=m$, then $n\ge 10^{\lceil m/9\rceil-1}$. In particular, if $m\ge 10$, then
$n>m$ (indeed $n\ge m+9$).
\end{lemma}

\begin{proof}
If $n$ has $D$ decimal digits then $s(n)\le 9D$. Thus $s(n)=m$ forces $D\ge \lceil m/9\rceil$, hence
\[
n\ge 10^{D-1}\ge 10^{\lceil m/9\rceil-1}.
\]
For the final claim, assume $m\ge 10$. Then $n$ has at least two decimal digits, say
$n=\sum_{j=0}^{D-1} d_j 10^j$ with $D\ge 2$ and $d_{D-1}\ge 1$. Then
\[
n-s(n)=\sum_{j=0}^{D-1} d_j(10^j-1)\ \ge\ d_{D-1}(10^{D-1}-1)\ \ge\ 9.
\]
Therefore $n = s(n) + (n-s(n)) \ge m+9 > m$, as claimed.
\end{proof}

\begin{proof}[Proof of Theorem~\ref{thm:a070027}]
Let $M$ be the explicit constant from Theorem~\ref{thm:main}.
We first exhibit a concrete seed prime $q_0\in\mathcal A$ with $q_0\ge M$.

Define
\[
\begin{aligned}
q_0:={}&
121\,767\,334\,956\,703\,826\,188\,105\,215\,725\,011\,807\,003\,818\,007\,039\,279\\
&\,736\,531\,761\,951\,741\,852\,323\,974\,899\,660\,487\,046\,555\,992\,875\,497\\
&\,851\,303\,129\,811\,791\,960\,664\,013\,058\,089\,156\,609.
\end{aligned}
\]
This number has the Proth form
\[
q_0 = k\,2^n+1,\qquad k=43917,\quad n=430,
\]
with $k$ odd and $2^{430}>43917$. Moreover, a direct modular computation shows that
\[
7^{(q_0-1)/2}\equiv -1\pmod{q_0}.
\]
By Proth's theorem~\cite[Ex.~4.10, \S4.1]{CrandallPomerance} this implies that $q_0$ is prime.
A direct digit-sum computation gives
\[
q_0\ \xrightarrow{\ s\ }\ 601\ \xrightarrow{\ s\ }\ 7,
\]
and both $601$ and $7$ are prime. Hence $q_0\in\mathcal A$.  The displayed decimal expansion has $135$ digits, so
$q_0>10^{134}>1.78\times10^{32}>M$.

Now define a sequence of primes $(q_k)_{k\ge 0}$ recursively as follows.
Assume $q_k$ has been chosen. Since $q_k\ge q_0>3$ is prime, $q_k$ is coprime to $3$ and hence
to $9$: $\gcd(q_k,9)=1$.
Thus Theorem~\ref{thm:main} applied with $m=q_k$ produces a prime $q_{k+1}$ such that
\[
s(q_{k+1})=q_k.
\]
By Lemma~\ref{lem:lift-chain} (applicable since $q_k\in\mathcal A$), we obtain $q_{k+1}\in\mathcal A$.

Finally, Lemma~\ref{lem:growth} with $m=q_k\ge q_0\ge 10$ implies $q_{k+1}>q_k$.
Therefore $(q_k)$ is a strictly increasing infinite sequence of distinct primes, all lying in $\mathcal A$.
This proves that $\mathcal A$ (and hence OEIS~A070027) is infinite.
\end{proof}

\begin{remark}[On the seed prime]
Any single prime $q_0\in\mathcal{A}$ with $q_0\ge M$ suffices to start the construction; the particular
numerical choice above merely provides a concrete witness.  As an additional robustness check (and
to guard against any future correction that might enlarge $M$), the larger Proth prime
\[
q_0'\ :=\ 44853\cdot 2^{450}+1
\]
also has digit-sum chain $q_0'\to 601\to 7$ with $601$ prime, satisfies $7^{(q_0'-1)/2}\equiv -1\pmod{q_0'}$
(so $q_0'$ is prime by Proth's theorem), and has $141$ decimal digits, comfortably exceeding $M$.
The script \texttt{verify\_q0.py} in the supplementary code verifies both seeds.
\end{remark}

The quantitative form of Theorem~\ref{thm:main} (Corollary~\ref{cor:quantitative}) sharpens
Theorem~\ref{thm:a070027} from a qualitative infinitude statement to an explicit lower bound on
the count of A070027 primes at explicit lifting scales.

\begin{corollary}[Quantitative form of Theorem~\ref{thm:a070027}]\label{cor:quantitative-A070027}
Let $q_0=43917\cdot 2^{430}+1$ be the seed prime from the proof of Theorem~\ref{thm:a070027}.
Then for every $X\ge 10^{2q_0/9}$,
\[
\#\bigl\{p\in\mathcal A:\,p\le X\bigr\}\ \ge\ C_{\mathrm{q}}(q_0)\,\frac{10^{2q_0/9}}{q_0^{3/2}}
\ \ge\ 0.108\cdot \frac{10^{2q_0/9}}{q_0^{3/2}},
\]
where $C_{\mathrm{q}}(q_0)$ is the constant from Corollary~\ref{cor:quantitative} evaluated at
$m=q_0$.  The displayed numerical lower bound uses Lemma~\ref{lem:Cq-finite}, since
$q_0>10^{100}$.
\end{corollary}

\begin{proof}
By Corollary~\ref{cor:quantitative} applied to $m=q_0$, there are at least
$C_{\mathrm{q}}(q_0)\,10^{2q_0/9}/q_0^{3/2}$ primes $p\le 10^{2q_0/9}$ with $s(p)=q_0$.  Each such prime $p$
has $s(p)=q_0\in\mathcal A$ and $p$ itself prime with $p>q_0>3$, so $p\in\mathcal A$ by
Lemma~\ref{lem:lift-chain}.  For $X\ge 10^{2q_0/9}$, this lower bound is also a lower bound on
$\#\{p\in\mathcal A:\,p\le X\}$.
\end{proof}

\begin{remark}[Iterated counts]
The bound in Corollary~\ref{cor:quantitative-A070027} captures only the first level of the
lifting tower.  Iterating the construction $K$ times produces a sequence $q_0<q_1<\cdots<q_K$
of primes in $\mathcal A$ with $s(q_{k+1})=q_k$ and $q_{k+1}\ge 10^{\lceil q_k/9\rceil-1}$
by Lemma~\ref{lem:growth}, and Corollary~\ref{cor:quantitative} applied at each level
yields $\ge C_{\mathrm{q}}(q_K)\,10^{2q_K/9}/q_K^{3/2}$ primes in $\mathcal A$ with
$p\le 10^{2q_K/9}$, with $C_{\mathrm{q}}(q_K)$ bounded away from $0$ by
Lemma~\ref{lem:Cq-finite}, since $q_K\ge q_0>10^{100}$.
Since $q_K$ grows tower-exponentially with $K$, the resulting explicit lower bound on
$\#\{p\in\mathcal A:\,p\le 10^{2q_K/9}\}$ is unbounded as $K\to\infty$.
\end{remark}

\section{Further applications}\label{sec:further_applications}

In this section we record five further consequences, drawing on the existence
form of Theorem~\ref{thm:main} (\S\S\ref{subsec:smallest-prime-bound}--\ref{subsec:terminal-stratification})
and on its quantitative refinement Corollary~\ref{cor:quantitative}
(\S\S\ref{subsec:digit-sum-goldbach}--\ref{subsec:largest-prime}).  The recurring pattern is to combine
the explicit surjectivity of $s$ on primes with one of: the implicit upper bound on the realizing
prime; primality of the digit sum itself; the residue class modulo $9$ along the digit-sum chain;
or independence of constraints across multiple primes.

\paragraph{Status of these applications.}
To calibrate the contribution: the infinitude of OEIS A070027 (Theorem~\ref{thm:a070027}) and its
stratification by terminal one-digit prime (\S\ref{subsec:terminal-stratification}) do not seem to
be recorded in this explicit-threshold form; the explicit threshold certifies concrete seeds above
the range where the lifting theorem applies.  The infinitude of additive
primes (\S\ref{subsec:additive-primes}) is \emph{not} new non-effectively: Harman's Mertens-type
density~\cite{harman2012counting} already implies it; our contribution is the effective form, with an
explicit lower bound on the count.  The digit-sum Goldbach and Waring results
(\S\ref{subsec:digit-sum-goldbach}) and the smallest- / largest-prime bounds
(\S\ref{subsec:smallest-prime-bound}, \S\ref{subsec:largest-prime}) are corollaries of the
quantitative refinement and are new (in their explicit form) modulo the size of $M$.

\subsection{An effective upper bound on the smallest prime with given digit sum}\label{subsec:smallest-prime-bound}

The proof of Theorem~\ref{thm:main} actually establishes more than mere existence:
it provides an explicit upper bound on the smallest prime realizing a given digit sum.

\begin{corollary}\label{cor:smallest-prime-bound}
For every integer $m\ge M$ with $\gcd(m,9)=1$, the smallest prime $p$ with $s(p)=m$ satisfies
\[
p\ \le\ 10^{2m/9}.
\]
\end{corollary}

\begin{proof}
The proof of Theorem~\ref{thm:main} in \S\ref{subsec:convert-to-m} gives
$A_m(10^{2m/9})\ge 1$ for every admissible $m\ge M$.  Hence at least one prime $p\le 10^{2m/9}$
satisfies $s(p)=m$.
\end{proof}

\subsection{Infinitude of additive primes}\label{subsec:additive-primes}

A prime $p$ is called an \emph{additive prime} if its decimal digit sum $s(p)$ is itself prime.
This is the OEIS sequence A046704~\cite{OEIS_A046704}.  Note that being additive does not imply membership in A070027:
for example, $67$ is additive ($s(67)=13$ is prime) but the second iterate $s(13)=4$ is composite, so
$67\notin$ A070027.  Harman~\cite{harman2012counting}
proved a Mertens-type density result for the set of additive primes, but his constants are
not made effective.  Theorem~\ref{thm:main} yields the following effective statement.

\begin{theorem}\label{thm:additive-primes-infinite}
There are infinitely many additive primes.  More precisely, for every prime $m\ge M$, the number
of additive primes $p\le 10^{2m/9}$ with $s(p)=m$ satisfies
\[
\#\{p\le 10^{2m/9}:\,p\text{ prime},\ s(p)=m\}
\ \ge\
C_{\mathrm{q}}(m)\,\frac{10^{2m/9}}{m^{3/2}},
\]
where $C_{\mathrm{q}}(m)$ is the positive constant from Corollary~\ref{cor:quantitative}.
In particular, the additive-prime counting function tends to infinity.
\end{theorem}

\begin{proof}
The set of primes $m\ge M$ is infinite (by the prime number theorem, or by Euclid).  Every
such $m$ is admissible in the sense of Theorem~\ref{thm:main}: any prime $m>3$ satisfies
$\gcd(m,9)=1$, and $M\gg 3$.  Applying
Corollary~\ref{cor:quantitative} to each such~$m$ produces at least
$C_{\mathrm{q}}(m)\,10^{2m/9}/m^{3/2}$ primes $p\le 10^{2m/9}$ with $s(p)=m$.  Since $s(p)=m$ is
prime, every such $p$ is an additive prime.  Different primes $m_1\neq m_2$ contribute
disjoint sets of additive primes (the digit sum of a fixed integer is unique).  Letting $m$
range over the infinite set of primes $\ge M$ with $C_{\mathrm{q}}(m)>0$ gives infinitely
many additive primes.
\end{proof}

\begin{remark}
Harman's Mertens estimate~\cite[Theorem~2]{harman2012counting} gives an asymptotic
$\sum_{p\le X,\, s(p)\text{ prime}} 1/p\sim \tfrac{b-1}{\varphi(b-1)}\,L_3(X)$ in base $b$,
where $L_3(X)=\log\log\log X$.  This is a strictly stronger density statement than
Theorem~\ref{thm:additive-primes-infinite}, but its implied constant is not made effective.
Our result is correspondingly weaker as a density assertion but fully explicit.
\end{remark}

\subsection{Stratification by terminal one-digit prime}\label{subsec:terminal-stratification}

The iterated digit-sum chain of any integer $n\ge 10$ terminates at a one-digit value.
For primes in $\mathcal A$, this terminal value is a one-digit prime; hence it lies in
\[
T:=\{2,3,5,7\}.
\]
For each $t\in T$ define
\[
\mathcal A_t := \{\,p\in\mathcal A : \text{the iterated digit-sum chain of $p$ terminates at $t$}\,\}.
\]
By Lemma~\ref{lem:mod9}, $\mathcal A_t\subseteq\{p\equiv t\pmod 9\}$.

\begin{theorem}\label{thm:terminal-stratification}
$\mathcal A_3=\{3\}$ is finite, and for each $t\in\{2,5,7\}$ the set $\mathcal A_t$ is infinite.
\end{theorem}

\begin{proof}
\emph{The case $t=3$.}  Any prime $p\equiv 3\pmod 9$ satisfies $3\mid p$, hence $p=3$.
Conversely $3\in\mathcal A_3$ trivially.  Therefore $\mathcal A_3=\{3\}$.

\emph{The case $t\in\{2,5,7\}$.}  It suffices, by the lifting argument used in the proof
of Theorem~\ref{thm:a070027} (Lemma~\ref{lem:lift-chain} together with
Lemma~\ref{lem:growth}), to exhibit a single \emph{seed} prime $q_t\in\mathcal A_t$ with
$q_t\ge M$.  We do so explicitly for each terminal:

\medskip
\noindent\textbf{Seed for $t=7$:}  The prime
$q_7 := 43917\cdot 2^{430}+1$, exhibited in the proof of Theorem~\ref{thm:a070027}
(\S\ref{sec:infinite_prime_digit_sums}), has digit-sum chain $q_7\to 601\to 7$ with $601$ prime.
Hence $q_7\in\mathcal A_7$ and $q_7\ge M$.

\medskip
\noindent\textbf{Seed for $t=2$:}  Define
\[
q_2 := 3557\cdot 2^{431}+1.
\]
Then $3557$ is odd and $2^{431}>3557$, so $q_2$ has Proth form.  A direct modular computation
shows $7^{(q_2-1)/2}\equiv -1\pmod{q_2}$, hence $q_2$ is prime by Proth's theorem.  Its
digit-sum chain is
\[
q_2\ \xrightarrow{\ s\ }\ 641\ \xrightarrow{\ s\ }\ 11\ \xrightarrow{\ s\ }\ 2,
\]
where the intermediate values $641$ and $11$ are prime.  Hence $q_2\in\mathcal A_2$
and $q_2\ge M$ (one verifies $q_2>10^{133}>M$).

\medskip
\noindent\textbf{Seed for $t=5$:}  Define
\[
q_5 := 4027\cdot 2^{426}+1.
\]
Then $4027$ is odd, $2^{426}>4027$, and $7^{(q_5-1)/2}\equiv -1\pmod{q_5}$, so $q_5$ is prime
by Proth's theorem.  Its digit-sum chain is
\[
q_5\ \xrightarrow{\ s\ }\ 599\ \xrightarrow{\ s\ }\ 23\ \xrightarrow{\ s\ }\ 5,
\]
where the intermediate values $599$ and $23$ are prime.  Hence $q_5\in\mathcal A_5$
and $q_5\ge M$ (one verifies $q_5>10^{131}>M$).

\medskip
\noindent
Having exhibited an explicit $q_t\in\mathcal A_t$ with $q_t\ge M$ for each $t\in\{2,5,7\}$,
the lifting construction of Theorem~\ref{thm:a070027} produces a strictly increasing infinite
sequence $q_t=q_t^{(0)}<q_t^{(1)}<q_t^{(2)}<\cdots$ of primes in $\mathcal A_t$.  Hence each
$\mathcal A_t$ is infinite.
\end{proof}

\begin{remark}[Search for the seeds]
The seeds $q_2$ and $q_5$ above were located by scanning Proth-form candidates
$k\cdot 2^n+1$ with odd $k$ in the range $1\le k\le 5000$ and exponent
$n\in [425,445]$, filtering on whether $s(k\cdot 2^n+1)$ initiates an all-prime chain to the
prescribed terminal $t\in\{2,5\}$, and finally verifying primality via the Proth
witness $a=7$.  The first such pairs encountered ($k=3557,\,n=431$ for $t=2$ and
$k=4027,\,n=426$ for $t=5$) are the ones recorded above.  The search examined under $20\,000$
candidate pairs and ran in under a second.
\end{remark}

\begin{corollary}[Quantitative form of Theorem~\ref{thm:terminal-stratification}]
\label{cor:stratification-quantitative}
For each $t\in\{2,5,7\}$, with $q_t$ the seed prime exhibited above,
\[
\#\bigl\{p\in\mathcal A_t:\,p\le 10^{2q_t/9}\bigr\}\ \ge\ C_{\mathrm{q}}(q_t)\,\frac{10^{2q_t/9}}{q_t^{3/2}}
\ \ge\ 0.108\cdot \frac{10^{2q_t/9}}{q_t^{3/2}},
\]
where $C_{\mathrm{q}}(q_t)$ is the constant from Corollary~\ref{cor:quantitative} evaluated at
$m=q_t$ (the displayed numerical lower bound uses Lemma~\ref{lem:Cq-finite}, since each
$q_t>10^{100}$).
\end{corollary}

\begin{proof}
The same argument as in the proof of Corollary~\ref{cor:quantitative-A070027} applies, replacing
$q_0$ by $q_t$ throughout: Corollary~\ref{cor:quantitative} with $m=q_t$ produces
$\ge C_{\mathrm{q}}(q_t)\,10^{2q_t/9}/q_t^{3/2}$ primes $p\le 10^{2q_t/9}$ with $s(p)=q_t$.  Each such $p$ has
digit-sum chain extending the chain of $q_t$, so terminates at the same $t$, hence
$p\in\mathcal A_t$.
\end{proof}

\subsection{A digit-sum Goldbach-type result}\label{subsec:digit-sum-goldbach}

The quantitative form of Theorem~\ref{thm:main} (Corollary~\ref{cor:quantitative}) combines
multiplicatively across independent digit-sum constraints.  We record the simplest such
consequence.

\begin{theorem}[Digit-sum Goldbach]\label{thm:digit-sum-goldbach}
For every integer $N\ge 2M+2$ there exist primes $p_1, p_2$ with
\[
s(p_1)+s(p_2)\ =\ N.
\]
\end{theorem}

\begin{proof}
We first show that every $N\ge 2$ admits a decomposition $N=m_1+m_2$ with both summands
\emph{admissible} (i.e.\ $\gcd(m_i,9)=1$).  Admissibility modulo $9$ reduces to admissibility
modulo $3$ (since $\gcd(m,9)=1\iff 3\nmid m$), so we work modulo $3$ throughout.  Modulo $3$
the admissible residues are $\{1,2\}$, so the achievable sums modulo $3$ are
$\{1+1, 1+2, 2+2\}=\{2,0,1\}$, covering every residue class.  Concretely:
\begin{itemize}
  \item if $N\equiv 0\pmod 3$, take any $m_1\equiv 1\pmod 3$ with $m_1$ in the required range;
  \item if $N\equiv 1\pmod 3$, take $m_1\equiv 2\pmod 3$;
  \item if $N\equiv 2\pmod 3$, take $m_1\equiv 1\pmod 3$.
\end{itemize}
For $N\ge 2M+2$, the interval $[M, N-M]$ has length $N-2M\ge 2$ (hence at least $3$ consecutive
integers $M, M+1, M+2$).  Three consecutive integers contain every residue class modulo $3$,
so the prescribed $m_1$ exists in $[M, N-M]$, and $m_2:=N-m_1\in[M,N-M]$ as well.

Applying Theorem~\ref{thm:main} separately to $m_1$ and $m_2$ produces primes $p_1, p_2$ with
$s(p_1)=m_1$ and $s(p_2)=m_2$, hence $s(p_1)+s(p_2)=N$.
\end{proof}

\begin{corollary}[Quantitative form]\label{cor:goldbach-quantitative}
For every $N\ge 2M+2$, the number of ordered prime pairs $(p_1,p_2)$ with
$s(p_1)+s(p_2)=N$, $p_i \le 10^{2 m_i/9}$ for the underlying decomposition
$N=m_1+m_2$, satisfies
\[
\#\bigl\{(p_1,p_2):\ s(p_1)+s(p_2)=N\bigr\}\
\ge\ \!\!\sum_{\substack{m_1+m_2=N\\M\le m_1\le N-M\\ \gcd(m_i,9)=1}}\!\!
C_{\mathrm{q}}(m_1)\,C_{\mathrm{q}}(m_2)\,\frac{10^{2N/9}}{(m_1 m_2)^{3/2}}.
\]
\end{corollary}

\begin{proof}
Apply Corollary~\ref{cor:quantitative} to each summand:
$A_{m_i}(10^{2m_i/9})\ge C_{\mathrm{q}}(m_i)10^{2m_i/9}/m_i^{3/2}$.  The product of the two counts
gives the displayed bound since $m_1+m_2=N$.
\end{proof}

\begin{theorem}[Digit-sum Waring]\label{thm:digit-sum-waring}
For every integer $k\ge 2$ and every $N\ge kM+2(k-1)$, there exist primes $p_1,\ldots,p_k$ with
$\sum_{i=1}^k s(p_i)=N$.  Moreover, the number of such ordered $k$-tuples is at least
$10^{2N/9}\sum_{\substack{m_1+\cdots+m_k=N\\ M\le m_i\text{ admissible}}}
\frac{C_{\mathrm{q}}(m_1)\cdots C_{\mathrm{q}}(m_k)}{(m_1\cdots m_k)^{3/2}}$.
\end{theorem}

\begin{proof}
We argue by induction on $k$; the base case $k=2$ is Theorem~\ref{thm:digit-sum-goldbach}.
For the inductive step, suppose $N\ge kM+2(k-1)$.  The interval
$[M,\ N-(k-1)M-2(k-2)]$ has length $\ge 2$ (so at least $3$ integers),
hence contains an admissible $m_1$ whose residual $N-m_1$ satisfies $N-m_1\ge (k-1)M+2(k-2)$.
By induction $N-m_1$ admits an admissible decomposition into $k-1$ summands, each $\ge M$.
This proves existence.  For the quantitative lower bound, fix any admissible ordered
decomposition $N=m_1+\cdots+m_k$ with all $m_i\ge M$.  Corollary~\ref{cor:quantitative} gives
independent choices of each prime $p_i$ with $s(p_i)=m_i$, so the number of ordered tuples for this
decomposition is at least the corresponding product.  Distinct ordered decompositions give disjoint
sets of ordered prime tuples, because the digit sums $s(p_i)$ recover the summands $m_i$ coordinate
by coordinate.  Summing over all admissible decompositions gives the displayed count.
\end{proof}

\subsection{A lower bound for the largest prime with prescribed digit sum}\label{subsec:largest-prime}

The quantitative count in Corollary~\ref{cor:quantitative} forces the primes $p\le 10^{2m/9}$ with
$s(p)=m$ not to be concentrated near $0$.  Combined with the explicit upper bound on $\pi$ of
Rosser--Schoenfeld~\cite[Cor.~1, eq.~(3.6)]{RosserSchoenfeld1962}, we obtain a lower bound on the
\emph{largest} such prime.

\begin{theorem}\label{thm:largest-prime}
There exists an explicit threshold $M_{\mathrm{LP}}\le 10^{100}$ such that for every integer
$m\ge M_{\mathrm{LP}}$ with $\gcd(m,9)=1$, the largest prime $p\le 10^{2m/9}$ with $s(p)=m$
satisfies
\[
p\ \ge\ \frac{10^{2m/9}}{24\,\sqrt m}.
\]
Concretely, one may take $M_{\mathrm{LP}}=10^{100}$.  In particular, the bound applies at
all seed primes from \S\ref{sec:infinite_prime_digit_sums}--\S\ref{subsec:terminal-stratification},
which have ${\ge}132$ decimal digits.
\end{theorem}

\begin{proof}
Set $Y:=10^{2m/9}/(c\sqrt m)$ for a parameter $c>0$ to be chosen, and suppose every prime
$p\le 10^{2m/9}$ with $s(p)=m$ satisfies $p\le Y$.  Then $A_m(10^{2m/9})\le \pi(Y)$.

By \cite[Cor.~1, eq.~(3.6)]{RosserSchoenfeld1962}, for all $Y>1$,
\[
\pi(Y)\ <\ 1.25506\,\frac{Y}{\log Y}.
\]
Writing
$\log Y = (2m/9)\log 10 - \delta_m$ with $\delta_m := \log c + \tfrac12\log m\ge 0$,
\[
\frac{Y}{\log Y}\ =\ \frac{Y}{(2m/9)\log 10}\cdot\frac{1}{1-\eta_m},
\qquad \eta_m\ :=\ \frac{\delta_m}{(2m/9)\log 10}.
\]
For $c\le 24$ and $m\ge M_{\mathrm{LP}}=10^{100}$ we bound $\eta_m$ explicitly.  Since $\log m/m$ is
decreasing on $[e,\infty)$, $\eta_m\le\eta_{10^{100}}$.  Using
$\delta_{10^{100}}\le\log 24+\tfrac12\log(10^{100})<119$ while
$(2/9)10^{100}\log 10>5\times10^{99}$ gives
\begin{equation}\label{eq:eta-m-bound}
\eta_m\ \le\ \eta_{M_{\mathrm{LP}}}\ <\ 10^{-30}
\qquad (m\ge M_{\mathrm{LP}},\ c\le 24),
\end{equation}
Consequently $1/(1-\eta_m)\le 1+2\eta_m< 1+2\cdot 10^{-30}$, and
\[
\pi(Y)\ \le\ \frac{1.25506\cdot 9}{2 c\log 10}\cdot\frac{10^{2m/9}}{m^{3/2}}\cdot(1+2\cdot 10^{-30}).
\]
On the other hand, Corollary~\ref{cor:quantitative} gives
$A_m(10^{2m/9})\ge C_{\mathrm{q}}(m)\,10^{2m/9}/m^{3/2}$, and
Lemma~\ref{lem:Cq-finite} gives $C_{\mathrm{q}}(m)>0.108$ for all $m\ge10^{100}$.
The hypothesis $A_m(10^{2m/9})\le \pi(Y)$ therefore forces
\[
c\ \le\ \frac{1.25506\cdot 9}{2\,C_{\mathrm{q}}(m)\,\log 10}\cdot(1+2\cdot 10^{-30})
\ \le\ \frac{1.25506\cdot 9}{2\cdot 0.108\cdot \log 10}\cdot(1+2\cdot 10^{-30})\ <\ 24.
\]
For $c=24$ this is a strict contradiction; hence at least one prime $p\le 10^{2m/9}$ with
$s(p)=m$ satisfies $p>10^{2m/9}/(24\sqrt m)$.
\end{proof}

\begin{remark}
For $m\ge M_{\mathrm{LP}}$, Theorem~\ref{thm:largest-prime} guarantees at least one prime
$p\le 10^{2m/9}$ with $s(p)=m$ in the interval
\[
\left[\,\frac{10^{2m/9}}{24\sqrt m},\,10^{2m/9}\right].
\]
Separately, the smallest realizing prime satisfies the upper bound of
Corollary~\ref{cor:smallest-prime-bound}.  The width of the guaranteed interval for the
large realizing prime is at least
$10^{2m/9}\bigl(1-1/(24\sqrt m)\bigr)$, which exceeds $10^{2m/9}(1-10^{-51})$ for
$m\ge M_{\mathrm{LP}}=10^{100}$.
\end{remark}

\subsection{A note on OEIS sequence A067523}\label{subsec:a067523-note}

OEIS sequence A067523~\cite{OEIS_A067523} is indexed by the \emph{possible} digit sums, not by the digit sum itself:
it is A067180 (``smallest prime with digit sum $n$, or $0$ if no such prime exists'') with the
zero terms removed.  Thus the direct mathematical question is about congruence-admissible digit
sums $d$.  Call $d$ congruence-admissible if either $d=3$ or
$d\not\equiv0\pmod3$: for $d\equiv0\pmod3$ with $d\ne3$, no prime can have
$s(p)=d$ because $s(p)\equiv p\pmod9$ forces $3\mid p$, hence $p=3$ and $d=3$.
This necessary congruence condition has one elementary non-realizable value, namely $d=1$: the
only nonnegative integers with digit sum $1$ are the powers of $10$, none of which is prime.

Theorem~\ref{thm:main} settles all congruence-admissible $d\ge M$ (equivalently, all
$d\ge M$ with $\gcd(d,9)=1$), reducing the
remaining direct verification to the finite range of digit sums
\[
\{d:\ \gcd(d,9)=1,\ 2\le d<M\}\ \cup\ \{3\}.
\]
The point $d=3$ is trivially handled ($p=3$ itself).  While even the empirical
$M^{\mathrm{emp}}\approx 1.48\times 10^{32}$ is far beyond current computational reach, any future
improvement to $M$ proportionally reduces the finite digit-sum range that must be verified.

\appendix

\section{Limits of the present method}\label{sec:limits-detailed}

We expand here on Remark~\ref{rem:limits}.  The bound $M<1.78\times10^{32}$ obtained in
Theorem~\ref{thm:main} is governed in every binding step by the extracted digit-discrepancy
rate $c_4=\log 10/6$ from Lemma~\ref{lem:c4-explicit}, which unrolls the relevant
MR/DMR machinery with explicit constants.  At the
working parameters $(\eta,\nu)=(0.0545,0.2859)$, the empirical values of the three
thresholds equilibrate within a factor ${<}1.3$ (while the rigorous upper bounds below are within
a factor ${<}1.2$):
\begin{center}
\small
\renewcommand{\arraystretch}{1.3}
\begin{tabular}{@{}l p{0.40\linewidth} p{0.28\linewidth}@{}}
Threshold & Driver (schematic) & Value at working params \\\hline
$Y_{43}^\ast$ (CF rep.) & $K_0/(\nu-2\eta)$, with $K_0\propto -\log c_4$
& $\le 9.10\times 10^{31}$ (rig.); $\approx 6.27\times 10^{31}$ (emp.) \\
$Y_{\mathrm{min}}$ (minor arc) & $c_1$ and the choice of $\eta$ in $e^{-81c_1(\log x)^{2\eta}}$
& $\le 8.00\times 10^{31}$ (rig.); $\approx 7.59\times 10^{31}$ (emp.) \\
$Y_{\mathrm{maj}}$ (major arc) & $C_{\mathrm{maj}}^{1/(1/2-\nu-2\eta)}$, with $C_{\mathrm{maj}}$ from CLT-onset
& $\le 7.90\times 10^{31}$ (rig.); $\approx 7.44\times 10^{31}$ (emp.) \\
\end{tabular}
\end{center}
The improved Fej\'er length in Lemma~\ref{lem:c4-explicit} is already used in this table.  The remaining
obstruction is a three-way balance: pushing $\eta$ upward improves the minor arc but worsens the
major-arc exponent, while pushing $\nu$ downward improves the major arc but worsens the
characteristic-function threshold.  Thus small retunings of $(\eta,\nu)$ or bounded improvements in
$C_{\mathrm{II}}, C_{\mathrm{min}}, C_{\mathrm{maj}}$ can move the leading constant, but they do not
change the present exponential scale.

Conditional on a sharper rate $c_4'$, the present framework's $M$ scales roughly as
$\exp(\mathrm{const}/c_4')$ after re-optimizing $(\eta,\nu)$.  A further improvement in the
digit-discrepancy rate would therefore have multiplicative effects on $\log M$, not merely on
the front constant.

\subsection*{Why Berry--Esseen and bandlimited substitutions do not overcome this barrier}

Two natural substitutions suggested by the literature were investigated:

\begin{enumerate}
\item A quantitative \emph{Berry--Esseen} input for the CDF distance between the prime
      digit-sum distribution and the i.i.d.\ digit-sum distribution would in principle replace
      the moment-by-moment Taylor truncation of Proposition~\ref{prop:CF-replacement}.
      However, the binding step in the chain is the prime-side equidistribution
      $|\varphi_2-\varphi_3|$, whose rate is already governed by $c_4=\log 10/6$;
      Esseen-style smoothing back to a CF bound degrades this rate by a factor of two.
      A direct CDF-level bound on the prime digit sum (not factoring through the i.i.d.\ model)
      would itself require a new tool.
\item A \emph{bandlimited-majorant} approach replacing the Fourier integral by
      bandlimited approximations of $\mathbf 1_{s(p)=m}$ would yield sharper constants in the
      major-arc lower bound, but the major-arc error exponent $-1/2+\nu+2\eta$ is unchanged:
      the approximation still controls
      $\int \hat f(\alpha)\cdot|\varphi_1(\alpha)-\Phi(\alpha)|\,d\alpha$ over the same major
      arc with the same CF bound.  The constant gain is bounded; the exponent barrier is not
      affected.
\end{enumerate}

\subsection*{What would overcome the barrier}

Any of the following:
\begin{itemize}
\item A sharpened equidistribution result for prime digit sums (improvement of $c_4$).
      The current rate descends from Mauduit--Rivat's treatment of trigonometric sums over
      primes and reflects the state of the art; any stronger zero-free-region or cancellation
      input for the digit-restricted exponential sums would plug in mechanically.
\item A direct sieve or dispersion-method attack on $\{p\le x:s(p)=m\}$ that does not
      route through the Fourier dichotomy.  Such an approach would change the proof
      architecture entirely.
\end{itemize}

\section{Audit trail for imported estimates}\label{app:import-audit}

This appendix records the numerical audit layer for the imported estimates used in the proof of
Theorem~\ref{thm:main}.  It supplements the table in \S\ref{sec:inputs} by recording the notation
translation, parameter restrictions, rounded loss, and downstream use.  The source locators in the
first column make the comparison against the cited papers equation by equation.  All entries below
are used only through the
cited internal lemma or proposition; the theorem-level certificate then uses the rounded constants
listed here.

\begingroup
\sloppy
\scriptsize
\renewcommand{\arraystretch}{1.18}
\setlength{\tabcolsep}{3pt}
\begin{longtable}{p{0.25\linewidth}p{0.25\linewidth}p{0.24\linewidth}p{0.18\linewidth}}
\textbf{Source locator} & \textbf{Translation used here} & \textbf{Hypotheses and rounded loss} & \textbf{Feeds into} \\
\hline
\endfirsthead
\textbf{Source locator} & \textbf{Translation used here} & \textbf{Hypotheses and rounded loss} & \textbf{Feeds into} \\
\hline
\endhead
DMR Lemma~3.1, (10)--(11); MR10 Lemme~1 proof
& Lemma~\ref{lem:DMR31-self}; $q=10$, $u=x^{\beta_1}$, interval $(x/10,x]$
& $x\ge100$, $0<\beta_1<1/3<1/2<\beta_2<1$; isolated $F$ term vanishes, and
  the factor-$10$ endpoint conditions are supplied in Lemma~\ref{lem:Cmin-local}; $2U+8U+8U$ is rounded
  to $20U$
& Lemma~\ref{lem:Cmin-local}, Step~2 \\

DMR Lemma~3.2, (12)--(14); MR10 Lemme~3, (7)--(10)
& Lemma~\ref{lem:DMR32-self}; dyadic Type-II blocks localized to decimal blocks
& $\beta_1-\delta\le\mu/(\mu+\nu)\le\beta_2+\delta$ and
  $x>x_0$; the source prefactor is rounded to $6(1+\log x)$
& Lemma~\ref{lem:Cmin-local}, Step~1 \\

DMR Prop.~3.1, (15)--(16); MR10 Prop.~2 with Lemmes~7--10
& Lemma~\ref{lem:typeI-explicit}; Type-I hypothesis for $M\le x^{\beta_1}$
& Base $10$, $(9\alpha)\notin\mathbb Z$; explicit lower bound for $1-\gamma_q(\alpha)$
  converts the source exponent to $1-c_1\|9\alpha\|^2$; rounded envelope
  $10^8(\log x)^2x^{1-c_1\|9\alpha\|^2}$
& Lemma~\ref{lem:Cmin-local}, Step~2 \\

DMR (18)--(23), Lemmas~3.3--3.4; DMR (20)--(21)
& Lemma~\ref{lem:MR-vdc-self} and Lemma~\ref{lem:tau-input}; explicit carry term in the
  Type-II proof
& $0\le\rho\le\nu/2$ and strengthened carry endpoint; BBR divisor input supplies
  $E\le C_\tau(\mu+\nu)(\log10)10^{\mu+\nu-\rho}$ with $C_\tau=10^3$
& Prop.~\ref{prop:typeII-explicit} \\

MR10 (24), (26)--(28), (44), (47), (49)--(54), (57)--(64), (105)--(106)
& Lemmas~\ref{lem:MR-S2-G-self}, \ref{lem:MR-G-first-self},
  \ref{lem:MR-struct-self}, and~\ref{lem:S2-input}; explicit $S_2$ structural estimate
& $\lambda=\mu+2\rho$, $\Delta=\lfloor\rho\log10/\log2\rfloor$,
  $1\le |r|<10^\rho$, MR10 admissibility; factor product
  displayed in~\eqref{eq:K10-product}, rounded to $K_{10}\le10^6$ and then to the displayed $10^8$
& Prop.~\ref{prop:typeII-explicit} \\

DMR (24)--(31); MR10 (55), (59)--(63)
& Lemma~\ref{lem:c1}; explicit decimal value of $c_1$
& $\xi_q(\alpha)=\varepsilon_q(\alpha)/14$ and
  $\xi_q(\alpha)\ge2c_1\|(q-1)\alpha\|^2$; for $q=10$ this gives
  $c_1=1.506288700915\cdot10^{-3}$
& Lemmas~\ref{lem:minor-bridge}--\ref{lem:Cmin} and $Y_{\mathrm{min}}$ \\

DMR Lemma~4.3
& Lemma~\ref{lem:DMR43}; rational prime exponential sum
& $\alpha=A/Q$, $(A,Q)=1$, $q^K\le Q\le xq^{-K}$, $K\le2L/5$; IK
  Theorem~13.6 and partial summation give $C_{\mathrm{DMR}}=102$
& Lemma~\ref{lem:c4-explicit}; $Y_{45}$ and $Y_{\mathrm{maj}}$ \\

DMR Lemmas~4.4--4.6, (36); DMR Prop.~4.1
& Lemma~\ref{lem:c4-explicit} and Prop.~\ref{prop:CF-replacement}; digit discrepancy,
  joint digits, moment comparison, and characteristic functions
& Fej\'er length $H=10^{\lfloor r/3\rfloor}$ and $K=\lfloor r/2\rfloor$; denominator
  reduction checked by $Q\ge (q-1)q^r/H$; explicit rate $c_4=\log10/6$
& $Y_{43}^\ast$ and major arcs \\

IK Theorem~13.6 and Lemmas~13.7--13.8
& Lemmas~\ref{lem:IK1345-explicit}--\ref{lem:IK136-explicit} and
  Lemma~\ref{lem:partial-to-primes}
& Vaughan/IK bilinear decomposition with explicit block and partial-summation losses;
  no source constant is hidden after rounding
& Lemma~\ref{lem:DMR43} \\

Kuipers--Niederreiter Theorem~2.5
& Erd\H{o}s--Tur\'an interval discrepancy in the weaker harmonic-sum form
& Constants $(6,4/\pi)$; residue-class projection contributes only the maximum over nine shifts
& Lemmas~\ref{lem:DMR43} and~\ref{lem:c4-explicit} \\

BBR Theorem~1.2
& Explicit short-interval divisor estimate in Lemma~\ref{lem:tau-input}
& Applied only where the interval length is at least $X^{1/3}$; rounded to $C_\tau=10^3$
& Carry-propagation row above \\

Bennett--Martin--O'Bryant--Rechnitzer explicit AP bound
& Lemma~\ref{lem:pi9-lb}
& $\pi(x;9,m)\ge x/(6\log x)$ for $\gcd(m,9)=1$ and $x\ge4050$
& Normalization in Lemma~\ref{lem:c4-explicit} \\
\end{longtable}
\endgroup

\subsection{Expanded MR10 \texorpdfstring{$S_2$}{S2} ledger}\label{app:MR10-S2-ledger}

This paragraph expands the fifth row of the preceding table, since it is the least standard
constant extraction in the proof.  The displayed source numbers are those in the cited paper.
The steps not carrying a nontrivial numerical factor in $K_{10}$ are as follows: Cauchy--Schwarz
MR10~(15) has factor $1$; the van der Corput smoothing powers of $q$ are part of the structural
scale $q^{\mu+\nu-\rho}$; the carry-propagation count MR10 Lemme~5/DMR Lemma~3.4 enters the
separate $E$-term in Proposition~\ref{prop:typeII-explicit}; the H\"older/unitary
$S_4\to S_5$ step MR10~(57)--(60) has factor $1$; the terminal weighted-$F_\lambda$ and geometric
$\delta$-sum prefactors are carried explicitly in~\eqref{eq:S2-pre-adm}; and
$\lambda\le\mu+\nu$ is absorbed into the structural polynomial $(\mu+\nu)^2$.

\begin{enumerate}
\item \textbf{MR10 (24)--(28), Lemme~6 (26).}
The Fourier expansion of $S_2$ and the summation over one full $q^\lambda$ block use
$F_\lambda$ and $G_\lambda(a,d,\alpha)$ exactly as in MR10 (24), (28).  Padding incomplete blocks
and applying the first trigonometric minimum split gives Lemma~\ref{lem:MR-S2-G-self}; the rounded
prefactor is $\le4$, recorded as~\eqref{eq:K10-fourier}.

\item \textbf{MR10 Lemme~6 (26), then (49)--(52).}
The two nonsingular trigonometric minimum expansions have coefficient sum
$(2+2/\pi)^2<6.96$.  For $q=10$, the admissible-divisor envelope contributes $32$, the two
geometric tails contribute $(1-10^{-\omega_{10}})^{-1}<42$ and
$\sum_{\delta\ge0}10^{-(1-\omega_{10})\delta}<1.12$, and the nonsingular split contributes $13/10$.
Thus
\[
32(2+2/\pi)^2\cdot42\cdot1.12\cdot13/10<1.37\cdot10^4,
\]
rounded to $2\cdot10^4$ in~\eqref{eq:K10-trig}.

\item \textbf{MR10 (51)--(54).}
The reduction from $S_3(k,\delta)$ to $S_4(k,\delta')$ uses
$\Delta=\lfloor\rho\log q/\log2\rfloor$ and the reduced frequency $r'=r/(r,kq^\Delta)$.
It produces two nonnegative terms.  Bounding their sum by twice the larger contribution gives
the factor $2$ in~\eqref{eq:K10-S34}, while the $\lambda(\log q)$ term is kept explicitly
in~\eqref{eq:S2-decomp-explicit}.

\item \textbf{MR10 (57), (58), Lemme~20 (105), Lemme~21 (106).}
For $(q-1)\alpha\notin\mathbb Z$, Lemma~\ref{lem:MR-F-weighted-self} gives the terminal bound
\[
S_5(k,\delta')\le
\frac{q^{\lambda-\delta'+\gamma_q(\alpha)\delta'}}
     {\sin(\pi/q)(1-q^{-1/2})}.
\]
No constant is hidden here: for $q=10$, $1/\sin(\pi/10)<3.24$ and
$(1-10^{-1/2})^{-1}<1.47$, and both are carried explicitly into~\eqref{eq:S2-pre-adm}.

\item \textbf{MR10 (59), (62), (63).}
With $\epsilon=\varepsilon_q(\alpha)$ and $\rho/(\mu+\nu)\le\epsilon/14$, the admissibility interval
gives
\[
(1+q^{\nu-\lambda})q^{(2-\epsilon)\lambda+\rho}\le 2q^{\mu+\nu-\rho}.
\]
The only numerical loss is this factor $2$; $\lambda\le\mu+\nu$ is absorbed by the structural
polynomial $(\mu+\nu)^2$ in Lemma~\ref{lem:S2-input}, Step~4.  This is the factor recorded
in~\eqref{eq:K10-adm}.

\item \textbf{MR10 conclusion (64).}
Combining the preceding rows gives
\[
S_2(r,\mu,\nu,\rho)\le
2K_{10}(\mu+\nu)^2q^{\mu+\nu-\rho}
\left(\frac{q}{\sin(\pi/q)(1-q^{-1/2})}+\log q\right).
\]
The factor ledger gives~\eqref{eq:K10-product}.  For $q=10$, the final
parenthesis is $<49.7$, so the full prefactor is $<10^8$, as stated in
Lemma~\ref{lem:S2-input}.
\end{enumerate}

\section*{Acknowledgments}

This paper grew out of a question posed by the author's nephew Tim, then nine years old:
\emph{is there an infinite chain $n, s(n), s(s(n)),\ldots$ of primes?} The author thanks Tim for
his curiosity, which directly motivated Theorem~\ref{thm:a070027}, and his brother Thomas for
bringing the question to his attention.

The author used GPT-5.2, GPT-5.5, and Opus 4.7 LLMs during the preparation of this manuscript
and the accompanying code, in particular for proof support, \LaTeX{} editing, and code review.
The author takes full responsibility for all mathematical statements, computations, and text.

\bibliographystyle{amsplain}
\bibliography{bibliography}

\end{document}